\documentclass[11pt]{amsart}
 
\usepackage{amsmath}
\usepackage{amsthm}  
\usepackage{verbatim}
\usepackage{enumerate} 
\usepackage{mathtools}  
\usepackage{makecell}
\usepackage{amssymb}
\usepackage{mathrsfs}
\usepackage[abs]{overpic}
\usepackage{dsfont}
\usepackage[top=1.5in, bottom=1.5in, left=0.7in, right=0.7in]{geometry}  
\usepackage[colorlinks]{hyperref}
\hypersetup{
	colorlinks,
	citecolor=blue,
	linkcolor=red
}   
\numberwithin{equation}{section}
\usepackage{xypic}
\usepackage{bm}
\usepackage{cleveref}
\usepackage{tikz}
\usepackage{float}
\usetikzlibrary{arrows.meta}
\usetikzlibrary{arrows,positioning,decorations.pathreplacing} 
\tikzset{
    >=stealth',
    punkt/.style={
           rectangle,
           rounded corners,
           draw=black, very thick,
           text width=6.5em,
           minimum height=2em,
           text centered},
    pil/.style={
           ->,
           thick,
           shorten <=2pt,
           shorten >=2pt,}
}

\newtheorem*{theorem*}{\textbf{Theorem}}
\newtheorem{theorem}{\textbf{Theorem}}[section]

\newtheorem{definition}[theorem]{\textbf{Definition}}
\newtheorem{proposition}[theorem]{\textbf{Proposition}}
\newtheorem{lemma}[theorem]{\textbf{Lemma}}

\newtheorem*{claim*}{\textbf{Claim}}

\newtheorem{corollary}[theorem]{\textbf{Corollary}}
\newtheorem{remark}[theorem]{\textbf{Remark}}

\newtheorem{example}[theorem]{\textbf{Example}}

\newtheorem{definition/proposition}[theorem]{\textbf{Definition/Proposition}}

\def\N{{\mathbb N}}
\def\R{\mathbb{R}}
\def\Z{{\mathbb Z}}

\def\C{{\mathbb C}}
\def\D{{\mathbb D}}

\def\Q{{\mathbb Q}}

\newcommand{\CP}{\mathbb{C}\mathbb{P}}

\def\cM{{\mathcal M}}

\def\cO{{\mathcal O}}

\def\cT{{\mathcal T}}

\def\bH{{\bm H}}

\def\bk{{\bm k}}

\def\rd{{\rm d}}

\DeclareMathOperator{\Ima}{im}

\DeclareMathOperator{\Id}{id}

\DeclareMathOperator{\Aug}{Aug}
\DeclareMathOperator{\RSFT}{RSFT}
\DeclareMathOperator{\SFT}{SFT}

\DeclareMathOperator{\Pl}{P}
\DeclareMathOperator{\PT}{APT}

\DeclareMathOperator{\TT}{T}
\DeclareMathOperator{\CHA}{CHA}

\def\cont{{\mathfrak{Con}}}
\def\mc{\mathfrak{mc}}
\def\OB{\mathrm{OB}}
\def\BO{\mathrm{BO}}

\newcommand{\Addresses}{{
		\bigskip
		\footnotesize
		
	     Agustin Moreno, \par\nopagebreak
        \textsc{Institute for Advanced Study, Princeton, U.S. / Heidelberg Universität, Germany}\par\nopagebreak
         \textit{E-mail address:} \href{mailto:agustin.moreno2191@gmail.com}{agustin.moreno2191@gmail.com}
		
		\medskip
		
	     Zhengyi Zhou, \par\nopagebreak
        \textsc{State Key Laboratory of Mathematical Sciences, Chinese Academy of Sciences;}\par\nopagebreak
	    \textsc{Morningside Center of Mathematics, Chinese Academy of Sciences;}\par\nopagebreak
         \textsc{Academy of Mathematics and Systems Science, Chinese Academy of Sciences, China}\par\nopagebreak
		\textit{E-mail address}: \href{mailto:zhengyizhounju@gmail.com}{zhengyizhounju@gmail.com}

}}

\date{}
\title{RSFT functors for strong cobordisms and applications}
\author{Agustin Moreno, Zhengyi Zhou}

\begin{document}
	\maketitle
\begin{abstract}
We extend the hierarchy functors of \cite{MZ22} to the case of strong symplectic cobordisms, via deformations with Maurer--Cartan elements. In particular, we prove that the concave boundary of a strong cobordism has finite algebraic planar torsion if the convex boundary does, which yields a functorial proof of finite algebraic planar torsion for contact manifolds admitting strong cobordisms to overtwisted contact manifolds. We also show the existence of contact $3$-folds without strong cobordisms to the standard contact $3$-sphere, that are not cofillable. We also include generalizations of the theory relating our notion of algebraic planar torsion to Latschev--Wendl's notion of algebraic torsion \cite{LW}, discussing variations from counting holomorphic curves with general constraints and invariants extracted from higher genera holomorphic curves from an algebraic perspective. 
\end{abstract}


\section{Introduction}

In \cite{MZ22}, the authors introduced the concept of $BL_\infty$ algebras, a framework underlying the algebraic aspects of rational symplectic field theory (RSFT). This setup was used to introduce several invariants for contact manifolds, including the notion of \emph{planarity} (roughly speaking, from rational curves with a point constraint in symplectizations), as well as \emph{algebraic planar torsion} (roughly speaking, from rational curves without negative punctures in symplectizations), designed to provide a notion of complexity within the category of all contact manifolds, together with \emph{Liouville/exact} cobordisms between them as morphisms. Several properties and examples were discussed, and in particular, every level of algebraic planar torsion and planarity was shown to be achieved by a contact manifold. In this paper, we consider the $BL_\infty$ formalism in the \emph{strong} symplectic cobordism category and study the functorial properties of algebraic planar torsion and planarity under strong cobordisms. We also discuss several variations, including the higher genus theory and the theory with general constraints from an algebraic perspective.

\subsection{Hierarchy functors for strong cobordisms}
In \cite{MZ22}, the authors introduced the notion of algebraic planar torsion $\PT$, which is an analog, in the context of RSFT, of Latschev and Wendl's algebraic torsion in the context of the full SFT. The finiteness of $\PT$ is an obstruction to strong fillings, and $\PT$ of the convex boundary is at least $\PT$ of the concave boundary in an exact cobordism. By \cite[Theorem 3.16, 3.17]{MZ22} and \cite[Theorem 1]{MR3128981}, such monotonicity does not hold for strong cobordisms. However, the following theorem implies that the \emph{finiteness} of $\PT$ is preserved under a strong cobordism. This answers a question of Latschev and Wendl \cite[Question 2]{LW} affirmatively in the context of RSFT. The original question can be solved similarly.

\begin{theorem}\label{thm:strong_torsion}
If there is a strong cobordism from $Y_-$ (concave end) to $Y_+$ (convex end) such that $\PT(Y_+)<+\infty$, then $\PT(Y_-)<+\infty$. In particular, if $Y_+$ is algebraically overtwisted, i.e.\ the contact homology of $Y_+$ vanishes, then $Y_-$ has finite algebraic planar torsion.
\end{theorem}
There are many examples of tight contact manifolds that admit strong cobordisms to overtwisted contact manifolds, e.g.\ any closed contact $3$-manifold with planar torsion \cite{MR3128981}. \Cref{thm:strong_torsion} yields a new proof that contact manifolds with planar torsion have finite algebraic planar torsion. Strong cobordisms induce $BL_\infty$ morphisms from the RSFT of the convex boundary $Y_+$ to the RSFT deformed by a Maurer--Cartan element of the concave boundary $Y_-$. The Maurer--Cartan element is defined by counting rigid genus $0$ holomorphic caps with multiple negative punctures in the strong cobordism. The deformed functoriality implies that the deformed RSFT of $Y_-$ has $\PT$ no larger than $\PT(Y_+)$, and then we can deduce that the undeformed $\PT(Y_-)$ is also finite, but it could be larger than $\PT(Y_+)$. Maurer-Cartan deformations have been studied in various symplectic invariants, starting from bounding cochains in Lagrangian Floer theory by Fukaya, Oh, Ohta, and Ono \cite{zbMATH05616008}, Maurer-Cartan deformations of $BV_\infty$ algebras in the full SFT by Cieliebak and Latschev \cite{cieliebak2009role}, Maurer-Cartan deformations of $IBL_\infty$ algebras in the full SFT by Cieliebak, Fukaya and Latschev \cite{cieliebak2015homological}, and Maurer-Cartan deformations of $L_\infty$ algebras in symplectic cohomology by Borman, Sheridan, and Varolgunes \cite{zbMATH07557488}.

Another hierarchy functor introduced in \cite{MZ22} is the planarity $\Pl$, which measures the tightness of, typically, fillable contact manifolds in the exact cobordism category. Similar to $\PT$, the finiteness of $\Pl_{\Lambda}$ (the planarity using augmentations over the Novikov field $\Lambda$) of a contact manifold $Y$ is an obstruction to strong semi-fillings, i.e.\ connected strong fillings $W$ of a disjoint union $Y\sqcup Y'$, where both $Y,Y'\neq \emptyset$. Understanding obstructions to semi-fillings can be useful for obstructions to fillings, see e.g.\ \cite{MR1383953}. We say that $Y$ is \emph{strongly co-fillable} if there is a strong semi-filling of $Y$.
\begin{theorem}\label{thm:cofilling}
    If $\Pl_{\Lambda}(Y)<+\infty$, then $Y$ is not strongly co-fillable.
\end{theorem}
There are many non-co-fillable contact manifolds, e.g.\ $(S^{2n-1},\xi_{\mathrm{std}})$ \cite{MR1091622}, or more generally planar contact manifolds \cite{MR2126827}, and iterated planar contact manifolds \cite{MR4391886}. All of them have finite planarity \cite{MZ22}. Indeed, the proof of \Cref{thm:cofilling} is essentially an RSFT packaging of Eliashberg's argument for the simply connected condition for fillings of $(S^{2n-1},\xi_{\mathrm{std}})$, which forms part of the proof of the celebrated Eliashberg--Floer--McDuff theorem \cite{MR1091622}.

One would expect that the finiteness of planarity is also preserved under strong cobordisms. The functorial properties of Maurer-Cartan deformed structures in the definition of planarity do hold. But we will face two problems: one is extracting finite planarity from a completed algebra, as Maurer-Cartan deformed theory has to be defined on the completion of the $BL_\infty$ algebra with respect to a suitable filtration to reflect the Gromov compactness for holomorphic curves in strong cobordisms; the other is the possibility of having the planarity blowing up when we enumerate all possible augmentations, i.e.\ the upper bound of planarity of the concave boundary from the deformed functoriality depends on the choice of augmentations and eventually goes to infinity when we test through all augmentations. Hence, we can only prove the functoriality of the finiteness of planarity if the convex boundary's finite planarity is contributed by curves that do \emph{not} depend on augmentations. This is, in fact, the situation for all examples of finite planarity considered in \cite{MZ22}.  

\begin{theorem}\label{thm:planarity_strong}
    Let $Y_+$ be a contact manifold with finite planarity, where the contributing curves do not depend on augmentations in the sense of \Cref{prop:simple} (holds for all explicit examples in \cite{MZ22}), and let $X$ be a strong cobordism from $Y_-$ to $Y_+$. Then $\Pl(Y_-)\le \Pl_{\Lambda}(Y_-)<+\infty$.
\end{theorem}

Following \cite[Lemma 4.3]{Seidel_biased}, affine varieties are equipped with natural Liouville/Weinstein structures, whose ideal boundary is equipped with natural contact structures. As a corollary of the proof of \Cref{thm:planarity_strong}, but based on different geometric inputs, we have the following explicit upper bounds, taken from \cite{MZ22}. 

\begin{theorem}\label{cor:divsor}
    Let $s$ be a holomorphic section of $\cO(d)$ for $d\ge 1$ over $\CP^n$. Let $Y$ be the (ideal) contact boundary of the affine variety $\CP^n \backslash (s^{-1}(0)\cup H)$, where $H$ is a hyperplane ($s^{-1}(0)\cup H$ is an ample divisor representing $\cO(d+1)$\footnote{Alternatively, $\CP^n \backslash (s^{-1}(0)\cup H)=\C^n\backslash V_f$, where $V_f$ is the zero set of a polynomial of degree at most $d$. Then $\C^n_{z_1,\ldots,z_n}\backslash V_f=V_{z_0f}\subset\C^{n+1}_{z_0,z_1,\ldots,z_n}$.}). Then $\Pl(Y)\le \Pl_{\Lambda}(Y)\le d+1$. 
\end{theorem}

Following \cite[Lemma 7.1]{MZ22}, it suffices to prove \Cref{cor:divsor} for a generic $s$, such that $s$ is transversal to zero and $H$ intersects $s^{-1}(0)$ transversely.

\begin{remark}
    Let $D_1,D_2$ be two ample divisors of $\CP^n$. Naively, one may expect that the planarity of $\partial(\CP^n\backslash D_1)$ is not larger than that of $\partial(\CP^n\backslash (D_1\cup D_2))$, as many computations in \cite{MZ22} indicated. Combining \Cref{cor:divsor} and \cite[Theorem 7.13]{MZ22}, we know that such an expectation does not hold in general, e.g.\ for $D_1$ a smooth hypersurface of sufficiently high degree and $D_2=H$ in $\CP^n$ for $n\ge 2$. 
\end{remark}

Lastly, we recall that, given an open book supporting a contact structure on $M$, Bourgeois \cite{Bourgeois} constructed a contact structure on $M\times \mathbb T^2$. Several applications of this construction were obtained in \cite{BGM,BGMZ}. For these manifolds, we can estimate their complexity in general.

\begin{theorem}\label{cor:BO}
The planarity (over $\Q$ or $\Lambda$) of any Bourgeois contact manifold with trivial first Chern class is at most $2$. 
\end{theorem}

We also make the following observation that there exist contact $3$-manifolds without strong cobordism to $(S^3,\xi_{\mathrm{std}})$ that are not cofillable, which settles a question of Wendl \cite[Question 2]{MR3128981} in the negative.

\begin{theorem}\label{thm:no_cofilling}
For $g\ge 1$ and $d>2g-2$, let $Y_{d,g}$ be the pre-quantization bundle, equipped with the Boothby-Wang contact structure, over a surface $\Sigma_g$ of genus $g$ with first Chern class $-d$. Then $Y_{d,g}$ is not strongly co-fillable and there is no strong cobordism with concave boundary $Y_{d,g}$ and convex boundary $(S^3,\xi_{\mathrm{std}})$. 
\end{theorem}

\subsection{Higher genera and other generalizations}
In \cite{MR2284049}, Gay asked if, under the equivalence relation defined by the existence of strong cobordisms in both directions, there is an equivalence class of contact $3$-manifolds beyond $[\emptyset]$ and $[(S^3,\xi_{\mathrm{ot}})]$, where $\xi_{\mathrm{ot}}$ is an overtwisted contact structure. This question is equivalent to whether there is a contact $3$-manifold that is not strongly fillable and has no strong cobordism to an overtwisted manifold. The same question was also asked by Wendl \cite[Question 1]{MR3128981}. It was answered in the affirmative by the existence of contact manifolds with non-vanishing Heegaard-Floer/monopole-Floer/ECH contact invariants that are not strongly fillable \cite{MR2087073}. This follows from the fact that the vanishing of contact invariants is preserved under strong cobordisms by \cite{MR4127085} (in monopole-Floer) and \cite{hutchings} (in ECH).

If we approach this question (and its higher-dimensional analog) from the SFT perspective, the functoriality in \Cref{thm:strong_torsion} implies that a non-strongly-fillable contact manifold with infinite $\PT$ can also serve as a solution. Therefore, we need obstructions to strong fillings beyond the finiteness of $\PT$. The obvious direction is upgrading the invariants from RSFT to those from the full SFT, i.e.\ considering curves with all genera. The counterpart of $\PT$ in the full SFT is the algebraic torsion $\mathrm{AT}$ in \cite{LW}. Although both torsions share the same class of examples of $3$-manifolds with finite torsions, they do not have any obvious direct relation, at least algebraically. In fact, they are related by a grid of torsions filtered by both the number of punctures and genera. However, finding examples realizing those torsions is a challenging task, as we need to find certain higher genus curves in the symplectization of the contact manifold. In \S \ref{s5}, which was contained in an earlier version of \cite{MZ22}, we discuss the higher genus analogs of our hierarchy functors, and the relation between two notions of torsions (see also \cite{Janko22} for a reformulation of our invariants in the language of the original description of SFT by Eliashberg, Givental and Hofer \cite{EGH}). We will also discuss several variations, in the spirit of Siegel's higher capacities \cite{siegel2019higher}, on invariants coming from counting rational curves with multiple point constraints, as well as more general constraints for our future applications. 

\medskip

\textbf{Acknowledgements.}We thank the anonymous referee for their careful reading of the manuscript and their insightful comments that greatly improved our paper. A.\ Moreno is supported by the National Science Foundation under Grant No.\ DMS-1926686, by the Sonderforschungsbereich TRR 191 Symplectic Structures in Geometry, Algebra and Dynamics, funded by the DFG (Projektnummer 281071066 – TRR 191), and also by the DFG under Germany's Excellence Strategy EXC 2181/1 - 390900948 (the Heidelberg STRUCTURES Excellence Cluster). Z.\ Zhou is supported by the National Key R\&D Program of China under Grant No. 2023YFA1010500, the National Natural Science Foundation of China under Grant No.\ 12288201 and 12231010.
\section{Algebraic Formalism}\label{s2}
\subsection{$BL_\infty$ Algebras}\label{ss:BL}
We begin by recalling the main algebraic notions from \cite{MZ22}, following the same (but streamlined) exposition. Let $V$ be a $\Z/2$-graded vector space over a field $\bk$ of characteristic zero. Consider the $\Z/2$-graded symmetric algebra $S V:=\bigoplus_{k\ge 0} S^k V$ and the non-unital symmetric algebra $\overline{S} V =\bigoplus_{k\ge 1} S^k V$, where $S^kV=\otimes^k V/\Sigma_k$ (with $\Sigma_k$ the symmetric group on $k$ letters) in the graded sense, i.e.,
$$a b=(-1)^{|a||b|} b a$$
for homogeneous elements $a,b$ in $SV$ and $\overline{S}V$.

Let $EV=\overline{S}SV$. We use $\odot$ for the product on the outer symmetric product $\overline{S}$ and $\ast$ (often omitted) for the product on the inner symmetric product $S$. Given the data of a linear operator $p^{k,l}:S^kV \to S^l V$ for $k\ge 1, l\ge 0$, we define a map $\widehat{p}^{k,l}:S^k SV \to SV$ by the following properties.
\begin{enumerate}
	\item $\widehat{p}^{k,l}|_{\odot^k V\subset S^kSV}$ is defined by $p^{k,l}$.
	\item If $w_i\in \bk$, then $\widehat{p}^{k,l}(w_1\odot \ldots \odot w_k)=0$.
	\item $\widehat{p}^{k,l}$ satisfies the Leibniz rule in each argument, i.e., we have
        \begin{equation}\label{eqn:pkl}
            \widehat{p}^{k,l}(w_1\odot \ldots \odot w_k)=\sum_{j=1}^m(-1)^{\square} v_1 \ldots  v_{j-1} \widehat{p}^{k,l}(w_1\odot \ldots \odot v_j \odot \ldots \odot w_k) v_{j+1} \ldots  v_m,
        \end{equation}
	where $w_i=v_1 \ldots  v_m$ and
 \begin{equation}\label{eqn:sign}
\square = \sum_{s=1}^{i-1}|w_s|\cdot \sum_{s=1}^{j-1}|v_s|+\sum_{s=1}^{j-1}|v_s||p^{k,l}|+\sum_{s=i+1}^n |w_s|\cdot \sum_{s=j+1}^{m}|v_s|.
 \end{equation}
\end{enumerate}
Explicitly, $\widehat{p}^{k,l}$ is defined by
\begin{equation}\label{eq:p_hat_1}
    w_1\odot \ldots \odot w_k \mapsto \sum_{\substack{(i_1,\ldots,i_k) \\ 1\le i_j\le n_j}} (-1)^{\bigcirc} p^{k,l}(v^1_{i_1}\ldots v^k_{i_k}) \check{w}_1 \ldots \check{w}_k,
\end{equation}
where $w_j=v^j_1 \ldots  v^j_{n_j}$, $\check{w}_j=v^j_1 \ldots  \check{v}^j_{i_j} \ldots  v^j_{n_j}$, and $\bigcirc$ is determined by $w_1 \ldots  w_k=(-1)^\bigcirc v^1_{i_1} \ldots  v^k_{i_k} \check{w}_1 \ldots \check{w}_k$. Then we define $\widehat{p}^k:S^k S V \to S V$ by $\bigoplus_{l\ge 0} \widehat{p}^{k,l}$. Then we can define $\widehat{p}:EV \to EV$ by
\begin{equation}\label{eqn:q}
    w_1\odot \ldots \odot w_n \mapsto \sum_{k=1}^n\sum_{\sigma \in Sh(k,n-k)}(-1)^{\diamond} \widehat{p}^k(w_{\sigma(1)}\odot \ldots \odot w_{\sigma(k)})\odot w_{\sigma(k+1)}\odot \ldots \odot w_{\sigma(n)},
\end{equation}
where $Sh(k,n-k)$ is the subset of permutations $\sigma$ such that $\sigma(1)<\ldots<\sigma(k)$ and $\sigma(k+1)<\ldots < \sigma(n)$, and
$$\diamond=\sum_{\substack {1\le i < j \le n\\ \sigma(i)>\sigma(j)}}|w_i||w_j|.$$

\begin{definition}[{\cite[Definition 2.3]{MZ22}}]\label{def:BL}
$(V,\{p^{k,l}\})$ is a $BL_\infty$ algebra if $\widehat{p}\circ \widehat{p}=0$ and $|\widehat{p}|=1$.
\end{definition}
This definition implicitly requires that $\widehat{p}$ is well-defined. To ensure this, we can, for example, impose that for any $v_1,\ldots,v_k\in V$, there are at most finitely many $l$ such that $p^{k,l}(v_1\ldots v_k)\ne 0$.

\subsection{The Rules for Tree Calculus}\label{ss:tree}
A useful way to explain the combinatorics of operations is the following description using graphs, which appeared in \cite[\S 3.4.2]{siegel2019higher}. The main advantage of this graphical language is that it frees us from keeping track of signs and explicit components of compositions, e.g., in \Cref{eqn:pkl}, which are governed by graphs.

Let $w\in S^kV$. We can represent $w$ by an element $\overline{w}$ in $\otimes^k V$, i.e., $\overline{w}=\sum_{i=1}^N c_iv^i_1\otimes \ldots \otimes v^i_k$ for $c_i\in \bk$ and $v_*^*\in V$, such that $\pi(\overline{w})=w$ for $\pi:\otimes^kV \to S^kV$. We represent it by a rooted tree with $k$ leaves (represented by $\bullet$) labeled by $\overline{w}$. The leaves are ordered from left to right to indicate the $k$ copies of $V$ in $\otimes^k V$. When $\overline{\omega} = v_1\otimes \ldots \otimes v_k$, we may label the leaves by $v_1,\ldots,v_k$ to mean the same thing. We can view a general labeled tree as a formal linear combination of such trees with labeled leaves.
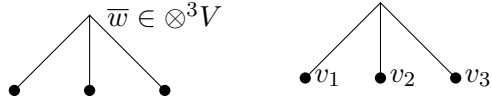
    	\begin{center}
        \begin{tikzpicture}
        \node at (0,0) [circle,fill,inner sep=1.5pt] {};
		\node at (1,0) [circle,fill,inner sep=1.5pt] {};
		\node at (2,0) [circle,fill,inner sep=1.5pt] {};
		\draw (0,0) to (1,1) to (1,0);
		\draw (1,1) to (2,0);
        \node at (2,1) {$\overline{w}\in \otimes^3 V$};
        \end{tikzpicture}
        \qquad
		\begin{tikzpicture}
		\node at (0,0) [circle,fill,inner sep=1.5pt] {};
        \node at (0.3,0) {$v_1$};
		\node at (1,0) [circle,fill,inner sep=1.5pt] {};
        \node at (1.3,0) {$v_2$};
		\node at (2,0) [circle,fill,inner sep=1.5pt] {};
        \node at (2.3,0) {$v_3$};
		\draw (0,0) to (1,1) to (1,0);
		\draw (1,1) to (2,0);
		\end{tikzpicture}
	\end{center}
Now let $s\in S^kSV$. We can represent $s$ by $\overline{s}\in \boxtimes^k TV$, where $TV=\oplus_{k\in \N}(\otimes^k V)$. Here we use $\boxtimes$ to differentiate it from the inner tensor $\otimes$. We write
$$\overline{s}=\sum_{i=1}^N c_i\overline{w}^i_1 \boxtimes \ldots  \boxtimes \overline{w}^i_k, \quad c_i\in \bk, \overline{w}_*^*\in \otimes^{m^*_*}V.$$
We represent $\overline{w}^i_1 \boxtimes\ldots  \boxtimes \overline{w}^i_k$ by an ordered forest of labeled trees as follows. Then $\overline{s}$ is a formal linear combination of such forests.
\begin{figure}[H]
    \begin{center}
		\begin{tikzpicture}
		\node at (0,0) [circle,fill,inner sep=1.5pt] {};
		\node at (1,0) [circle,fill,inner sep=1.5pt] {};
		\node at (2,0) [circle,fill,inner sep=1.5pt] {};
		\node at (3,0) [circle,fill,inner sep=1.5pt] {};
		\node at (4,0) [circle,fill,inner sep=1.5pt] {};
		\node at (5,0) [circle,fill,inner sep=1.5pt] {};
		\node at (6,0) [circle,fill,inner sep=1.5pt] {};
		\node at (7,0) [circle,fill,inner sep=1.5pt] {};
		
		\draw (0,0) to (1,1) to (1,0);
		\draw (1,1) to (2,0);
		\draw (3,0) to (4,1) to (4,0);
		\draw (4,1) to (5,0);
		\draw (6,0) to (6.5,1) to (7,0);

       \node at (2,1) {$\overline{w}_1\in \otimes^3V$};
       \node at (5,1) {$\overline{w}_2\in \otimes^3V$};
       \node at (7.5, 1) {$\overline{w}_3\in \otimes^2V$};
	\end{tikzpicture}
	\end{center}
    \caption{A forest of labeled trees}
    \label{fig:forest}
\end{figure}
We represent the operation $p^{k,l}:S^kV\to S^lV$ by a graph with $k+l+1$ vertices: $k$ top input vertices, $l$ bottom output vertices, and one middle vertex $\tikz\draw[black,fill=white] (0,-1) circle (0.4em);$ labeled by $p^{k,l}$ representing the operation type.
    
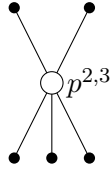
\begin{center}
        \begin{tikzpicture}
        \node at (2,0) [circle,fill,inner sep=1.5pt] {};
		\node at (3,0) [circle,fill,inner sep=1.5pt] {};
        \draw (2,0) to (2.5,-1) to (3,0);
		\draw (2,-2) to (2.5,-1) to (2.5,-2);
		\draw (2.5,-1) to (3,-2);
		\node at (2.5,-1) [circle, fill=white, draw, outer sep=0pt, inner sep=3 pt] {};
        \node at (2,-2) [circle,fill,inner sep=1.5pt] {};
	    \node at (3,-2) [circle,fill,inner sep=1.5pt] {};
	    \node at (2.5,-2) [circle,fill,inner sep=1.5pt] {};
        \node at (3,-1) {$p^{2,3}$};
        \end{tikzpicture}
    \end{center}
    
So far the discussion is completely formal without any actual content. The real content lies in the following interpretation of a glued graph, whose definition will be clear from an example.
\begin{figure}[H]
	\begin{center}
		\begin{tikzpicture}
		\node at (0,0) [circle,fill,inner sep=1.5pt] {};
		\node at (1,0) [circle,fill,inner sep=1.5pt] {};
		\node at (2,0) [circle,fill,inner sep=1.5pt] {};
		\node at (3,0) [circle,fill,inner sep=1.5pt] {};
		\node at (4,0) [circle,fill,inner sep=1.5pt] {};
		\node at (5,0) [circle,fill,inner sep=1.5pt] {};
		\node at (6,0) [circle,fill,inner sep=1.5pt] {};
		\node at (7,0) [circle,fill,inner sep=1.5pt] {};
		
		\draw (0,0) to (1,1) to (1,0);
		\draw (1,1) to (2,0);
		\draw (3,0) to (4,1) to (4,0);
		\draw (4,1) to (5,0);
		\draw (6,0) to (6.5,1) to (7,0);

		\draw (2,0) to (2.5,-1) to (3,0);
		\draw (2,-2) to (2.5,-1) to (2.5,-2);
		\draw (2.5,-1) to (3,-2);
		\node at (2.5,-1) [circle, fill=white, draw, outer sep=0pt, inner sep=3 pt] {};
		
		\node at (0,-2) [circle,fill,inner sep=1.5pt] {};
		\node at (1,-2) [circle,fill,inner sep=1.5pt] {};
		\node at (2,-2) [circle,fill,inner sep=1.5pt] {};
		\node at (3,-2) [circle,fill,inner sep=1.5pt] {};
		\node at (4,-2) [circle,fill,inner sep=1.5pt] {};
		\node at (5,-2) [circle,fill,inner sep=1.5pt] {};
		\node at (6,-2) [circle,fill,inner sep=1.5pt] {};
		\node at (7,-2) [circle,fill,inner sep=1.5pt] {};
		\node at (2.5,-2) [circle,fill,inner sep=1.5pt] {};
		
		\draw[dashed] (0,0) to (0,-2);
		\draw[dashed] (1,0) to (1,-2);
		\draw[dashed] (4,0) to (4,-2);
		\draw[dashed] (5,0) to (5,-2);
		\draw[dashed] (6,0) to (6,-2);
		\draw[dashed] (7,0) to (7,-2);
		\node at (0.4,0) {$v_1$};
		\node at (1.4,0) {$v_2$};
		\node at (2.4,0) {$v_3$};
		\node at (3.4,0) {$v_4$};
		\node at (4.4,0) {$v_5$};
		\node at (5.4,0) {$v_6$};
		\node at (6.4,0) {$v_7$};
		\node at (7.4,0) {$v_8$};
		\end{tikzpicture}
	\end{center}
    \caption{Gluing forests $\Leftrightarrow$ applying operations}
    \label{fig:gluing}
\end{figure}
The above glued graph represents a forest: we first fix a representative $\overline{p}$ of $p^{2,3}(v_3v_4)$ in $\otimes^3V$. The glued forest in Figure \ref{fig:forest} represents $\pm (v_1\otimes v_2 \otimes \overline{p}\otimes v_5\otimes v_6)\boxtimes (v_7\otimes v_8)$. In the gluing, we do not create cycles in the glued graph (which in the context of SFT translates to the fact that underlying curves of those SFT buildings, viewed as nodal curves, have arithmetic genus zero). Each dashed line represents the identity map, and each connected component represents a tree in the output. Drawing the input element as a forest of ordered trees with ordered leaves corresponds to choosing representatives from the tensor product, not the symmetric product. Finally, when we draw the glued graph on a plane as above (i.e., choosing an order of the trees and leaves, hence edges may cross), it will determine a representative in the tensor product, but we view different orders as equivalent up to the obvious sign change. For example, the following is an equivalent gluing to Figure \ref{fig:gluing}, but with an extra sign when viewing it in the tensor product.
	
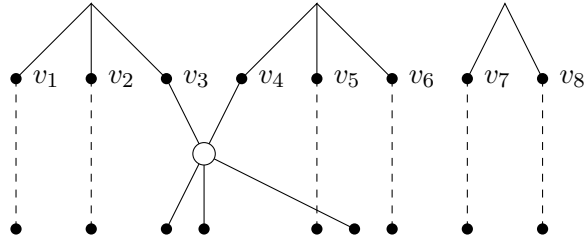
\begin{center}
		\begin{tikzpicture}
		\node at (0,0) [circle,fill,inner sep=1.5pt] {};
		\node at (1,0) [circle,fill,inner sep=1.5pt] {};
		\node at (2,0) [circle,fill,inner sep=1.5pt] {};
		\node at (3,0) [circle,fill,inner sep=1.5pt] {};
		\node at (4,0) [circle,fill,inner sep=1.5pt] {};
		\node at (5,0) [circle,fill,inner sep=1.5pt] {};
		\node at (6,0) [circle,fill,inner sep=1.5pt] {};
		\node at (7,0) [circle,fill,inner sep=1.5pt] {};
		
		\draw (0,0) to (1,1) to (1,0);
		\draw (1,1) to (2,0);
		\draw (3,0) to (4,1) to (4,0);
		\draw (4,1) to (5,0);
		\draw (6,0) to (6.5,1) to (7,0);

		\draw (2,0) to (2.5,-1) to (3,0);
		\draw (2,-2) to (2.5,-1) to (2.5,-2);
		\draw (2.5,-1) to (4.5,-2);
		\node at (2.5,-1) [circle, fill=white, draw, outer sep=0pt, inner sep=3 pt] {};
		
		\node at (0,-2) [circle,fill,inner sep=1.5pt] {};
		\node at (1,-2) [circle,fill,inner sep=1.5pt] {};
		\node at (2,-2) [circle,fill,inner sep=1.5pt] {};
		\node at (4.5,-2) [circle,fill,inner sep=1.5pt] {};
		\node at (4,-2) [circle,fill,inner sep=1.5pt] {};
		\node at (5,-2) [circle,fill,inner sep=1.5pt] {};
		\node at (6,-2) [circle,fill,inner sep=1.5pt] {};
		\node at (7,-2) [circle,fill,inner sep=1.5pt] {};
		\node at (2.5,-2) [circle,fill,inner sep=1.5pt] {};
		
		\draw[dashed] (0,0) to (0,-2);
		\draw[dashed] (1,0) to (1,-2);
		\draw[dashed] (4,0) to (4,-2);
		\draw[dashed] (5,0) to (5,-2);
		\draw[dashed] (6,0) to (6,-2);
		\draw[dashed] (7,0) to (7,-2);
		\node at (0.4,0) {$v_1$};
		\node at (1.4,0) {$v_2$};
		\node at (2.4,0) {$v_3$};
		\node at (3.4,0) {$v_4$};
		\node at (4.4,0) {$v_5$};
		\node at (5.4,0) {$v_6$};
		\node at (6.4,0) {$v_7$};
		\node at (7.4,0) {$v_8$};
		\end{tikzpicture}
	\end{center}
The sign is determined similarly to \Cref{eqn:sign}; in the case of Figure \ref{fig:gluing}, the sign is $(-1)^{(|v_1|+|v_2|)|p^{2,3}|}$. In a formal description, we apply order changes to the input forest (edges can cross), then we glue $p^{k,l}$ such that there is no new edge crossing, and finally, we change the output order back to the chosen one. The final sign is given by the product of the sign changes of the two order changes and the sign of the composition using the Koszul-Quillen convention. Writing the forest using a glued graph, as in Figure \ref{fig:gluing}, contains slightly more refined information than just labeling the forest as in Figure \ref{fig:forest}, namely, we keep track of which leaves are from $p^{k,l}$ in a representative. The following observation is tautological.
\begin{proposition}
    The output of a glued forest is well-defined in $EV$.
\end{proposition}
To enumerate all admissible gluings, each output leaf and tree is considered as different. However, we do not differentiate the input leaves of $p^{k,l}$. Therefore, when we glue a $p^{k,l}$ component, we pick $k$ trees (this is $Sh(k,n-k)$ in \Cref{eqn:q}) from the forest and then one leaf from each chosen tree (that is $1\le i_j\le n_j$ in \Cref{eq:p_hat_1}) to glue to $p^{k,l}$. For example, in the situation of Figure \ref{fig:gluing}, we have $3*3+3*2+3*2=21$ direct ways to glue $p^{2,3}$. The ambiguity from choosing a representative of the input is then eliminated by summing over all possible gluings, as shown by the following tautological observation.
\begin{proposition}
    When summed over all possible gluings of one $p^{k,l}$, the output is independent of the choice of representatives of the input forest.
\end{proposition}
Combining the above two propositions, we see that gluing forests corresponds to operations on $EV$. Indeed, the language of trees and forests packages the signs and components in \Cref{eqn:pkl,eq:p_hat_1,eqn:q} by providing a geometric intuition. The translation into forests makes it easier to understand algebraic relations; for example, many relations come from interpreting the same glued forests in two different ways. The following is a dictionary of the algebraic formulae in \S \ref{ss:BL} in terms of forests.
\begin{center}
\begin{tabular}{c|c}
  $\widehat{p}^{k,l}$ on $\odot^kV$ & \makecell{(unique) gluing of $p^{k,l}$ to a  \\ forest of $k$ trees with a single leaf} \\
  \hline
  $\widehat{p}^{k,l}$ in \Cref{eq:p_hat_1}  & \makecell{sum of gluings of $p^{k,l}$ to a \\ forest of $k$ trees to get a tree}    \\
  \hline
  $\widehat{p}^k$  & \makecell{sum of gluings of $p^{k,*}$ to a \\ forest of $k$ trees to get a tree}  \\
  \hline
   $\widehat{p}$ in  \Cref{eqn:q}  & \makecell{sum of gluings of $p^{*,*}$ to a \\ forest to get a forest} \\
   \hline
   $\widehat{p}^{k,l}(w_1\odot \ldots \odot w_k)=0$ if $w_i\in \bk$ & \makecell{no way to glue $p^{k,l}$ to a forest \\with a tree without a leaf to get a tree}\\
\end{tabular}
\end{center}
In the following, we will use the graphical description for morphisms, augmentations, and Maurer-Cartan elements without giving explicit formulae like \Cref{eqn:pkl,eqn:q}. The explicit formulae, except for the Maurer-Cartan elements, can be found in \cite[\S 2]{MZ22}.

\subsection{Morphisms} We now define morphisms between $BL_{\infty}$ algebras. Consider a family of operators $\{\phi^{k,l}:S^k V\to S^l V'\}_{k\ge 1,l\ge 0}$; we can construct a map $\widehat{\phi}:EV\to EV'$ from the following tree description. To represent $\phi^{k,l}$, we use a graph similar to the one representing $p^{k,l}$ but replace $\tikz\draw[black,fill=white] (0,-1) circle (0.4em);$ by $\tikz\draw[black,fill=black] (0,-1) circle (0.4em);$ to indicate that they are maps of different roles. To represent a component configuration in the definition of $\widehat{\phi}$ on $S^{i_1}V\odot\ldots \odot S^{i_n}V$, we glue a family of graphs representing $\phi^{k,l}$ such that the input vertices and the output vertices of the top forest are completely paired and glued and the resulting graph has no cycles. Then $\widehat{\phi}$ is the sum of all possible configurations. Unlike the definition of $\widehat{p}$, where we need to glue exactly one $p^{k,l}$ graph, it is possible that we do not glue in any $\phi^{k,l}$ graphs. This is the case when the input is in $\odot^m\bk$ and $\widehat{\phi}$ is the identity in such a case, i.e., $\widehat{\phi}(1\odot \ldots \odot 1) = 1\odot \ldots \odot 1$.
\begin{figure}[H]
	\begin{center}
		\begin{tikzpicture}
		\node at (0,0) [circle,fill,inner sep=1.5pt] {};
		\node at (1,0) [circle,fill,inner sep=1.5pt] {};
		\node at (2,0) [circle,fill,inner sep=1.5pt] {};
		\node at (3,0) [circle,fill,inner sep=1.5pt] {};
		\node at (4,0) [circle,fill,inner sep=1.5pt] {};
		\node at (5,0) [circle,fill,inner sep=1.5pt] {};
		\node at (6,0) [circle,fill,inner sep=1.5pt] {};
		\node at (7,0) [circle,fill,inner sep=1.5pt] {};
		
		\draw (0,0) to (1,1) to (1,0);
		\draw (1,1) to (2,0);
		\draw (3,0) to (4,1) to (4,0);
		\draw (4,1) to (5,0);
		\draw (6,0) to (6.5,1) to (7,0);

		\draw (2,0) to (2.5,-1) to (3,0);
		\draw (2,-2) to (2.5,-1) to (2.5,-2);
		\draw (2.5,-1) to (3,-2);
		\node at (2.5,-1) [circle, fill, draw, outer sep=0pt, inner sep=3 pt] {};
		
		\node at (0,-2) [circle,fill,inner sep=1.5pt] {};
		\node at (1,-2) [circle,fill,inner sep=1.5pt] {};
		\node at (2,-2) [circle,fill,inner sep=1.5pt] {};
		\node at (3,-2) [circle,fill,inner sep=1.5pt] {};
		\node at (4,-2) [circle,fill,inner sep=1.5pt] {};
		\node at (2.5,-2) [circle,fill,inner sep=1.5pt] {};
	
		\draw (0,0) to (0,-2);
		\draw (1,0) to (1,-2);
		\draw (4,0) to (4,-2);
		\draw (5,0) to (5.5,-1);
		\draw (5.5,-1) to (6,0);
		\draw (7,0) to (7,-1);
		
		\node at (0,-1) [circle, fill, draw, outer sep=0pt, inner sep=3 pt] {};
		\node at (1,-1) [circle, fill, draw, outer sep=0pt, inner sep=3 pt] {};
		\node at (4,-1) [circle, fill, draw, outer sep=0pt, inner sep=3 pt] {};
		\node at (5.5,-1) [circle, fill, draw, outer sep=0pt, inner sep=3 pt] {};
		\node at (7,-1) [circle, fill, draw, outer sep=0pt, inner sep=3 pt] {};
		\end{tikzpicture}
	\end{center}
	\caption{A component of $\widehat{\phi}$ from $S^3V \odot S^3 V \odot S^2 V$ to $S^6V^\prime$}
\end{figure}

\begin{definition}[{\cite[Definition 2.10]{MZ22}}]\label{def:morphism}
 $\widehat \phi$ is a $BL_{\infty}$ morphism from $(V,\widehat p)$ to $(V',\widehat p')$ if $\widehat{\phi}\circ \widehat{p}=\widehat{p}'\circ \widehat{\phi}$ and $|\widehat{\phi}|=0$.
\end{definition}
This definition implicitly assumes that $\widehat{\phi}$ is well-defined. For example, when $\phi$ arises from counting holomorphic curves in an exact cobordism, we have for any $v_1 \ldots  v_k\in S^kV$ there are at most finitely many $l$ such that $\phi^{k,l}(v_1 \ldots  v_k)\ne 0$. Hence $\widehat{\phi}$ is well-defined.

The $(\psi\circ \phi)^{k,l}$ component of the composition of two $BL_\infty$ morphisms $\phi,\psi$ is represented by connected graphs without cycles glued from one level of $\phi$ and one level of $\psi$. It is clear from the tree description that $\widehat{\psi\circ \phi}=\widehat{\psi}\circ \widehat{\phi}$ and the composition is associative; see \cite[\S 2.4]{MZ22}.

\subsection{Augmentations}
The trivial vector space $V=\{0\}$ has a unique trivial $BL_\infty$ algebra structure with $p^{k,l}=0$. We use $\mathbf{0}$ to denote this trivial $BL_\infty$ algebra.
\begin{definition}[{\cite[Definition 2.13]{MZ22}}]
A $BL_\infty$ augmentation is a $BL_\infty$ morphism $\epsilon:V \to \mathbf{0}$, i.e., a family of operators $\epsilon^k:S^kV \to \bk$ satisfying Definition \ref{def:morphism}.
\end{definition}

Note that the existence of an augmentation implies that $H_*(EV,\widehat{p})\neq 0$, as $\epsilon$ descends to homology and $\epsilon_*1=1\ne 0\in H_*(E\mathbf{0},0)$. For a $BL_\infty$ algebra $V$, we let \begin{equation}\label{eq:B}
    E^kV=\overline{B}^kSV:=\bigoplus_{j=1}^k S^jSV,
\end{equation} which is a filtration (sentence length filtration) on $EV$ compatible with the differential $\widehat{p}$. We have $E\mathbf{0}=\bk \oplus S^2\bk \oplus +\ldots$ with $\widehat{p}=0$, and $H_*(E^k\mathbf{0})=E^k\mathbf{0}$ for all $k\ge 1$. We define $1_\mathbf{0}$ to be the generator in $E^1\mathbf{0}$; then $1_{\mathbf{0}}\ne 0\in H_*(E^k\mathbf{0})$ for all $k\ge 1$. Then we define $1_V\in H_*(E^kV)$ to be the image of $1_{\mathbf{0}}$ under the chain map $E^k\mathbf{0}\to E^kV$ induced by the trivial $BL_\infty$ morphism $\mathbf{0}\to V$.

\begin{definition}[{\cite[Definition 2.15]{MZ22}}]
We define the torsion of a $BL_\infty$ algebra $V$ to be
$$\TT(V):= \min\{k-1| 1_V=0 \in H_*(E^kV),k\ge 1\}.$$
Here, the minimum of an empty set is defined to be $\infty$.
\end{definition}
By definition, we have that $\TT(V)=0$ iff $1_V\in H^*(SV,\widehat{p}^1)$ is zero. Since $H^*(SV,\widehat{p}^1)$ is an algebra with $1_V$ a unit, in this case we have $H^*(SV,\widehat{p}^1)=0$. In general, a finite torsion is an obstruction to augmentations, and the hierarchy of torsion is functorial with respect to $BL_\infty$ morphisms.
\begin{proposition}[{\cite[\S 2.5]{MZ22}}]
 If there is a $BL_\infty$ morphism from $V$ to $V'$, then $\TT(V)\ge \TT(V')$. In particular, if $\TT(V)<\infty$, there is no augmentation (as $\TT(\mathbf{0})=+\infty$).
\end{proposition}

Given a $BL_\infty$ augmentation $\epsilon$, one can change coordinates to eliminate all constant terms $p^{k,0}$ following \cite[\S 2.5]{MZ22}. When $p^{k,0}=0$, $p^{1,1}$ becomes a differential on $V$ and $\{p^{k,1}\}$ form an $L_\infty$ structure (after fixing suitable sign conventions; see \cite[\S 2.1, 2.2]{MZ22}) on $V$. Therefore, we call the $BL_\infty$ structural maps, after such modification, the linearized structural maps. More precisely, we can define linearized structural maps via graphs as follows.
\begin{figure}[H]
	\begin{center}
		\begin{tikzpicture}
		\node at (0,0) [circle,fill,inner sep=1.5pt] {};
		\node at (1,0) [circle,fill,inner sep=1.5pt] {};
		\node at (2,0) [circle,fill,inner sep=1.5pt] {};
		\node at (3,0) [circle,fill,inner sep=1.5pt] {};
		\node at (0,-2) [circle,fill,inner sep=1.5pt] {};
		\node at (1,-2) [circle,fill,inner sep=1.5pt] {};
		\node at (1.5,-2) [circle,fill,inner sep=1.5pt] {};
		\node at (2,-2) [circle,fill,inner sep=1.5pt] {};
		\node at (3,-2) [circle,fill,inner sep=1.5pt] {};
		\draw[dashed] (0,0) to (0,-2);
		\draw[dashed] (3,0) to (3,-2);
		\draw (1,0) to (1.5,-1) to (1,-2) to (0.5,-3) to (0,-2);
		\draw (2,0) to (1.5,-1) to (2,-2) to (2.5,-3) to (3,-2);
		\draw (1.5,-1) to (1.5,-2);
		\node at (0.5,-3) [circle, fill, draw, outer sep=0pt, inner sep=3 pt] {};
		\node at (2.5,-3) [circle, fill, draw, outer sep=0pt, inner sep=3 pt] {};
		\node at (1.5,-1) [circle, fill=white, draw, outer sep=0pt, inner sep=3 pt] {};
		\node at (0.5,-3.4) {$\epsilon^2$};
		\node at (2.5,-3.4) {$\epsilon^2$};
		\node at (2, -1) {$p^{2,3}$};
		\end{tikzpicture}
	\end{center}
    \caption{A component of $p^{4,1}_{\epsilon}$}
\end{figure}
The linearized structural map $p^{k,l}_{\epsilon}$ is defined by gluing the level containing a $p^{*,*}$ in the definition of $\widehat{p}$ to the bottom level of the definition of $\widehat{\epsilon}$, such that the glued graph is a tree with $k$ inputs and $l$ outputs.

\begin{proposition}[{\cite[Proposition 2.18]{MZ22}}]\label{prop:linearized_BL}
    $\{p^{k,l}_{\epsilon}\}$ is a $BL_\infty$ algebra structure on $V$ and $p^{k,0}_{\epsilon}=0$ for all $k>0$.
\end{proposition}

$\{p^{k,l}_{\epsilon}\}$ is $\{p^{k,l}\}$ under a change of coordinates in the $BL_\infty$ sense. Namely, we let $F^{1,1}_{\epsilon}=\Id_V$ and $F^{k,0}_{\epsilon}=\epsilon^{k}$ and all other $F^{k,l}_{\epsilon}=0$. The same recipe for constructing $\widehat{\phi}$ from $\phi^{k,l}$ applies to give a $\widehat{F}_{\epsilon}$ on $EV$. If $\widehat{F}_{-\epsilon}$ denotes the map on $EV$ defined by $F^{1,1}_{-\epsilon}=\Id_V$ and $F^{k,0}_{-\epsilon}=-\epsilon^{k,0}$ and all other $F^{k,l}_{-\epsilon}=0$, then $\widehat{F}_{-\epsilon}$ is the inverse of $\widehat{F}_{\epsilon}$. Then we have $\widehat{p}_{\epsilon}:=\widehat{F}_{\epsilon}\circ \widehat{p}\circ \widehat{F}_{-\epsilon}:EV \to EV$; see discussions before \cite[Proposition 2.18]{MZ22}.

\subsection{Pointed Maps} Let $p_\bullet^{k,l}:S^kV \to S^lV, k\ge 1, l\ge 0$ be a family of linear maps; we can define $\widehat{p}_\bullet$ just like $\widehat{p}$ by gluing trees with the modification that the grading $|p_{\bullet}^{k,l}|$ (independent of $k,l$) is not necessarily $1\in \Z/2$.

\begin{definition}[{\cite[Definition 2.19]{MZ22}}]\label{def:pointed}
	We say $\{p_\bullet^{k,l}\}$ is a pointed map iff $\widehat{p}_\bullet\circ \widehat{p}=(-1)^{|\widehat{p}_{\bullet}|}\widehat{p}\circ \widehat{p}_\bullet$.
\end{definition}
In the context of SFT, $p_{\bullet}$ comes from counting holomorphic curves in the symplectization with a marked point mapped to a closed submanifold/closed chain. The degree of $p_{\bullet}$ is the dimension of the constraint. In applications, we typically consider a point constraint and $|p_{\bullet}|=0$. If we have an augmentation $\epsilon$ to $p$, we then define $$\widehat{p}_{\bullet,\epsilon}:=\widehat{F}_{\epsilon}\circ \widehat{p}_{\bullet}\circ \widehat{F}_{-\epsilon}.$$ This map is determined by the $p^{k,l}_{\bullet,\epsilon}$, which is defined similarly to $p^{k,l}_{\epsilon}$. We also have $\widehat{p}_{\epsilon}\circ \widehat{p}_{\bullet,\epsilon}=\widehat{p}_{\bullet,\epsilon}\circ \widehat{p}_{\epsilon}$. However, it is not necessarily true that $p^{k,0}_{\bullet,\epsilon}=0$. We let $\ell_{\bullet,\epsilon}^{k,0}:=p^{k,0}_{\bullet,\epsilon}$. Then $\widehat{\ell}_{\bullet,\epsilon}:=\sum_{k\ge 0} \ell_{\bullet,\epsilon}^{k,0}$ defines a chain morphism $(\overline{S} V, \widehat{\ell}_{\epsilon}) \to \bk$, where $\widehat{\ell}_{\epsilon}$ is the restriction of $\widehat{p}_{\epsilon}$ on $\overline{S} V \cong \bigoplus_{i=1}^{\infty} \odot^iV \subset EV$ then composed with the projection to $\bigoplus_{i=1}^{\infty} \odot^iV$. In other words, $\widehat{\ell}_{\epsilon}$ is determined by $\ell^k_{\epsilon}:=p^{k,1}_{\epsilon}$ by a similar recipe to $\widehat{p}$, i.e., the differential on the reduced bar complex of the $L_\infty$ algebra $\ell^k_{\epsilon}$ on $V[-1]$; see \cite[\S 2.1]{MZ22}.

\begin{definition}[{\cite[Definition 2.21]{MZ22}}]\label{def:order}
		Given a $BL_{\infty}$ augmentation and a pointed morphism $p_{\bullet}$, the \textbf{$(\epsilon,p_{\bullet})$-order} of $V$ is defined to be
		$$O(V,\epsilon,p_{\bullet}):=\min \left\{k\left|1\in \Ima \widehat{\ell}_{\bullet,\epsilon}|_{H_*(\overline{B}^k V, \widehat{\ell}_{\epsilon})}\right. \right\},$$
		where the minimum of an empty set is defined to be $\infty$.
\end{definition}

To discuss the functoriality of the order above, we need to introduce the notion of morphisms between pointed maps. Given a $BL_\infty$ morphism $\phi:(V,p) \to (V',q)$ and a family of morphisms $\phi^{k,l}_{\bullet}:S^k V \to S^l V^\prime$, we can define $\widehat{\phi}_{\bullet}:EV\to EV'$ by the same rule as $\widehat{\phi}$ with exactly one $\phi^{k,l}_{\bullet}$ component and all the others being $\phi^{k,l}$ components.

\begin{definition}[{\cite[Definition 2.22]{MZ22}}]\label{def:compatible}
Assume $\{p^{k,l}_{\bullet}: S^kV \rightarrow S^lV\}_{k,l}, \{q^{k,l}_{\bullet}:S^kV'\rightarrow S^lV'\}_{k,l}$ are two pointed maps on $(V,p),(V',q)$ respectively, of the same degree. We say $p_{\bullet},q_{\bullet},\phi$ are compatible if there is a family of $\phi^{k,l}_{\bullet}$ such that $\widehat{q}_\bullet\circ \widehat{\phi}-(-1)^{|\widehat{q}_{\bullet}|}\widehat{\phi}\circ \widehat{p}_{\bullet}=\widehat{q}\circ \widehat{\phi}_{\bullet}-(-1)^{|\widehat{\phi}_{\bullet}|}\widehat{\phi}_{\bullet}\circ \widehat{p}$ and $|\widehat{\phi}_{\bullet}|=|\widehat{p}_{\bullet}|+1$.
\end{definition}

The following proposition is the functoriality of the $(\epsilon,p_{\bullet})$-order.
\begin{proposition}[{\cite[Proposition 2.24]{MZ22}}]\label{prop:order}
	Assume $\phi$ is a $BL_{\infty}$-morphism from $(V,p)$ to $(V',q)$ with pointed maps $p_{\bullet},q_{\bullet}$ of degree $0$ respectively, such that $p_{\bullet},q_{\bullet},\phi$ are compatible. Then for any $BL_{\infty}$ augmentation $\epsilon$ of $V'$, we have $O(V,\epsilon\circ \phi,p_{\bullet})\ge O(V',\epsilon,q_{\bullet})$.
\end{proposition}

\begin{remark}
    In view of the discussion above, from the perspective of $(\epsilon,p_{\bullet})$-order, we should view $(p,p_{\bullet})$ as an object, i.e., $p_{\bullet}$ is an additional structure on the $BL_\infty$ algebra and $(\phi,\phi_{\bullet})$ as a morphism. This is the reason why we call $p_{\bullet}$ a pointed map instead of a morphism, besides the different combinatorial nature compared to a $BL_\infty$ morphism.
\end{remark}

\subsection{Completions and Twisted Coefficients}
\begin{definition}
    We define the Novikov ring $\Lambda_0$ over a field $\bk$ as
    $$\Lambda_0=\left\{ \sum_{i=1}^\infty a_i T^{\lambda_i}\left| a_i\in \bk, \lambda_i\ge 0, \displaystyle \lim_{i\rightarrow \infty} \lambda_i = +\infty \right.\right\},$$
    and the Novikov field $\Lambda$ as
    $$\Lambda=\left\{ \sum_{i=1}^\infty a_i T^{\lambda_i}\left| a_i\in \bk, \displaystyle \lim_{i\to \infty} \lambda_i = +\infty \right.\right\}.$$
    Let $\omega:G\to \R $ be a group homomorphism. The Novikov completion $\overline{\bk[G]}$ of the group ring $\bk[G]$ is
    $$\overline{\bk[G]}=\left\{ \sum_{i=1}^\infty a_i T^{g_i}\left| a_i\in \bk, \displaystyle \lim_{i\to \infty} \omega(g_i) = +\infty \right.\right\}.$$
    These coefficient rings are equipped with a decreasing filtration and are complete with respect to the filtration.
\end{definition}
\begin{remark}
    For applications in this paper, we will not use $\overline{\bm{k}[G]}$. In the context of SFT, this coefficient is needed for stable Hamiltonian structures and weak fillings, where $G$ is the second homology group. See \cite[p.624]{EGH} for examples.
\end{remark}

Now let $V$ be a $\bk$-vector space with filtration degree $0$; then $V\otimes_{\bk} \Lambda_0$ and $V\otimes_{\bk} \Lambda$ have induced decreasing filtrations, and so do the base changes for $SV$ and $EV$. We use $\overline{V\otimes_{\bk} \Lambda_0}$, $\overline{SV\otimes_{\bk} \Lambda_0}$, etc., to denote the corresponding completions. For example,
$$\overline{V\otimes_{\bk}\Lambda_0}=\left\{ \sum_{i=1}^{\infty} v_iT^{\lambda_i}\left|v_i\in V,\lambda_i\ge 0, \lim_{i\to \infty} \lambda_i=+\infty \right. \right\}.$$
The filtration on $\overline{V\otimes_{\bm{k}}\Lambda_0}$ is given by
$$(\overline{V\otimes_{\bk}\Lambda_0})_{\rho}=\left\{ \sum_{i=1}^{\infty} v_iT^{\lambda_i}\left|v_i\in V,\lambda_i\ge \max\{\rho,0\}, \lim_{i\to \infty} \lambda_i=+\infty \right. \right\},$$
so that
$$(\overline{V\otimes_{\bk}\Lambda_0})_{\rho} \subset (\overline{V\otimes_{\bk}\Lambda_0})_{\rho'}, \text{ when } \rho\ge \rho'.$$

\begin{definition}\label{def:complete}
A filtered completed $BL_\infty$ algebra structure on $V$ over $R=\Lambda_0,\Lambda$ consists of maps $p^{k,l}:\overline{S^kV\otimes_{\bk} R} \to \overline{S^lV\otimes_{\bk} R}$
for $k\ge 1, l\ge 0$, such that the assembled map $\widehat{p}$ is well-defined on $\overline{EV\otimes_{\bk} R}$, has degree $1$, squares to zero, and preserves the filtration. The filtered completed versions for morphisms, augmentations, pointed maps, and compatible pointed morphisms are similar.
\end{definition}

\begin{remark}
    In the context of SFT, $p^{k,l}$ is defined for both $R=\Lambda_0$ and $\Lambda$ for a fixed contact form and auxiliary choices from the virtual setup; see \S \ref{ss:RSFT}. When $R=\Lambda_0$, the $BL_\infty$ algebra is only expected to give invariants of the contact form, not the contact structure, as a trivial cobordism interpolating between two contact forms does not necessarily give an isomorphism over $\Lambda_0$. The theory over $\Lambda_0$ can be used for quantitative questions like embedding problems and symplectic capacities. On the other hand, when $R=\Lambda$, the $BL_\infty$ algebra is expected to give invariants of the contact structure.
\end{remark}

\subsection{Maurer-Cartan Elements}\label{ss:MC_Alg}
Although one can define Maurer-Cartan elements for uncompleted algebras, they typically arise in the completed theory in the context of SFT, where they are defined by counting holomorphic curves without positive punctures in a strong cobordism. That they are defined in the completed theory is a consequence of SFT compactness \cite{bourgeois2003compactness}, coupled with a perturbation scheme such as \cite{pardon2019contact} (which we use in this article).

\begin{definition}
    A Maurer-Cartan element of a filtered completed $BL_\infty$ algebra is an element $\mathfrak{mc}$ of degree $0$ with a positive filtration degree in $\overline{SV\otimes_{\bk}R}$, such that $\widehat{p}(e^\mathfrak{mc}-1)=0$, where $$e^{\mathfrak{mc}}=\sum_{i=0}^{\infty} \frac{\odot^i \mathfrak{mc}}{i!} \in \overline{SSV\otimes_{\bk}R}=\overline{R}\oplus \overline{EV\otimes_{\bk}R}.$$ It is clear that $e^\mathfrak{mc}-1$ is the projection of $e^{\mc}$ to $\overline{EV\otimes_{\bk}R}$.
\end{definition}
Given $a\in \overline{SV\otimes_{\bk}R}$ of degree $0$ with a positive filtration degree, we define
$$\exp_a:\overline{EV\otimes_{\bk}R}\to \overline{EV\otimes_{\bk}R}, \qquad x\mapsto x\odot e^a,$$
which preserves the filtration.

Note that for $x\in R$, we have $\widehat{p}(e^x-1)=0$. Therefore for any Maurer-Cartan element $\mathfrak{mc}$ and $x\in R$, we have $\mathfrak{mc}+x$ is also a Maurer-Cartan element. This is because $$\widehat{p}(e^{\mathfrak{mc}+x}-1) = \widehat{p}(e^{\mathfrak{mc}}\odot e^x -1) = \widehat{p}(e^{\mathfrak{mc}}-1)\odot e^x=0.$$
Note here that $\widehat{p}(a\odot b) = \widehat{p}(a)\odot b+(-1)^{|a|}a\odot \widehat{p}(b)$ does not hold in general; see Lemma \ref{lem:pmc}. In particular, we can always modify a Maurer-Cartan element to have no constant term, i.e., to lie in $\overline{\overline{S}V\otimes_{\bk}R}$. Such a Maurer-Cartan element is denoted by $\overline{\mathfrak{mc}}$. However, it is helpful to allow the definition to include a constant term, as a $BL_\infty$ morphism can map a Maurer-Cartan element with no constant term to a Maurer-Cartan element with a constant term; see \Cref{eqn:exp} below. Geometrically, this corresponds to gluing $J$-holomorphic curves with only negative punctures and $J$-holomorphic curves with only positive punctures (both can have multiple punctures), hence encoding information of Gromov-Witten invariants.

We can represent an element $x$ of degree $0$ of $\overline{SV\otimes_{\bk}R}$ by a sum of labeled trees, i.e., $x=\sum T_k$ for $T_k\in \overline{S^kV\otimes_{\bk}R}$
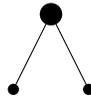
\begin{center}
\begin{tikzpicture}
		\node at (6,0) [circle,fill,inner sep=1.5pt] {};
		\node at (7,0) [circle,fill,inner sep=1.5pt] {};
		\draw (6,0) to (6.5,1) to (7,0);
  	\node at (6.5,1) [circle,fill,inner sep=3pt] {};
\end{tikzpicture}
\end{center}
Here we use a notation similar to a $BL_\infty$ morphism $\phi$, i.e., trees with $\tikz\draw[black,fill=black] (0,-1) circle (0.4em);$ representing the root, instead of trees with undotted roots in the description of elements in $SV$, since Maurer-Cartan elements behave like a ``morphism" from the trivial $BL_\infty$ algebra. We can represent the constant term of $\mathfrak{mc}$ by a tree without leaves, i.e., a single root. With such notation, $e^x$ in $\overline{EV\otimes_{\bk}R}$ is represented by linear combinations of forests generated by the labeled trees. As each component of $\mc$ has degree $0$, forests up to switching order are equivalent.
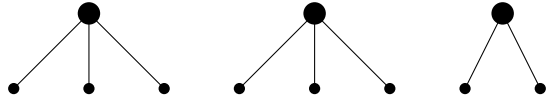
\begin{center}
\begin{tikzpicture}
		\node at (0,0) [circle,fill,inner sep=1.5pt] {};
		\node at (1,0) [circle,fill,inner sep=1.5pt] {};
		\node at (2,0) [circle,fill,inner sep=1.5pt] {};
		\node at (3,0) [circle,fill,inner sep=1.5pt] {};
		\node at (4,0) [circle,fill,inner sep=1.5pt] {};
		\node at (5,0) [circle,fill,inner sep=1.5pt] {};
		\node at (6,0) [circle,fill,inner sep=1.5pt] {};
		\node at (7,0) [circle,fill,inner sep=1.5pt] {};
    	\node at (6.5,1) [circle,fill,inner sep=3pt] {};
       \node at (1,1) [circle,fill,inner sep=3pt] {};
        \node at (4,1) [circle,fill,inner sep=3pt] {};
		
		\draw (0,0) to (1,1) to (1,0);
		\draw (1,1) to (2,0);
		\draw (3,0) to (4,1) to (4,0);
		\draw (4,1) to (5,0);
		\draw (6,0) to (6.5,1) to (7,0);
\end{tikzpicture}
\end{center}
Here the coefficient of a forest $\{\underbrace{T_{a_1},\ldots,T_{a_1}}_{i_1}, \ldots, \underbrace{T_{a_m},\ldots,T_{a_m}}_{i_m}\}$ is given by the ``isotropy order" $\frac{1}{(i_1)!\ldots (i_m)!}$. Given a $BL_\infty$ morphism $\phi$, then we can define $\phi(x)$ by gluing $\widehat{\phi}$ to the forest above (with the above coefficient) to get a tree, i.e.,
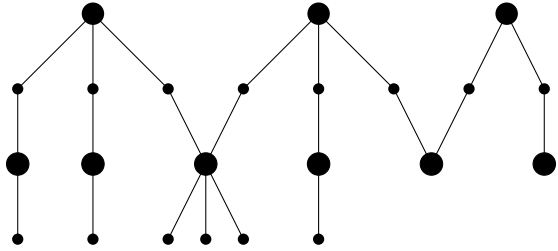
\begin{center}
		\begin{tikzpicture}
		\node at (0,0) [circle,fill,inner sep=1.5pt] {};
		\node at (1,0) [circle,fill,inner sep=1.5pt] {};
		\node at (2,0) [circle,fill,inner sep=1.5pt] {};
		\node at (3,0) [circle,fill,inner sep=1.5pt] {};
		\node at (4,0) [circle,fill,inner sep=1.5pt] {};
		\node at (5,0) [circle,fill,inner sep=1.5pt] {};
		\node at (6,0) [circle,fill,inner sep=1.5pt] {};
		\node at (7,0) [circle,fill,inner sep=1.5pt] {};
    	\node at (6.5,1) [circle,fill,inner sep=3pt] {};
       \node at (1,1) [circle,fill,inner sep=3pt] {};
        \node at (4,1) [circle,fill,inner sep=3pt] {};
		
		\draw (0,0) to (1,1) to (1,0);
		\draw (1,1) to (2,0);
		\draw (3,0) to (4,1) to (4,0);
		\draw (4,1) to (5,0);
		\draw (6,0) to (6.5,1) to (7,0);

		\draw (2,0) to (2.5,-1) to (3,0);
		\draw (2,-2) to (2.5,-1) to (2.5,-2);
		\draw (2.5,-1) to (3,-2);
		\node at (2.5,-1) [circle, fill, draw, outer sep=0pt, inner sep=3 pt] {};
		
		\node at (0,-2) [circle,fill,inner sep=1.5pt] {};
		\node at (1,-2) [circle,fill,inner sep=1.5pt] {};
		\node at (2,-2) [circle,fill,inner sep=1.5pt] {};
		\node at (3,-2) [circle,fill,inner sep=1.5pt] {};
		\node at (4,-2) [circle,fill,inner sep=1.5pt] {};
		\node at (2.5,-2) [circle,fill,inner sep=1.5pt] {};
	
		\draw (0,0) to (0,-2);
		\draw (1,0) to (1,-2);
		\draw (4,0) to (4,-2);
		\draw (5,0) to (5.5,-1);
		\draw (5.5,-1) to (6,0);
		\draw (7,0) to (7,-1);
		
		\node at (0,-1) [circle, fill, draw, outer sep=0pt, inner sep=3 pt] {};
		\node at (1,-1) [circle, fill, draw, outer sep=0pt, inner sep=3 pt] {};
		\node at (4,-1) [circle, fill, draw, outer sep=0pt, inner sep=3 pt] {};
		\node at (5.5,-1) [circle, fill, draw, outer sep=0pt, inner sep=3 pt] {};
		\node at (7,-1) [circle, fill, draw, outer sep=0pt, inner sep=3 pt] {};
	\end{tikzpicture}
\end{center}
\begin{lemma}\label{lem:phimc}
     For an element $x$ of degree $0$ with a positive filtration degree in $\overline{SV\otimes_{\bk}R}$ and $\phi$ a filtered completed $BL_\infty$ morphism, we have
     \begin{equation}\label{eqn:exp}
       e^{\phi(x)} = \widehat{\phi}(e^x).
      \end{equation}
\end{lemma}

In fact, the above holds for any $\{\phi^{k,l}\}$, not necessarily $BL_\infty$ morphisms, as long as $\widehat{\phi}$ is filtered and defined.

\begin{proof}
  It is clear that both sides of \Cref{eqn:exp} count the same type of forests; it suffices to verify the coefficients match. We can look at the coefficient from $T_k$. Assume the glued forest has $m_i$ identically labeled trees, each involving $n_i$ trees of $T_k$ for $1\le i\le p$. Then from the $e^{\phi(x)}$ side, the coefficient from $T_k$ is
  $$\prod_{i=1}^p \frac{1}{m_i!} \left(\frac{1}{n_i!}\right)^{m_i}.$$
  From the $\widehat{\phi}(e^x)$ side, the coefficient from $T_k$ is
  $$\frac{C}{\displaystyle\left(\sum_{i=1}^p m_in_i\right)!}$$
  where $C$ is the number of ways to write $(1,\ldots,\sum_{i=1}^p m_in_i)$ as $p$ \emph{ordered} subsets with $m_in_i$ elements for $1\le i \le p$, such that each subset is a union of $m_i$ subsets of $n_i$ elements. This number is the number of ways of grouping $T_k$ in $e^x$ so that the glued forest is the one above. Since $C$ is precisely
  $$\frac{\displaystyle\left(\sum_{i=1}^p m_in_i\right)!}{\displaystyle\prod_{i=1}^p m_i! (n_i!)^{m_i}}.$$
  The coefficient from $T_k$ then matches.
\end{proof}

As a consequence, given a Maurer-Cartan element $\mathfrak{mc}$, then $\phi(\mathfrak{mc})$ is again a Maurer-Cartan element. Since $\phi(c+x) = c+\phi(x)$ for $c\in R$, we have $\overline{\phi(\overline{\mathfrak{mc}})} = \overline{\phi(\mathfrak{mc})}$.

Given a Maurer-Cartan element $\mathfrak{mc}$, we have a deformed $BL_\infty$ structure $p^{k,l}_{\mathfrak{mc}}$ defined by trees from stacking the forests representing $e^{\mathfrak{mc}}-1$ over $\widehat{p}$, namely:
\begin{figure}[H]
\begin{center}
		\begin{tikzpicture}
		\node at (0,0) [circle,fill,inner sep=1.5pt] {};
		\node at (1,0) [circle,fill,inner sep=1.5pt] {};
		\node at (2,0) [circle,fill,inner sep=1.5pt] {};
		\node at (3,0) [circle,fill,inner sep=1.5pt] {};
		\node at (4,0) [circle,fill,inner sep=1.5pt] {};
		\node at (5,0) [circle,fill,inner sep=1.5pt] {};
       \node at (1,1) [circle,fill,inner sep=3pt] {};
        \node at (4,1) [circle,fill,inner sep=3pt] {};

		\draw (0,0) to (1,1) to (1,0);
		\draw (1,1) to (2,0);
		\draw (3,0) to (4,1) to (4,0);
		\draw (4,1) to (5,0);
        \draw (2.5,0) to (2.5,-1);
  
		\draw (2,0) to (2.5,-1) to (3,0);
		\draw (2,-2) to (2.5,-1) to (2.5,-2);
		\draw (2.5,-1) to (3,-2);
		\node at (2.5,-1) [circle, fill=white, draw, outer sep=0pt, inner sep=3 pt] {};
		
		\node at (0,-2) [circle,fill,inner sep=1.5pt] {};
	    \node at (1,-2) [circle,fill,inner sep=1.5pt] {};
	    \node at (2,-2) [circle,fill,inner sep=1.5pt] {};
	    \node at (3,-2) [circle,fill,inner sep=1.5pt] {};
	    \node at (4,-2) [circle,fill,inner sep=1.5pt] {};
	    \node at (5,-2) [circle,fill,inner sep=1.5pt] {};
	    \node at (2.5,-2) [circle,fill,inner sep=1.5pt] {};
	    \node at (2.5,0) [circle,fill,inner sep=1.5pt] {};
	    \draw[dashed] (0,0) to (0,-2);
		\draw[dashed] (1,0) to (1,-2);
		\draw[dashed] (4,0) to (4,-2);
		\draw[dashed] (5,0) to (5,-2);

		\end{tikzpicture}
\end{center}
    \caption{A component of $p^{1,7}_{\mathfrak{mc}}$}\label{fig:pmc}
\end{figure}
In a more formal way, we have
\begin{equation}\label{eqn:pmc}
    p_{\mathfrak{mc}}^{k,l}(v_1\ldots v_k) = \pi_{1,l}\circ\widehat{p}(v_1\odot \ldots \odot v_k\odot e^{\mathfrak{mc}})
\end{equation}
where $\pi_{1,l}$ is the completed version of the projection from $EV$ to $S^lV \subset SV \subset EV$. Since $\mathfrak{mc}$ has degree $0$, it does not matter where to put the Maurer-Cartan element in \Cref{fig:pmc}. The same construction holds for any $a\in \overline{SV\otimes_{\bk}R}$ of degree $0$, and we have the following.

\begin{lemma}\label{lem:pmc}
For $s\in \overline{EV\otimes_{\bk}R}$ and $a\in \overline{SV\otimes_{\bk}R}$ such that $\deg(a)=0$, we have $$\widehat{p}(s\odot e^a)=\widehat{p}_a(s)\odot e^a+(-1)^{|s|} s\odot \widehat{p}(e^a-1),$$
assuming that $s$ is of pure degree.
\end{lemma}
\begin{proof}
    The proof is again by examining that both sides correspond to the same linear combination of forests. The left-hand side can be interpreted as summing over the following forests (with coefficient from $e^a$) with the top level representing $s$.
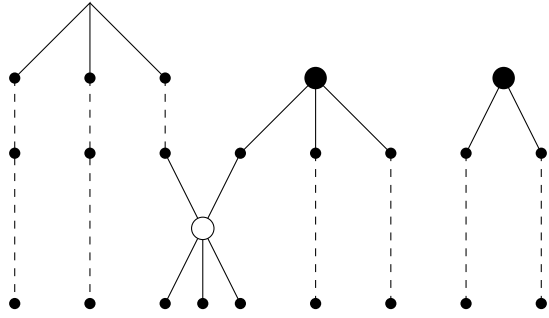
    \begin{center}
		\begin{tikzpicture}
		\node at (0,0) [circle,fill,inner sep=1.5pt] {};
		\node at (1,0) [circle,fill,inner sep=1.5pt] {};
		\node at (2,0) [circle,fill,inner sep=1.5pt] {};
		\node at (3,0) [circle,fill,inner sep=1.5pt] {};
		\node at (4,0) [circle,fill,inner sep=1.5pt] {};
		\node at (5,0) [circle,fill,inner sep=1.5pt] {};
		\node at (6,0) [circle,fill,inner sep=1.5pt] {};
		\node at (7,0) [circle,fill,inner sep=1.5pt] {};
        \node at (6.5,1) [circle,fill,inner sep=3pt] {};

        \node at (0,1) [circle,fill,inner sep=1.5pt] {};
		\node at (1,1) [circle,fill,inner sep=1.5pt] {};
		\node at (2,1) [circle,fill,inner sep=1.5pt] {};
		\node at (4,1) [circle,fill,inner sep=3pt] {};

		\draw (0,1) to (1,2) to (1,1);
		\draw (1,2) to (2,1);
		\draw (3,0) to (4,1) to (4,0);
		\draw (4,1) to (5,0);
		\draw (6,0) to (6.5,1) to (7,0);

		\draw (2,0) to (2.5,-1) to (3,0);
		\draw (2,-2) to (2.5,-1) to (2.5,-2);
		\draw (2.5,-1) to (3,-2);
		\node at (2.5,-1) [circle, fill=white, draw, outer sep=0pt, inner sep=3 pt] {};
		
		\node at (0,-2) [circle,fill,inner sep=1.5pt] {};
	    \node at (1,-2) [circle,fill,inner sep=1.5pt] {};
	    \node at (2,-2) [circle,fill,inner sep=1.5pt] {};
	    \node at (3,-2) [circle,fill,inner sep=1.5pt] {};
	    \node at (4,-2) [circle,fill,inner sep=1.5pt] {};
	    \node at (5,-2) [circle,fill,inner sep=1.5pt] {};
	    \node at (6,-2) [circle,fill,inner sep=1.5pt] {};
	    \node at (7,-2) [circle,fill,inner sep=1.5pt] {};
	    \node at (2.5,-2) [circle,fill,inner sep=1.5pt] {};
	    
	    \draw[dashed] (0,1) to (0,-2);
		\draw[dashed] (1,1) to (1,-2);
  	\draw[dashed] (2,1) to (2,0);
		\draw[dashed] (4,0) to (4,-2);
		\draw[dashed] (5,0) to (5,-2);
		\draw[dashed] (6,0) to (6,-2);
		\draw[dashed] (7,0) to (7,-2);
   
		\end{tikzpicture}
	\end{center}
    While the same forest can be interpreted in another way. Namely, when the $p^{k,l}$ component is glued to the forest representing $s$, then it corresponds to a component of $\widehat{p}_a(s)\odot e^a$. And when the $p^{k,l}$ component is not glued to $s$, it corresponds to a component of $(-1)^{|s|}s\odot \widehat{p}(e^a-1)$.
    The coefficients match by the same argument as \Cref{lem:phimc}. The extra sign of the last term comes from the fact that $|\widehat{p}|=1$.
\end{proof}

Then $p^{k,l}_{\mathfrak{mc}}$ is the structural map under the change of coordinates by $\exp_{\mc}$. More precisely, we have the following.
\begin{proposition}
    For a Maurer-Cartan element $\mathfrak{mc}$ and a filtered completed $BL_\infty$ algebra $p^{k,l}$, we have $p^{k,l}_{\mathfrak{mc}}$ forms a filtered completed $BL_\infty$ algebra. Moreover, the deformed structure does not depend on the constant term of $\mathfrak{mc}$.
\end{proposition}
\begin{proof}
By \Cref{lem:pmc},
$p^{k,l}_{\mathfrak{mc}}(v_1,\ldots,v_k)= \pi_{1,l} \circ \exp_{-\mathfrak{mc}} \circ \widehat{p} \circ \exp_{\mathfrak{mc}}(v_1\odot \ldots \odot v_k)$ and $\widehat{p}_{\mc} = \exp_{-\mc}\circ \widehat{p}\circ \exp_{\mc}$.
Therefore $\widehat{p}_{\mathfrak{mc}}\circ \widehat{p}_{\mathfrak{mc}}(s)=\exp_{-\mathfrak{mc}}\circ\widehat{p}\circ\widehat{p}\circ \exp_{\mc}(s)=0$. That $p^{k,l}_{\mc}$ does not depend on the constant term of $\mc$ follows easily from definitions.
\end{proof}

Another instant corollary of \Cref{lem:pmc} is the following.
\begin{proposition}\label{prop:MC_r}
Let $\mathfrak{mc}$ be a Maurer-Cartan element; then $\exp_{\mc}$ defines a chain map from $(\overline{EV\otimes_{\bk}R}, \widehat{p}_{\mathfrak{mc}})$ to $(\overline{EV\otimes_{\bk}R}, \widehat{p})$.
 \end{proposition}
However, since $\exp_{\mc}$ does not preserve the sentence length filtration $\overline{E^kV\otimes_{\bk}R}\subset \overline{EV\otimes_{\bk}R}$, we cannot guarantee that the torsion of $p$ is smaller than that of $p_{\mathfrak{mc}}$.

\begin{definition}
    Let $p_{\bullet}$ be a filtered completed pointed map on a completed $BL_\infty$ algebra $(V,p)$ and $\mathfrak{mc}$ a Maurer-Cartan element. We say $\mathfrak{mc}_{\bullet}\in \overline{SV\otimes_{\bk}R}$ of degree $|p_{\bullet}|+1$ with a positive filtration degree is a pointed Maurer-Cartan element if $\widehat{p}_{\bullet}(e^{\mathfrak{mc}}-1) =\widehat{p}_{\mc}(\mc_{\bullet})$.
\end{definition}
We want to use $\exp_{\mc}$ as a change of coordinates to deform $\widehat{p}_{\bullet}$. The analog of \Cref{eqn:pmc} holds for $p_{\bullet}$ as well. However, $\widehat{p}_{\bullet}(e^{\mc}-1)$ may not vanish. We will explain below how a pointed Maurer-Cartan element can be used to correct this error in the deformed pointed map. In the context of SFT, when the Maurer-Cartan element arises from a strong cobordism $W$ and the pointed map is defined by counting holomorphic curves subject to a constraint in a closed chain $C$ in the negative boundary $\partial_-W$, the pointed Maurer-Cartan element is defined by counting holomorphic curves without positive puncture in the cobordism subject to a chain $D$ in $(W,\partial W)$ such that the boundary of $D$ in $\partial_-W$ is $C$; see \S \ref{ss:MC} for details.

\begin{proposition}\label{prop:MC}
    Let $\phi$ be a filtered completed $BL_\infty$ morphism from $(V,p)$ to $(V',q)$ and $\mathfrak{mc}$ a Maurer-Cartan element of $V$.
    \begin{enumerate}
        \item\label{m1} We have a deformed filtered completed $BL_\infty$ morphism $\phi_{\mathfrak{mc}}$ from $(V,p_{\mathfrak{mc}} = p_{\overline{\mathfrak{mc}}})$ to $(V',q_{\phi\circ \mathfrak{mc}}=q_{\overline{\phi\circ\mathfrak{mc}}})$.
        \item\label{m2} Given a filtered completed pointed map $p_{\bullet}$ for $V$ and a pointed Maurer-Cartan element $\mathfrak{mc}_{\bullet}$, we have a deformed filtered pointed map $p_{\bullet,\mathfrak{mc},\mathfrak{mc}_{\bullet}}$ for $p_{\mathfrak{mc}}$.
        \item\label{m3} If, in addition, we have a pointed map $q_{\bullet}$ for $V'$ compatible with $\phi$, then $p_{\bullet,\mathfrak{mc},\mathfrak{mc}_\bullet},q_{\bullet,\phi(\mathfrak{mc}), \mathfrak{mc}'_{\bullet}},\phi_{\mathfrak{mc}}$ are compatible, where $\mathfrak{mc}'_{\bullet}$ (which depends on $\phi,\phi_{\bullet},\mathfrak{mc},\mathfrak{mc}_{\bullet}$ and is defined in the proof) is a pointed Maurer-Cartan element.
    \end{enumerate}
\end{proposition}
\begin{proof}
     $\phi_{\mathfrak{mc}}$ is defined similarly to $p_{\mathfrak{mc}}$, i.e., by the following glued tree:
     \begin{figure}[H]
     \begin{center}
		\begin{tikzpicture}
		\node at (0,0) [circle,fill,inner sep=1.5pt] {};
		\node at (1,0) [circle,fill,inner sep=1.5pt] {};
		\node at (2,0) [circle,fill,inner sep=1.5pt] {};
		\node at (3,0) [circle,fill,inner sep=1.5pt] {};
        \node at (2.5,0) [circle,fill,inner sep=1.5pt] {};
		\node at (4,0) [circle,fill,inner sep=1.5pt] {};
		\node at (5,0) [circle,fill,inner sep=1.5pt] {};
		\node at (6,0) [circle,fill,inner sep=1.5pt] {};
		\node at (7,0) [circle,fill,inner sep=1.5pt] {};
    	\node at (6.5,1) [circle,fill,inner sep=3pt] {};
       \node at (1,1) [circle,fill,inner sep=3pt] {};
        \node at (4,1) [circle,fill,inner sep=3pt] {};
		
		\draw (0,0) to (1,1) to (1,0);
		\draw (1,1) to (2,0);
        \draw (2.5,0) to (2.5,-1);
		\draw (3,0) to (4,1) to (4,0);
		\draw (4,1) to (5,0);
		\draw (6,0) to (6.5,1) to (7,0);

		\draw (2,0) to (2.5,-1) to (3,0);
		\draw (2,-2) to (2.5,-1) to (2.5,-2);
		\draw (2.5,-1) to (3,-2);
		\node at (2.5,-1) [circle, fill, draw, outer sep=0pt, inner sep=3 pt] {};
		
		\node at (0,-2) [circle,fill,inner sep=1.5pt] {};
		\node at (1,-2) [circle,fill,inner sep=1.5pt] {};
		\node at (2,-2) [circle,fill,inner sep=1.5pt] {};
		\node at (3,-2) [circle,fill,inner sep=1.5pt] {};
		\node at (4,-2) [circle,fill,inner sep=1.5pt] {};
		\node at (2.5,-2) [circle,fill,inner sep=1.5pt] {};
	
		\draw (0,0) to (0,-2);
		\draw (1,0) to (1,-2);
		\draw (4,0) to (4,-2);
		\draw (5,0) to (5.5,-1);
		\draw (5.5,-1) to (6,0);
		\draw (7,0) to (7,-1);
		
		\node at (0,-1) [circle, fill, draw, outer sep=0pt, inner sep=3 pt] {};
		\node at (1,-1) [circle, fill, draw, outer sep=0pt, inner sep=3 pt] {};
		\node at (4,-1) [circle, fill, draw, outer sep=0pt, inner sep=3 pt] {};
		\node at (5.5,-1) [circle, fill, draw, outer sep=0pt, inner sep=3 pt] {};
		\node at (7,-1) [circle, fill, draw, outer sep=0pt, inner sep=3 pt] {};
	\end{tikzpicture}
\end{center}
\caption{A component of $\phi^{1,6}_{\mathfrak{mc}}$}
\end{figure}
They do not depend on the constant terms of $\mathfrak{mc}$. The formal definition is similar to \eqref{eqn:pmc}, i.e.,
$$\phi^{k,l}(v_1\ldots v_k)=\pi_{1,l}\circ \widehat{\phi}\circ \exp_{\mc}(v_1\odot \ldots \odot v_k).$$
Similar to \Cref{lem:pmc}, we have
$$\widehat{\phi}\circ \exp_{\mc}=\exp_{\phi\circ \mc}\circ \widehat{\phi}_{\mc}.$$
Here, since the composition with $\widehat{\phi}$ glues every vertex of the input forest, we do not have the extra term in \Cref{eqn:pmc}. Therefore we have
$$\widehat{\phi}_{\mc}=\exp_{-\phi\circ \mc} \circ \widehat{\phi}\circ \exp_{\mc}.$$
Hence from $\widehat{\phi}\circ \widehat{p}=\widehat{q}\circ \widehat{\phi}$ we have
$$\widehat{\phi}_{\mc}\circ \widehat{p}_{\mc}=\widehat{q}_{\phi\circ \mc}\circ \phi_{\mc},$$
that is, \eqref{m1} holds.

The deformed version of $p_{\bullet}$ is defined by trees from gluing $e^{\mc}-1$ to $\widehat{p}_{\bullet}$ and gluing $\mathfrak{mc}_{\bullet}\odot e^{\mc}$ to $\widehat{p}$. In terms of formula,
$$p^{k,l}_{\bullet,\mc,\mc_{\bullet}}(v_1\ldots v_k):=  \pi_{1,l}\left(p_{\bullet}\circ \exp_{\mc}(v_1\odot\ldots \odot v_k) -\widehat{p}\circ \exp_{\mc}(\mc_{\bullet}\odot v_1\odot \ldots \odot v_k)\right).$$
Similar to the proof of \Cref{lem:pmc}, we can get
$$\exp_{\mc}\circ \widehat{p}_{\bullet,\mc,\mc_{\bullet}}(s)=\widehat{p}_{\bullet}\circ \exp_{\mc}(s)-\widehat{p}\circ \exp_{\mc}(\mc_{\bullet}\odot s)+(-1)^{|\mc_{\bullet}|}\mc_{\bullet}\odot \widehat{p}\circ \exp_{\mc}(s)$$
Here $\widehat{p}\circ \exp_{\mc}(\mc_{\bullet}\odot s)$ counts forests of gluing $\widehat{p}$ to $\exp_{\mc}(\mc_{\bullet}\odot s)$. However, in $p^{k,l}_{\bullet,\mc,\mc_{\bullet}}$, we only count those with the $p$-component glued to the $\mc_{\bullet}$ component by definition. The extra term $(-1)^{|\mc_{\bullet}|}\mc_{\bullet}\odot \widehat{p}\circ \exp_{\mc}(s)$ is to compensate for the over-count. Therefore we have
\begin{eqnarray*}
\widehat{p}_{\mc}\circ \widehat{p}_{\bullet,\mc,\mc_{\bullet}}(s) & = & \exp_{-\mc}\circ \widehat{p}\circ \widehat{p}_{\bullet}\circ \exp_{\mc}(s) + (-1)^{|\mc_{\bullet}|}\exp_{-\mc}\circ \widehat{p}(\mc_{\bullet}\odot \widehat{p}\circ \exp_{\mc}(s)) \\
& = & (-1)^{|p_{\bullet}|} \exp_{-\mc}\circ \widehat{p}_{\bullet}\circ \widehat{p}\circ \exp_{\mc}(s) +(-1)^{|\mc_{\bullet}|}\exp_{-\mc}\circ \widehat{p}\circ \exp_{\mc}(\mc_{\bullet}\odot \widehat{p}_{\mc}(s))
\end{eqnarray*}
and
\begin{eqnarray*}
\widehat{p}_{\bullet,\mc,\mc_{\bullet}}\circ \widehat{p}_{\mc}(s) & = & \exp_{-\mc}\circ \widehat{p_{\bullet}}\circ \widehat{p}\circ \exp_{\mc}(s) -\exp_{-\mc}\circ \widehat{p}\circ \exp_{\mc}(\mc_{\bullet}\odot \widehat{p}_{\mc}(s)) \\
\end{eqnarray*}
As $|\mc_{\bullet}|=|p_{\bullet}|+1$, we have that $p_{\bullet,\mc,\mc_{\bullet}}$ is a pointed map of degree $|p_{\bullet}|$ for $p_{\mc}$. This proves \eqref{m2}.

The deformed version of $\phi_{\bullet}$ is defined by trees from gluing $e^{\mc}-1$ to $\widehat{\phi}_{\bullet}$ and gluing $\mathfrak{mc}_{\bullet}\odot e^{\mathfrak{mc}}$ to $\widehat{\phi}$; the pointed Maurer-Cartan element $\mathfrak{mc}'_{\bullet}$ is defined by trees (without input vertices) from gluing $e^{\mathfrak{mc}}-1$ and $\widehat{\phi}_{\bullet}$ and gluing $\mathfrak{mc}_{\bullet}\odot e^{\mathfrak{mc}}$ to $\widehat{\phi}$. In terms of formula, we have
$$\mc'_{\bullet} = \pi_1 (\widehat{\phi}_{\bullet}(e^{\mc})+\widehat{\phi}(\mc_{\bullet}\odot e^{\mc}))$$
where $\pi_1$ is the projection from $\overline{EV\otimes_{\bk}R}$ to $\overline{SV\otimes_{\bk}R}$, and the deformed version of $\phi_{\bullet}$ is defined by
$$\phi_{\bullet,\mc,\mc_{\bullet}}^{k,l}(v_1\ldots v_k)=\pi_{1,l}\left(\widehat{\phi}_{\bullet}\circ \exp_{\mc}(v_1\odot \ldots \odot v_k)+\widehat{\phi}\circ \exp_{\mc}(\mc_{\bullet}\odot v_1\odot \ldots \odot v_k)\right).$$
The desired algebraic properties follow from a similar argument as before.
\end{proof}

\begin{remark}
In addition to the structures explained above, we also have a Maurer-Cartan deformation theory associated to $L_\infty$ algebras \cite{zbMATH07557488}, which is naturally related to the $BL_\infty$ theory as $\{p^{k,1}\}_{k\ge 1}$ is an $L_\infty$ algebra (up to some grading shift; see \cite[\S 2.1]{MZ22}) whenever $p^{k,0}=0$ for all $k>0$, e.g., the linearized structure maps ${p^{k,l}_{\epsilon}}$ satisfy such conditions. If $\{\mc_k\}$ is a Maurer-Cartan element for such a $BL_\infty$ algebra, $\mc_1$ is a Maurer-Cartan element for the associated $L_\infty$ algebra. Such a theory is completely analogous to the discussion here and is simpler. The specialization of $BL_\infty$ algebras to $L_\infty$ algebras (along with pointed maps) is needed to obtain functorial properties like \Cref{prop:order}. However, analogous functorial properties in the presence of Maurer-Cartan elements will not be very useful to obtain \Cref{thm:planarity_strong,cor:divsor,cor:BO}, where we use more specific geometric properties instead; see \S \ref{s4} for more details. We omit the discussion of Maurer-Cartan deformation of $L_\infty$ algebras and their relations with the $BL_\infty$ version for simplicity.
\end{remark}
\section{Rational symplectic field theory}\label{s3}
Unless specified otherwise, $(Y,\alpha)$ is a strict compact contact manifold such that the contact form $\alpha$ is non-degenerate throughout this section. By an exact strict cobordism from $(Y_-,\alpha_-)$ to $(Y_+,\alpha_+)$, we mean a Liouville domain $(X,\lambda)$ such that $(Y_-,\alpha_-)$ is the concave boundary $(\partial_-X, \lambda)$ and  $(Y_+,\alpha_+)$ is the convex boundary $(\partial_+X, \lambda)$ as strict contact manifolds. For example, an ellipsoid is a strict contact manifold 
$$\left(\left\{ (z_i)_{1\le i \le n}\in \C^n\left|\displaystyle \sum_{i=1}^n \frac{|z_i|^2}{\pi a_i}=1\right.\right\}, \displaystyle\frac{\mathbf{i}}{4}\sum_{i=1}^n (z_i\rd \overline{z}_i-\overline{z}_i\rd z_i) \right).$$
The contact form is non-degenerate if and only if the $a_i$ are linearly independent over $\Q$. The solid ellipsoid
$$\left(\left\{ (z_i)_{1\le i \le n}\in \C^n\left|\displaystyle \sum_{i=1}^n \frac{|z_i|^2}{\pi a_i}\le 1\right.\right\}, \displaystyle\frac{\mathbf{i}}{4}\sum_{i=1}^n (z_i\rd \overline{z}_i-\overline{z}_i\rd z_i) \right)$$
is a strict cobordism from $\emptyset$ to the ellipsoid above.
\subsection{RSFT of contact manifolds}\label{ss:RSFT}
\subsubsection{RSFT as a $BL_\infty$ algebra}
To assign a $BL_\infty$ algebra to a strict contact manifold $(Y,\alpha)$, let $V_{\alpha}$ denote the free $\Q$-module generated by formal variables $q_\gamma$ for each (good) orbit $\gamma$ of $(Y,\alpha)$ and $\Lambda$ the Novikov field over $\Q$. We take the grading $\vert q_\gamma\vert=\mu_{CZ}(\gamma)+n-3$. We define $p^{k,l}:S^kV_\alpha\otimes_{\Q} \Lambda \rightarrow S^lV_\alpha\otimes_{\Q} \Lambda$ by
\begin{equation}\label{eqn:p}
p^{k,l}(q^{\Gamma^+})=\sum_{A,|[\Gamma^-]|=l} \# \overline{\cM}_{Y,A}(\Gamma^+,\Gamma^-)\frac{T^{\int_A \rd \alpha }}{\mu_{\Gamma^-}\kappa_{\Gamma^-}}q^{\Gamma^-}.
\end{equation}
where $\overline{\cM}_{Y,A}(\Gamma^+,\Gamma^-)$ denotes the SFT compactification of the moduli space of rational holomorphic curves in the symplectization $(\mathbb R\times Y, d(e^t\alpha))$, in the relative homology class $A\in H_2(Y,\Gamma^+\cup\Gamma^-;\mathbb Z)$, which are asymptotic at $\pm\infty$ to the sets of orbits $\Gamma^\pm$ (w.r.t.\ fixed parameterizations of Reeb orbits and asymptotic markers at punctures\footnote{This is the reason why we have $\frac{1}{\kappa_{\Gamma^{-}}}$ in \Cref{eqn:p}.}, the same applies to all moduli spaces below), modulo automorphisms of the domain, and $\R$ translation, see \cite[\S 3.2]{MZ22} for details. Here, the sum is over all classes $A$ and all multisets $[\Gamma]$, i.e.\ sets with duplicates, of size $l$, such that the expected dimension of $\overline{\cM}_{Y,A}(\Gamma^+,\Gamma^-)$  is zero. $\Gamma$ is an ordered representation of $[\Gamma]$, e.g.\ $\Gamma=\{\underbrace{\eta_1,\ldots,\eta_1}_{i_1}, \ldots, \underbrace{\eta_m,\ldots,\eta_m}_{i_m}\}$ is an ordered set of good orbits with $\eta_i\ne \eta_j$ for $i\ne j$ and $\sum i_j = l$. We write $\mu_{\Gamma}=i_1!\ldots i_m!$ and $\kappa_{\Gamma}=\kappa^{i_1}_{\eta_1}\ldots \kappa^{i_m}_{\eta_m}$ with $\kappa_{\eta_j}$ the multiplicity of $\eta_j$, and $q^{\Gamma}=q_{\eta_1} \ldots  q_{\eta_m}$. Using a suitable virtual machinery, the boundary configurations of the $1$-dimensional moduli spaces $\overline{\cM}_{Y,A}(\Gamma^+,\Gamma^-)$ make $p^{k,l}$ into a $BL_\infty$ algebra. 

\begin{definition} More precisely, there are three versions of $BL_\infty$ algebra determined by counting rational SFT curves:
\begin{enumerate}
    \item By setting $T=1$, $p^{k,l}$ defines a $BL_\infty$ algebra on the $\Q$-vector space $V_{\alpha}$, which was considered in \cite[\S 3]{MZ22}. We use  $\RSFT(Y)$ to denote this $BL_\infty$ algebra for simplicity.
    \item $p^{k,l}$ in \Cref{eqn:p} defines a $BL_\infty$ algebra on the $\Lambda$-vector space $V_{\alpha}\otimes_{\Q}\Lambda$, which we referred to as the elementary  version over $\Lambda$. We use $\RSFT_{\Lambda}(Y)$ to denote this $BL_\infty$ algebra.  
    \item Finally, $p^{k,l}$ defines a filtered completed $BL_\infty$ algebra as in \Cref{def:complete}. We use $\overline{\RSFT}(Y)$ to denote this $BL_\infty$ algebra.
\end{enumerate}
To be more precise, these $BL_\infty$ algebras depend on the choice of contact form $\alpha$ as well as auxiliary choices in the virtual machinery. However, the homotopy type thereof, with a proper definition, is expected to be independent of those extra choices. 
\end{definition}

Note that if $u\in \overline{\cM}_{Y,A}(\Gamma_+,\Gamma_-)$, then we have
$$0\leq \int u^*d\alpha = \sum_{\gamma\in \Gamma^+} \int\gamma^*\alpha - \sum_{\gamma \in \Gamma^-} \int\gamma^*\alpha.$$
As a consequence, \Cref{eqn:p} is a finite sum, hence $\overline{\RSFT}(Y)$ is defined. At the $BL_\infty$ algebra level, these three versions carry the same amount of information, as $\RSFT_{\Lambda}$ is a restriction of $\overline{\RSFT}$,  and by setting $T=1$ in $\RSFT_{\Lambda}$, we can get $\RSFT$. On the other hand, since the $\rd\alpha$ energy of a curve is determined by its asymptotics, we can recover $\RSFT_{\Lambda}$ and $\overline{\RSFT}$ from $\RSFT$ and periods of Reeb orbits. However, working with these different coefficients is necessary; for example, they may admit inequivalent sets of augmentations.

\subsubsection{Functoriality} Given a strict exact cobordism $X$ from $Y_-$ to $Y_+$, we define a $BL_{\infty}$ morphism from $\RSFT(Y_+)$ to $\RSFT(Y_-)$ (similarly for the other two versions as well) by (setting $T=1$)
\begin{equation}\label{eqn:phi}
\phi^{k,l}(q^{\Gamma^+})=\sum_{A,|[\Gamma^-]|=l}\#\overline{\cM}_{X,A}(\Gamma^+,\Gamma^-)\frac{T^{\int_A \rd \lambda}}{\mu_{\Gamma^-}\kappa_{\Gamma^-}}q^{\Gamma^-},
\end{equation}
where $|\Gamma^+|=k$. Here $\overline{\cM}_{X,A}(\Gamma^+,\Gamma^-)$ is the SFT compactification of the moduli space of rational holomorphic curves in the completion of $X$, with relative homology class $A\in H_2(X,\Gamma^+\cup\Gamma^-;\mathbb Z)$, which are asymptotic at $\pm\infty$ to the sets of orbits $\Gamma^\pm$, modulo automorphisms of the domain. Then the boundary of the $1$-dimensional moduli spaces $\overline{\cM}_{X,A}(\Gamma^+,\Gamma^-)$ implies that $\{\phi^{k,l}\}_{T=1}$ is a $BL_\infty$ morphism $\RSFT(Y_+)\to \RSFT(Y_-)$.

\subsubsection{Pointed morphisms} If we fix a point $o$ in $Y$, consider the compactified moduli space $\overline{\cM}_{Y,A,o}(\Gamma^+,\Gamma^-)$ of rational holomorphic curves with one interior marked point in the symplectization modulo automorphisms, where the marked point is required to be mapped to $(0,o)\in \R \times Y$, in homology class $A$ and with asymptotics $\Gamma^\pm$. We then define a pointed map $p_{\bullet_o}$ by
\begin{equation}\label{eqn:p.}
p^{k,l}_{\bullet_o}(q^{\Gamma^+})=\sum_{A,|[\Gamma^-]|=l} \# \overline{\cM}_{Y,A,o}(\Gamma^+,\Gamma^-)\frac{T^{\int_A \rd \alpha }}{\mu_{\Gamma^-}\kappa_{\Gamma^-}}q^{\Gamma^-}.
\end{equation}
Then the boundary of the $1$-dimensional moduli spaces $\overline{\cM}_{Y,A,o}(\Gamma^+,\Gamma^-)$ implies that $\{p^{k,l}_{\bullet_o}\}$ is a pointed map of degree $0$ for all three versions of rational SFT. 

For a strict exact cobordism $X$ from $Y_-$ to $Y_+$, we choose a path $\gamma$ in $X$ from $o_-\in Y_-$ to $o_+\in Y_+$. We complete the path $\gamma$ to a proper path $\widehat{\gamma}$ in the Liouville completion $\widehat{X}$ by constant paths (in $Y_{\pm}$) in the cylindrical ends. Then the pointed morphisms $p_{\bullet_{o_+}},p_{\bullet_{o_-}}$ determined by $o_-,o_+$ and the $BL_\infty$ morphism $\phi$ are compatible, with the map $\phi_{\bullet}$ given by
\begin{equation}\label{eqn:phi.}
\phi^{k,l}_{\bullet}(q^{\Gamma^+})=\sum_{A,|[\Gamma^-]|=l}\#\overline{\cM}_{X,A,\gamma}(\Gamma^+,\Gamma^-)\frac{T^{\int_A \rd \lambda}}{\mu_{\Gamma^-}\kappa_{\Gamma^-}}q^{\Gamma^-}.
\end{equation}
Here $\overline{\cM}_{X,A,\gamma}(\Gamma^+,\Gamma^-)$ is the moduli space of rational curves with one interior marked point in the completed cobordism $\widehat{X}$ modulo automorphism, in homology class $A$ and with asymptotics $\Gamma^\pm$, where the marked point is required to go through $\widehat{\gamma}$.

\subsection{Realizing RSFT using Pardons's VFC \cite{pardon2016algebraic,pardon2019contact}}

\begin{theorem}[{\cite[Theorem 3.11]{MZ22}}]\label{thm:vfc}
    Let $(Y,\alpha)$ be a strict contact manifold with a non-degenerate contact form, then we have the following.
\begin{enumerate}
	\item\label{BL:1} There exists a non-empty set of auxiliary data $\Theta$\footnote{The content of the $\Theta$ consists of a compatible almost complex structure, a system of thickening data used to construct implicit atlases \cite[Definition 3.9]{pardon2016algebraic}, and a diagram as in \cite[Example 3.32]{MZ22}, which can be viewed as the perturbation data in Pardon's VFC.}, such that for each $\theta \in \Theta$ we have a $BL_\infty$ algebra $p_{\theta}$ on $V_{\alpha}\otimes_{\Q}\Lambda$, which induces a $BL_\infty$ algebra on $V_{\alpha}$ by setting $T=1$ and a filtered completed $BL_\infty$ algebra.
	\item\label{BL:2} For any point $o\in Y$, there exists a set of auxiliary data $\Theta_o$ with a surjective map $\Theta_o\to \Theta$, such that for any $\theta_o\in \Theta_o$, we have a pointed map $p_{\bullet,\theta_o}$ for $p_{\theta}$, where $\theta$ is the image of $\theta_o$ in $\Theta_o\to \Theta$.
	\item\label{BL:3} Assume there is a strict exact cobordism $X$ from $(Y_-,\alpha_-)$ to $(Y_+,\alpha_+)$. Let $\Theta_-,\Theta_+$ be the sets of auxiliary data for $\alpha_-,\alpha_+$, then there exists a set of auxiliary data $\Xi$ with a surjective map $\Xi\to \Theta_- \times \Theta_+$, such that for $\xi\in \Xi$, there is a $BL_{\infty}$ morphism $\phi_{\xi}$ from $(V_{\alpha_+}\otimes_{\Q}\Lambda,p_{\theta_+})$ to $(V_{\alpha_-}\otimes_{\Q}\Lambda,p_{\theta_-})$, where $(\theta_-,\theta_+)$ is the image of $\xi$ under $\Xi\to \Theta_- \times \Theta_+$.
	\item\label{BL:4} Assume in addition that we fix a point $o_-\in Y_-$ that is in the same component of $o_+\in Y_+$ in $X$. Then for any compatible auxiliary data $\theta_-,\theta_+,\theta_{o_-},\theta_{o_+},\xi$, we have that $p_{\bullet,\theta_{o_-}},p_{\bullet,\theta_{o_+}},\phi_{\xi}$ are compatible.
	\item\label{BL:5} For compatible auxiliary data $\theta,\theta_o$, there exist compatible auxiliary data $k\theta,k\theta_o$ for $(Y,k\alpha)$ for $k\in \R_+$, such that $p_{k\theta},p_{\bullet, k\theta_o}$ are identified with $p_{\theta},p_{\bullet,\theta_o}$ by the canonical $\Lambda$-linear identification between $V_{k\alpha}\otimes_{\Q}\Lambda$ and $V_{\alpha}\otimes_{\Q}\Lambda$ defined by $q_{\gamma}\mapsto T^{(k-1)\int \gamma^*\alpha}q_{\gamma}$.
    \item\label{BL:6} A strong filling of $Y$ gives rise to a $BL_\infty$ augmentation for the completed version over $\Lambda$. 
    More generally, a strong semi-filling, i.e.\ a strong filling with $Y$ as a boundary component, gives rise to a $BL_\infty$ augmentation for the completed version over $\Lambda$. 
\end{enumerate}
\end{theorem}

\eqref{BL:6} of \Cref{thm:vfc} is a consequence of \eqref{BL:3} of \Cref{thm:vfc}. In the case of a semi-filling, as holomorphic curves with positive punctures asymptotic to Reeb orbits on $Y$ cannot approach other boundaries of the semi-filling by the maximum principle, this gives rise to augmentations just as in the case of a strong filling. 

In addition to the filtration on $\Lambda$, which induces filtrations on $\RSFT_{\Lambda}$ and $\overline{\RSFT}_{\Lambda}$, there is another closely related filtration called the weight filtration. 
\begin{definition}\label{def:weight}
For each generator $q_{\gamma}$ of $V_{\alpha}$, we define the weight $w(q_{\gamma})$ to be $\int \gamma^*\alpha$. For the formal variable $T$ in the Novikov field, we define the weight $w(T)$ to be $1$. From these, we can define the weight on  $SV_{\alpha}\otimes_{\Q}\Lambda$, $\overline{SV_{\alpha}\otimes_{\Q}\Lambda}$, as well as $EV_{\alpha}\otimes_{\Q}\Lambda$ and $\overline{EV\otimes_{\Q}\Lambda}$ by the following properties:
\begin{enumerate}
    \item $w(u\ast v)=w(u)+w(v), w(u\odot v) = w(u)+w(v)$.
    \item $w(\sum_{i=0}^\infty u_i) \ge \inf \{w(u_i)\}_{i\in \N}$, here the infinite sum is needed for completions, and we require for each $\rho$ there exists $N_\rho$ such that $u_i\in (\overline{SV_{\alpha}\otimes_{\Q}\Lambda})_{\rho}$ (or $(\overline{EV\otimes_{\Q}\Lambda})_{\rho}$ depending on the context) for $i>N_{\rho}$, hence the infinite sum converges in the completion.
\end{enumerate}
\end{definition}
To see that these two properties characterize the weight, the first property defines weights on monomials. Here a monomial in $SV_{\alpha}\otimes_{\Q}\Lambda$  or $\overline{SV_{\alpha}\otimes_{\Q}\Lambda}$ is $cT^aq^{\Gamma}$, where $\Gamma$ is a multiset of good Reeb orbits. A monomial in $EV_{\alpha}\otimes_{\Q}\Lambda$  or $\overline{EV_{\alpha}\otimes_{\Q}\Lambda}$ is $cT^aq^{\Gamma_1}\odot \ldots \odot q^{\Gamma_m}$, where $\Gamma_i$ is a multiset of good Reeb orbits. Then the second property implies that the weight of any linear combination of monomials is given by the minimal weight of the monomials. For completions, as $\int \gamma^*\alpha>0$ for each Reeb orbit $\gamma$, elements in $\overline{SV_{\alpha}\otimes_{\Q}\Lambda}$ or $\overline{EV_{\alpha}\otimes_{\Q}\Lambda}$ can be written as a countable sum of monomials whose weight goes to infinity. Therefore, the second property also characterizes the weight by the minimal weight of monomials in the series. The weight filtration on these spaces consists of elements of weight no less than a filtration degree. We note here that $\overline{SV_{\alpha}\otimes_{\Q}\Lambda}$ and  $\overline{EV_{\alpha}\otimes_{\Q}\Lambda}$ are \emph{not} complete with respect to the weight filtration. 

\begin{definition}\label{def:weight_map}
    We say an element is of pure weight if it is a linear combination of monomials of the same weight. Then any element can be decomposed into elements of pure weight (possibly an infinite sum). A linear map between these spaces has weight $w$ if it sends the linear space of elements of pure weight $c$ (here we consider $0$ having arbitrary weight) to the linear space of elements of pure weight $c+w$.
\end{definition}
It is clear from the definition that all of $p,p_{\bullet},\phi,\phi_{\bullet}$ have weight zero. This is because in an exact cobordism or symplectization, the $\omega$-energy (appearing in the power of the Novikov variable) is determined by periods of asymptotic Reeb orbits. This is no longer the case for a strong cobordism.  

Note that any element in $EV_{\alpha}\otimes_{\Q}\Lambda$ or $\overline{EV_{\alpha}\otimes_{\Q}\Lambda}$ with pure weight is contained in $\overline{B}^kB^m V_{\alpha}\otimes_{\Q}\Lambda'=(\oplus_{i=1}^k S^i(\oplus_{j=0}^m S^jV))\otimes_{\Q}\Lambda'$ for some $k,m$ with coefficients in $\Lambda':=\{\sum_{i=1}^N a_iT^{t_i}|N\in \N_+, a_i\in \Q \}\subset \Lambda$, as the period $\int \gamma^*\alpha$ of a Reeb orbit $\gamma$ is strictly positive and below any period there are only finitely many Reeb orbits. In other words, we have the following.
\begin{proposition}\label{prop:weight}
    An element of pure weight must be a finite sum of monomials.
\end{proposition}
Sometimes, $p$ being a weight $0$ map allows us to focus on pure weight elements, then \Cref{prop:weight} establishes finiteness. Such finiteness is needed in the proof of \Cref{thm:strong_torsion}. Before diving into Maurer-Cartan elements in RSFT, we first explain two basic facts using the weight.
\begin{proposition}\label{prop:APT}
    $\PT$ is the same for $\RSFT(Y),\RSFT_{\Lambda}(Y)$ and $\overline{\RSFT}_{\Lambda}(Y)$.
\end{proposition}
\begin{proof}
    By sending $q_{\gamma}$ to $T^{-\int \gamma^*\alpha}q_{\gamma}$, i.e., making it of weight zero, we get a chain map $E^kV_\alpha\to E^kV_\alpha\otimes_{\Q}\Lambda$. It is clear that we have a chain map  $E^kV_\alpha\otimes_{\Q}\Lambda\to \overline{E^kV_\alpha\otimes_{\Q}\Lambda}$. Therefore we have 
    $$ \PT \text{ of } \RSFT(Y)\ge \PT \text{ of } \RSFT_{\Lambda}(Y) \ge \PT \text{ of } \overline{\RSFT}_{\Lambda}(Y).$$
    Conversely, if $\PT$ of $\overline{\RSFT}_{\Lambda}(Y)$ is $k$, in particular, there is an element $x\in \overline{E^kV_{\alpha}\otimes_{\Q}\Lambda}$ such that $\widehat{p}(x)=1$. We pick out the weight zero part $x_0$ of $x$, which is a finite sum of monomials by \Cref{prop:weight}. Since $\widehat{p}$ has weight $0$, we have $\widehat{p}(x_0)=1$ on $\overline{\RSFT}_{\Lambda}(Y)$. Since $x_0$ is a finite sum of monomials, by setting $T=1$, we get $\widehat{p}(x_0|_{T=1})=1$ on $\RSFT(Y)$. Therefore $\PT$ of $\RSFT(Y)$ is at most $k$.
\end{proof}
\begin{proposition}\label{prop:aug}
    Augmentations of $\RSFT(Y)$ induce augmentations of $\RSFT_{\Lambda}(Y)$ and $\overline{\RSFT}_{\Lambda}(Y)$. Augmentations for $\overline{\RSFT}_{\Lambda}(Y)$ induce augmentations of $\RSFT_{\Lambda}(Y)$.
\end{proposition}
\begin{proof}
    Let $\epsilon$ be an augmentation of $\RSFT(Y)$, we define 
    $$\epsilon^k_{\Lambda}(q^{\Gamma}) = T^{\sum_{\gamma \in \Gamma} \int \gamma^*\alpha} \epsilon^k(q^{\Gamma})$$
    that is, we make $\epsilon_{\Lambda}$ weight zero. It is straightforward to check that $\epsilon_{\Lambda}$ is both an augmentation of  $\RSFT_{\Lambda}(Y)$ and $\overline{\RSFT}_{\Lambda}(Y)$. The second statement is clear, as $EV_{\alpha}\otimes_{\Q}\Lambda$ is a subcomplex of $\overline{EV_{\alpha}\otimes_{\Q}\Lambda}$. 
\end{proof}
The augmentations on $\RSFT_{\Lambda}(Y)$ and $\overline{\RSFT}_{\Lambda}(Y)$ induced from an augmentation of $\RSFT(Y)$ are of weight zero, i.e., they behave like an augmentation from an exact filling but using the elementary or completed Novikov coefficient. In the other direction, in general, we cannot specialize $\RSFT_{\Lambda}(Y)$  or $\overline{\RSFT}_{\Lambda}(Y)$  at $T=1$ due to convergence issues. Moreover, it is impossible to pick out a weight-zero part of an augmentation $\epsilon$ of $\RSFT_{\Lambda}(Y)$, as the weight zero part of $\widehat{\epsilon}$ (corresponding to disconnected curves) is assembled from weight zero parts of each $\epsilon^k$.

\subsection{Maurer-Cartan Elements}\label{ss:MC}
Let $(X,\omega)$ be a strict strong cobordism from $(Y_-,\alpha_-)$ to $(Y_+,\alpha_+)$. It induces a Maurer-Cartan element:
$$\mathfrak{mc}= \sum_{A,[\Gamma_-]} \# \overline{\cM}_{X,A}(\emptyset, \Gamma_-)\frac{T^{\int_A \overline{\omega}}}{\mu_{\Gamma_-}\kappa_{\Gamma_-}}q^{\Gamma^-}$$
where $\overline{\omega}$ defined on $\widehat{X}$ is $\omega$ on $X$ and $\rd \alpha_-,\rd\alpha_+$ on the cylindrical ends of $X$. Even though $\overline{\omega}$ is continuous but not smooth on $\widehat{X}$ (as $\overline{\omega}$ does not vary smoothly in the cylindrical direction along the boundary of $X$), it is easy to show that $\int_A \overline{\omega}$ depends only on the relative homology class by Stokes' theorem. If $u\in \overline{\cM}_{X,A}(\emptyset, \Gamma_-)$, we have $0<\int u^*\overline{\omega} = \int_A \overline{\omega}$, which is called the $\omega$-energy in \cite[\S 6.1]{bourgeois2003compactness}. Then the Hofer energy of a curve $u$ in $\widehat{X}$ \cite[(22), (23)]{bourgeois2003compactness} asymptotic to Reeb orbits on both sides is defined by  
$$E_H(u):=\int u^*\overline{\omega} + \sup_{\phi_{\pm}} \left( \int_{u^{-1}(\R_+\times Y_+)} (\phi_+\circ s_+) \rd s_+\wedge \alpha_+ +  \int_{u^{-1}(\R_-\times Y_-)} (\phi_-\circ s_-) \rd s_-\wedge \alpha_-\right)$$
where the supremum is taken over all pairs $(\phi_+,\phi_-)$ of smooth functions such that $\phi_{\pm}:\R_{\pm}\to \R_+$ and $\int_{\R_{\pm}}\phi_{\pm}=1$, and $s_{\pm}$ are cylindrical coordinates on the two ends of $\widehat{X}$. 

\begin{proposition}\label{prop:Hofer}
    Let $u\in \overline{\cM}_{X,A}(\emptyset, \Gamma_-)$. Then there exists a constant $C>0$ depending only on $X$ such that $E_H(u)\le C \int u^*\overline{\omega}$. 
\end{proposition}
\begin{proof}
   Since $u\in \overline{\cM}_{X,A}(\emptyset, \Gamma_-)$, the maximum principle for the $s_+$ coordinate implies that $u$ does not enter $\R_+\times Y_+$. Therefore, the term with the supremum only has the second term. Let $\psi_-(s)=\int_{-\infty}^s \phi_-$, then $\rd \psi_-\ge 0$ and $\psi_-(0)=1$ and $\psi_-(-\infty)=0$. There exists a collar $[0,\delta)_s\times Y_-\to X$ such that the symplectic form  $\omega$ is pulled back to $\rd(e^{s}\alpha_-)$. Since $\psi_-\ge 0$, we have 
   \begin{eqnarray*}
       \int_{u^{-1}(\R_-\times Y_-)} (\phi_-\circ s_-) \rd s_-\wedge \alpha_- & \le & \int_{u^{-1}(\R_-\times Y_-)} \rd(\psi_-(s_-)\alpha_-)\\
       & \le  & \int_{u^{-1}(\R_-\times Y_-)} \rd(\psi_-(s_-)\alpha_-)+ \int_{u^{-1}(X)} u^*\omega\\
       & = & \int u^*\omega_\psi
   \end{eqnarray*}
   where $\omega_{\psi}$ is a continuous $2$-form which is $\omega$ on $X$ and is $\rd(\psi_-(s_-)\alpha_-)$ on $\R_-\times Y_-$. Consider a smooth function $f$ defined on $(-\infty,\delta)$ such that $f'\ge 0$, $f=0$ on $\R_-$ and $f=e^s$ near $\delta$. Then we define a form $\omega_f$ on $\widehat{X}$ by $\omega$ on $X\backslash [0,\delta)_s\times Y_- $, and $\rd(f(s)\alpha_-)$ on $(-\infty,\delta)_s\times Y_-$; the definition on $\R_+\times Y_+$ will not be relevant. It is straightforward to check by Stokes' theorem that 
   $$\int u^*\omega_\psi = \int u^*\omega_f.$$
   As $\omega_f=\rd f\wedge \alpha_-+ f\wedge \rd \alpha_-$ and both terms are non-negative when pulled back by $u$, we have 
   $$u^*\omega_f \le Cu^*\omega \text{ on } u^{-1}(X),$$
   where $C=\max_{s\in [0,\delta]}\{ f'/e^s, f/e^s\}$. Consequently, we have 
\begin{eqnarray*}
       \int_{u^{-1}(\R_-\times Y_-)} (\phi_-\circ s_-) \rd s_-\wedge \alpha_- & \le & \int u^*\omega_\psi
        =   \int u^*\omega_f\\
       & = &  \int_{u^{-1}(X)}u^*\omega_f 
        =   (C+1)\int_{u^{-1}(X)}u^*\omega
   \end{eqnarray*}
Therefore we have 
$$E_H(u)\le (C+2)\int u^*\overline{\omega}$$
The constant $C$ depends only on $X$, i.e., the width of the collar around $Y_-$.
\end{proof}
Therefore \Cref{prop:Hofer} and the SFT compactness \cite{bourgeois2003compactness} imply that $\mathfrak{mc}$ is in the completion $\overline{SV_{\alpha_-}\otimes_{\Q}\Lambda}$ of a positive filtration degree, assuming the counting is taken care of by some virtual technique. We can also define a map $\phi^{k,l}:\overline{S^kV_{\alpha_+}\otimes_{\Q}\Lambda}$ to $\overline{S^lV_{\alpha_-}\otimes_{\Q}\Lambda}$ for $k\ge 1$ by 
$$\phi^{k,l}(q^{\Gamma^+})=\sum_{A,|[\Gamma^-]|=l}\#\overline{\cM}_{X,A}(\Gamma^+,\Gamma^-)\frac{T^{\int_A \overline{\omega}}}{\mu_{\Gamma^-}\kappa_{\Gamma^-}}q^{\Gamma^-}.$$
It is well-defined in the completion since the Hofer energy of $u\in \overline{\cM}_{X,A}(\Gamma^+,\Gamma^-)$ is bounded by $C\int u^*\overline{\omega}+\sum_{\gamma \in  \Gamma_+}\int \gamma^*\alpha_+$, where $C$ is a fixed constant depending only on $X$, and the extra $\sum_{\gamma \in  \Gamma_+}\int \gamma^*\alpha_+$ comes from the first term in the supremum term of the Hofer energy. The proof of this bound is similar to \Cref{prop:Hofer}.

Similarly, for a path $\gamma$ in $X$ with ends $o_-$ in $Y_-$ and $o_+$ in $Y_+$, we have 
\begin{equation}\label{eqn:MC_bullet}
    \mathfrak{mc}_{\bullet}= \sum_{A,[\Gamma_-]} \# \overline{\cM}_{X,A,\gamma}(\emptyset, \Gamma_-)\frac{T^{\int_A \overline{\omega}}}{\mu_{\Gamma_-}\kappa_{\Gamma_-}}q^{\Gamma^-}
\end{equation}
is a pointed Maurer-Cartan element for the pointed map induced from $o_-$. It is helpful to point out that a curve in $\overline{\cM}_{X,A,\gamma}(\emptyset, \Gamma_-)$ cannot pass through the ray in $Y_+\times \R_+$ from $o_+$ by the maximum principle. Then the boundary configuration of $\overline{\cM}_{X,A,\gamma}(\emptyset, \Gamma_-)$ implies the algebraic property for a pointed Maurer-Cartan element. Finally, $\phi_{\bullet}$ from $X$ is defined similarly. 

Since the Maurer-Cartan deformed theory has a description by trees, the following proposition follows from the same argument as \Cref{thm:vfc}.
\begin{proposition}\label{prop_MC_vfc}
Assume $X$ is a strict strong cobordism from $(Y_-,\alpha_-)$ to $(Y_+,\alpha_+)$. Then there exists a set of auxiliary data $\Xi$ and a surjective map $\Xi \to \Theta_-\times \Theta_+$ such that for each $\xi\in \Xi$, which is mapped to $(\theta_-,\theta_+)\in \Theta_-\times \Theta_+$, we have a Maurer-Cartan element $\mathfrak{mc}_{\xi}$ of $p_{\theta_-}$ and a filtered completed $BL_\infty$ morphism $\phi_{\xi}$ from $p_{\theta_+}$ to $p_{\theta_-,\mathfrak{mc}_{\xi}}$. Assume $o_+,o_-$ are two points in $Y_+,Y_-$ that are in the same component of $X$. Then $\phi_{\xi},p_{\bullet,\theta_{o_+}},p_{\bullet,\theta_{o_-},\mathfrak{mc}_{\xi},\mc_{\bullet}}$ are compatible with a pointed Maurer-Cartan element $\mathfrak{mc}_{\bullet}$.
\end{proposition}
\section{Invariants}\label{s4}
\subsection{Algebraic planar torsion} As explained in \cite{MZ22}, while setting up the theory requires the choice of contact forms and auxiliary data coming from virtual techniques, the following is independent of choices:

\begin{definition}[{\cite[Proposition 3.11]{MZ22}}]
For a contact manifold $(Y,\xi)$, we define its \textbf{algebraic planar torsion} $\PT(Y)$ to be the torsion of the $BL_{\infty}$ algebra $\RSFT(Y,\alpha,\theta)$, for any choice of contact form $\alpha$ and auxiliary choice $\theta$ in \Cref{thm:vfc}.
\end{definition}

By \Cref{prop:APT}, algebraic planar torsion is independent of the choice of coefficients, i.e.\ $\Q,\Lambda$ and the completed version over $\Lambda$. The finiteness of $\PT$ is an obstruction to $BL_\infty$ augmentations, hence implies no strong filling (by \eqref{BL:6} of \Cref{thm:vfc}) and $\PT$ is non-decreasing in an exact cobordism. We refer to \cite[\S 3.5]{MZ22} for properties that are enjoyed by this invariant as well as their proofs.

\begin{proof}[Proof of \Cref{thm:strong_torsion}]
By assumption, $\PT(Y_+)<+\infty$, so there exists an element $v\in E^kV_{\alpha_+}\otimes_{\Q}{\Lambda} \subset \overline{EV_{\alpha_+}\otimes_{\Q}\Lambda}$ of weight $0$ such that $\widehat{p}_+(v)=1$, as $\widehat{p}_+$ is a weight zero map.  By \Cref{prop:MC_r,prop_MC_vfc}, we have filtered (with respect to the filtration induced from $\Lambda$, not the weight filtration) chain maps
$$(\overline{EV_{\alpha_+}\otimes_{\Q}\Lambda},\widehat{p}_+)\stackrel{\widehat{\phi}}{\longrightarrow} (\overline{EV_{\alpha_-}\otimes_{\Q}\Lambda},\widehat{p}_{-,\mathfrak{mc}}) \stackrel{\exp_{\mc}}{\longrightarrow}(\overline{EV_{\alpha_-}\otimes_{\Q}\Lambda},\widehat{p}_-).$$
As a consequence, we have $\widehat{p}_-\circ \exp_{\mc} \circ \widehat{\phi} (v) = \exp_{\mc}(1)=1\odot e^{\mc}$ for $1\in SV\subset \overline{EV\otimes_{\bk}R}$. Let $v_0$ be the weight zero part of $\exp_{\mc} \circ \widehat{\phi} (v)$; then $\widehat{p}_-(v_0)=1$ as the weight zero part of $1\odot e^{\mc}$ is $1$ and $\widehat{p}_-$ is of weight zero. Then there must be $m\in \N$ such that $v_0\in E^mV_{\alpha_-}\otimes_{\Q}\Lambda$ by \Cref{prop:weight}, i.e.\ $\PT(Y_-)<\infty$.
\end{proof}

\begin{example}
It is expected that containing a bordered Legendrian open book (\emph{bLob}) \cite[Section 3]{niederkruger2011weak} would imply that the contact homology vanishes, i.e.\ $\PT=0$ (see e.g.\ \cite{Huang} in dimension $5$). Assuming this is established, the cobordism in \cite{niederkruger2011weak}, along with \Cref{thm:strong_torsion}, yields another proof that the algebraic planar torsion of contact manifolds with higher-dimensional Giroux torsion \cite{niederkruger2011weak} is finite. Such results were established by the first author in \cite{MorPhD} using a more direct approach with explicit upper bounds. The functorial approach using \Cref{thm:strong_torsion} can be refined in a similar way to \Cref{cor:divsor,cor:BO} to obtain upper bounds if we know more precisely about how finite algebraic torsion is formed for the convex boundary. 
\end{example}

\begin{remark}\label{rmk:LW}
    In \cite[Question 2]{LW}, Latschev and Wendl asked for the version of \Cref{thm:strong_torsion} for algebraic torsions \cite[Definition 1.1]{LW}. This follows from the same proof using the weight zero part of the deformed functoriality in the Maurer-Cartan formalism for $BV_\infty$ algebras, assuming \Cref{thm:vfc} and \Cref{prop_MC_vfc} hold for the $BV_{\infty}$ formalism of the full SFT in \cite{cieliebak2009role,LW}.
\end{remark}

\subsection{Planarity} If $\RSFT_{\Q}(Y)=(V_\alpha,p)$ has a $BL_\infty$ augmentation $\epsilon$ over $\Q$, then $\PT(Y)$ is $\infty$, but we may still define a meaningful invariant as follows. Given a point $o\in Y$, we consider the associated pointed morphism $p_{\bullet}$, and the corresponding order $O(V_{\alpha},\epsilon,p_{\bullet})$. Let $\Aug_{\Q}(V_\alpha)$ denote the set of $BL_\infty$ augmentations of $V_\alpha$ over $\Q$. Again, the following definitions are independent of all auxiliary choices \cite{MZ22}:

\begin{definition}[{\cite[Definition 3.21]{MZ22}}]\label{def:planarity}
	For a contact manifold $(Y,\xi)$, we define its \textbf{planarity} as $$\Pl(Y):=\max\left\{O(V_{\alpha},\epsilon,p_{\bullet})\left|\forall \epsilon \in \Aug_{\Q}(V_\alpha), o \in Y, \theta_o\in \Theta_o\right.\right\},$$
	where the maximum of an empty set is defined to be zero. Here, $\alpha$ is any choice of contact form.
\end{definition}
The basic property of $\Pl$ is that $\Pl(Y_+)\le \Pl(Y_-)$ if there is an exact cobordism from $Y_-$ to $Y_+$. We again refer to \cite[\S 3.5]{MZ22} for the properties concerning this invariant. Similarly, we can define $\Pl_{\Lambda}$ and $\overline{\Pl}_{\Lambda}$ using the elementary version over $\Lambda$ and the filtered completed version of the rational SFT, where in both cases we require the augmentation is filtered, i.e.\ preserves the filtration induced from the filtration on $\Lambda$.
\begin{proposition}
    We have $\Pl(Y)\le \overline{\Pl}_{\Lambda}(Y) \le \Pl_{\Lambda}(Y).$
\end{proposition}
\begin{proof}
     By \Cref{prop:aug}, augmentations of $\RSFT$ induce augmentations of $\RSFT_{\Lambda}$ and $\overline{\RSFT}_{\Lambda}$. Moreover, those induced augmentations are of weight zero; as a consequence, $\widehat{p}_{\epsilon}$ is also of weight zero. As $p_{\bullet}$ is of weight zero, by a similar argument to \Cref{prop:APT}, we see that $O(V_{\alpha},\epsilon,p_{\bullet})$ is the same for those three augmentations on $\RSFT$, $\RSFT_{\Lambda}$ and $\overline{\RSFT}_{\Lambda}$. As a consequence, we have $\Pl \le \Pl_{\Lambda},\overline{\Pl}_{\Lambda}$. Finally, by \Cref{prop:aug}, an augmentation $\overline{\epsilon}$ of $\overline{\RSFT}_{\Lambda}$ induces an augmentation $\epsilon$ on $\RSFT_{\Lambda}$, where $p_{\overline{\epsilon}}=p_{\epsilon}$ on monomials by construction.   As $(\overline{B}^kV_{\alpha},\widehat{\ell}_{\epsilon})$ is a subcomplex of $(\overline{B}^kV_{\alpha},\widehat{\ell}_{\overline{\epsilon}})$ such that the following is commutative,
     $$
     \xymatrix{
     (\overline{B}^kV_{\alpha},\widehat{\ell}_{\epsilon}) \ar[r]^{\subset} \ar[d]^{\widehat{\ell}_{\bullet,\epsilon}} & (\overline{B}^kV_{\alpha},\widehat{\ell}_{\overline{\epsilon}})\ar[d]^{\widehat{\ell}_{\bullet,\overline{\epsilon}}}\\
     \Lambda \ar[r]^{=} & \Lambda
     }
     $$
     we have $O(\epsilon,p_{\bullet})\ge O(\overline{\epsilon},p_{\bullet})$. As a consequence, we have $\overline{\Pl}_{\Lambda}\le \Pl_{\Lambda}$.
\end{proof}
It is not clear whether there are examples such that those inequalities become strict. All three of them enjoy the same functorial property with respect to exact cobordisms by the argument in \cite[\S 3.5]{MZ22}. Heuristically, $\Pl$ looks for rational curves passing through a fixed point in the contact manifold, typically those without negative punctures, and $\Pl$ counts the number of positive punctures.
\begin{proof}[Proof of \Cref{thm:cofilling}]
    Assume there is a connected strong filling $W$ of $Y\sqcup Y'$, which induces a $\Lambda$-augmentation $\epsilon$ of $\RSFT_{\Lambda}(Y)$. By taking a point $o\in Y$, we get a pointed map $p_{\bullet}$. If we consider curves in $\widehat{W}$ with positive punctures only asymptotic to Reeb orbits on $Y$ and with one marked point going through the completion of a ray $\gamma$ starting from $o$ and shooting out to infinity on the cylindrical end of $Y'$, we get a structural map $\epsilon_{\bullet}$, which is a special case of $\phi_{\bullet}$, i.e.\ $(\RSFT(Y),p,p_{\bullet}),(\{0\},0,0), \epsilon$ are compatible through $\epsilon_\bullet$. Even though $\gamma$ is non-compact, the moduli space is compactified as usual, as there is no curve with only positive punctures asymptotic to orbits on $Y$ going through the infinite end of $\gamma$ in $\R_+\times Y'$ by the maximum principle. Then by \cite[Proposition 2.20]{MZ22}, we have $O(\RSFT(Y),\epsilon,p_{\bullet})\ge  O(\{0\},0,0) = +\infty$, a contradiction! 
\end{proof}
\begin{remark}
    \Cref{thm:cofilling} holds for $\overline{P}_{\Lambda}<\infty$ as well by an identical proof.
\end{remark}
In the following, we consider planarity under strong cobordisms. Although the relevant algebraic structures for pointed maps deformed by Maurer-Cartan elements are discussed at the end of \S \ref{ss:MC_Alg}, unlike the situation in \Cref{thm:strong_torsion}, where $\widehat{p}$ has weight zero, $\widehat{\phi}$ from the strong cobordism $X$ and augmentations to $\RSFT_{\Lambda}(Y_-)$ do not have pure weight. It is hard to conclude that the output in $Y_-$ that contributes to the planarity has a finite length, and it is harder to conclude a universal upper bound of length independent of augmentations. Instead of using the assembled structured maps and their general properties considered in \S \ref{s2}, we will use the geometric picture behind \Cref{thm:vfc} by only looking at the boundary of certain moduli spaces of virtual dimension $1$ and additional knowledge of curves in the convex boundary explained in \Cref{prop:simple} below. 
\begin{proposition}[\Cref{thm:planarity_strong}]\label{prop:simple}
Let $X$ be a strong cobordism from $Y_-$ to $Y_+$. Assume there are auxiliary data $\alpha_+,\theta_+$ for $Y_+$ such that the following holds:
\begin{enumerate}
    \item\label{p1} There exist good Reeb orbits $\gamma_1,\ldots,\gamma_k$ of $\alpha_+$ such that for any subset $S\subset \{ 1,\ldots,k\}$, we have $p^{|S|,l}_+(\prod_{i\in S} q_{\gamma_i})=0$ for any $l\ge 0$.
    \item\label{p2} For any point $o$ with associated pointed map $p_{\bullet,+}$, we have $p^{k,0}_{\bullet,+}(\prod_{i=1}^k q_{\gamma_i}) \ne 0$ and $p^{|S|,l}_{\bullet,+}(\prod_{i\in S} q_{\gamma_i}) = 0$ for any $l>0$ and subset $S\subset \{ 1,\ldots,k\}$.
\end{enumerate}
These two conditions imply that $\Pl_{\Lambda}(Y_+)\le k$ (as the choice of augmentation does not affect the computation of planarity here). Then $\Pl_{\Lambda}(Y_-)<+\infty$.
\end{proposition}
As the planarity for $Y_+$ is solely contributed to by rational curves with $k$ positive punctures and no negative punctures in $\R\times Y_+$ from the assumption that $p^{k,0}_{\bullet,+}(\prod_{i=1}^k q_{\gamma_i})\ne 0$, the idea of the proof is simply viewing such curves as part of the boundary component of rational holomorphic curves with $k$ positive punctures and no negative punctures in $\widehat{X}$. Then the planarity of $Y_-$ will be observed by the other boundary components of those moduli spaces.

\begin{proof}[Proof of \Cref{prop:simple}]
Using the weight in \Cref{def:weight}, the operator determined by  $\overline{\cM}_{X,A}(\Gamma_+,\Gamma_-)$ will be of weight $\int_A \overline{\omega}+\sum_{\gamma\in \Gamma_-}\int \gamma^*\alpha_--\sum_{\gamma\in \Gamma_+} \int \gamma^*\alpha_+$ in the sense of \Cref{def:weight_map}, which is zero whenever $X$ is exact. In particular, the weight of such an operator is determined by the underlying topological data.

Let $\epsilon$ be an augmentation of $\RSFT_{\Lambda}(Y_-)$. We consider the element $v$ in $\overline{\overline{S}V_{\alpha_-}\otimes_{\Q} \Lambda}$ as follows:
\begin{itemize}
    \item It is obtained by gluing one level of $\phi^{k,l}$ and $\mc_k$ from $X$ to the forest representing $q_{\gamma_1}\odot\ldots\odot q_{\gamma_k}$ followed by a level of $\epsilon$.
    \item The glued forest has no cycles, and each tree in the forest has exactly one output leaf.
    \item The first level, i.e.\ the level of  $\phi^{k,l}$ and $\mc_k$, does not change the total weight, that is, the sum of the weight of each operator in the first level is zero. 
\end{itemize}
\begin{figure}[H]
 \begin{center}
		\begin{tikzpicture}
		\node at (0,0) [circle,fill,inner sep=1.5pt] {};
		\node at (1,0) [circle,fill,inner sep=1.5pt] {};
		\node at (2,0) [circle,fill,inner sep=1.5pt] {};
		\node at (3,0) [circle,fill,inner sep=1.5pt] {};
		\node at (4,0) [circle,fill,inner sep=1.5pt] {};

		\draw (2,0) to (2.5,-1) to (3,0);
		\draw (2,-2) to (2.5,-1) to (2.5,-2);
		\draw (2.5,-1) to (3,-2);
		\node at (2.5,-1) [circle, fill, draw, outer sep=0pt, inner sep=3 pt] {};
		
		\node at (0,-2) [circle,fill,inner sep=1.5pt] {};
		\node at (1,-2) [circle,fill,inner sep=1.5pt] {};
		\node at (2,-2) [circle,fill,inner sep=1.5pt] {};
		\node at (3,-2) [circle,fill,inner sep=1.5pt] {};
		\node at (4,-2) [circle,fill,inner sep=1.5pt] {};
		\node at (2.5,-2) [circle,fill,inner sep=1.5pt] {};
  	\node at (5,-2) [circle,fill,inner sep=1.5pt] {};
    	\node at (6,-2) [circle,fill,inner sep=1.5pt] {};
		\node at (7,-2) [circle,fill,inner sep=1.5pt] {};
		
         \draw (0,0) to (0,-2);
		\draw (1,0) to (1,-2);
		\draw (4,0) to (4,-2);
		\draw (5,-2) to (5.5,-1) to (6,-2);
        \draw (7,-2) to (7,-1);

		\node at (0,-1) [circle, fill, draw, outer sep=0pt, inner sep=3 pt] {};
		\node at (1,-1) [circle, fill, draw, outer sep=0pt, inner sep=3 pt] {};
		\node at (4,-1) [circle, fill, draw, outer sep=0pt, inner sep=3 pt] {};
		\node at (5.5,-1) [circle, fill, draw, outer sep=0pt, inner sep=3 pt] {};
		\node at (7,-1) [circle, fill, draw, outer sep=0pt, inner sep=3 pt] {};

        \draw (1,-2) to (1.5,-3) to (2,-2);
        \draw (3,-2) to (3,-3);
        \draw (4,-2) to (4.5,-3) to (5,-2);
        \node at (1.5,-3) [circle, fill, draw, outer sep=0pt, inner sep=3 pt] {};
		\node at (3,-3) [circle, fill, draw, outer sep=0pt, inner sep=3 pt] {};
		\node at (4.5,-3) [circle, fill, draw, outer sep=0pt, inner sep=3 pt] {};

        \node at (0,0.3) {$q_{\gamma_1}$};
        \node at (1,0.3) {$q_{\gamma_2}$};
        \node at (2,0.3) {$q_{\gamma_3}$};
        \node at (3,0.3) {$q_{\gamma_4}$};
        \node at (4,0.3) {$q_{\gamma_5}$};
        \node at (0.5,-1) {$\phi^{1,1}_{A_1}$};
        \node at (1.5,-1) {$\phi^{1,1}_{A_2}$};
        \node at (3,-1) {$\phi^{2,3}_{A_3}$};
        \node at (4.5,-1) {$\phi^{1,1}_{A_4}$};
        \node at (6.2,-1) {$\mathfrak{mc}_{2}^{A_5}$};
        \node at (7.7,-1) {$\mathfrak{mc}_{1}^{A_6}$};
        \node at (1.5,-3.3) {$\epsilon_2$};
        \node at (3,-3.3) {$\epsilon_1$};
        \node at (4.5,-3.3) {$\epsilon_{2}$};
	\end{tikzpicture}
\end{center}
\caption{A component of the output element $v$ in  $\overline{S^4V_{\alpha_-}\otimes_{\Q}\Lambda}$ when $k=5$. The decoration $A_*$ records the topological type of curves underlying the operation, i.e.\ the homology class and asymptotic. Hence, the weight of the operator is determined by $A_*$. Here, a component of $v'$ is the above figure without the last level of augmentation $\epsilon$.}\label{fig:planarity}
\end{figure}
Note that $\epsilon$ may not have pure weight, so $v$ is not necessarily of pure weight. Nevertheless, the output $v'$ before applying the augmentation $\epsilon$ has the same weight as $q_{\gamma_1}\cdots q_{\gamma_k}$. As the operators from the first level all increase the filtration degree from $\Lambda$, i.e.\ the powers of the Novikov variable $T$ appeared in those operators are positive, we know that $v'$ can have at most 
$$N:=\frac{\displaystyle \sum_{i=1}^k \int \gamma_i^*\alpha_+}{\displaystyle\min\left\{\int \gamma^*\alpha_-\left|\gamma \text{ is a Reeb orbit}\right.\right\}}<\infty$$
leaves, which is independent of $\epsilon$. Moreover, as there are only finitely many Reeb orbits on $Y_-$ with period smaller than any given threshold, $v'$ is composed of only finitely many Reeb orbits on $Y_-$. As $v$ must have fewer leaves than $v'$, we can conclude that $v\in \overline{S}V_{\alpha_-}\otimes_{\Q}\Lambda$ instead of the completion. 

We claim that $\widehat{\ell}_{\epsilon,-}(v)=0$. Note that $\widehat{\ell}_{\epsilon,-}(v)$ is computed by gluing a level of $\phi$ and $\mc$ of total weight zero to $q_{\gamma_1}\odot \ldots \odot q_{\gamma_k}$ and a level of $\widehat{p}_-$, then a level of $\epsilon$, such that the glued forest consists of trees of single leaf. As $\widehat{p}_-$ has weight zero, the first two levels correspond to part of the boundary of moduli spaces of one level of $\phi$ and $\mc$ of total weight zero, such that one of the components has virtual dimension $1$ and the rest components have virtual dimension $0$. The rest of the boundary components come from stacking $\widehat{p}_+$ over a level of $\phi$ and $\mc$; then condition \eqref{p1} implies that the virtual count of those boundary components is zero. We note here that even though \Cref{thm:vfc} only states that there exists a virtual count such that $p^{k,l}$ etc. assemble to desired structures, the proof of \Cref{thm:vfc} as well as those in \cite{pardon2019contact} imply more. Namely, the virtual count of the boundary of moduli spaces of virtual dimension $1$ is zero, and is equal to the sum of products of the virtual count of those moduli spaces of virtual dimension $0$ with suitable weights recording the isotropy (i.e.\ multiplicity and permutation, as the boundary is in fact an orbifold fiber product of moduli spaces) appearing in the degeneration, see e.g.\ \cite[(4.39), (4.30)]{pardon2019contact}. Therefore we have $\widehat{\ell}_{\epsilon,-}(v)=0$.

We claim that $\widehat{\ell}_{\bullet,\epsilon,-}(v)\ne 0$. Let $\gamma$ be a path in $X$ connecting a point in $Y_+$ and a point in $Y_-$. We consider the configuration of one level of $\phi,\mc$ and exactly one $\phi_{\bullet}$ or one $\mc_{\bullet}$ (defined in \Cref{eqn:MC_bullet} by counting holomorphic caps in $X$ with a point constraint on $\gamma$) glued to $q_{\gamma_1}\odot \ldots \odot q_{\gamma_k}$, then capped off by $\epsilon$, so that the resulting graph is connected and has no cycles, i.e.\ a tree. We also require that the first level has total weight zero, the virtual dimension of exactly one component is $1$, and virtual dimensions of the remaining components are $0$. Then it has the following boundary components (the input forest is always $q_{\gamma_1}\odot \ldots \odot q_{\gamma_k}$), which will contribute to the counting:
\begin{enumerate}[(a)]
    \item\label{a} Curves corresponding to $p_{\bullet,+}^{k,0}(q_{\gamma_1}\ldots q_{\gamma_k})$.
    \item\label{b} One level of $\widehat{p}_{\bullet,+}$, glued with one level of $\phi$ and $\mc$, with total weight zero, as $p_{\bullet,+}$ has weight zero, and finally a level of $\epsilon$, so that the glued graph is a tree. Note that the non-trivial part in the first level must be $p_{\bullet,+}^{k,l}$ for $l>0$ to glue to such a tree.
    \item\label{c} One level of $\widehat{p}$, glued with one level of $\phi,\mc$ and exactly one $\phi_{\bullet}$ or one $\mc_{\bullet}$ with total weight zero,  then capped off by $\epsilon$, such that the glued graph is a tree.
    \item\label{d} One level of $\phi,\mc$ of total weight zero, glued with one level of $p_{\bullet,-}$, then capped off by $\epsilon$, so that the glued graph is a tree.
    \item\label{e} One level of $\phi,\mc$ and exactly one $\phi_{\bullet}$ or one $\mc_{\bullet}$ with total weight zero, one level of $\widehat{p}_-$ and capped off by $\epsilon$, the glued graph is a tree.
\end{enumerate}
After picking auxiliary choices as in \Cref{thm:vfc}, boundary \eqref{a} counts non-trivially by condition \eqref{p2}; boundary \eqref{b} sums to zero by condition \eqref{p2}; boundary \eqref{c} sums to zero by condition \eqref{p1}; boundary \eqref{d} sums to $\widehat{\ell}_{\bullet,\epsilon,-}(v)$ by construction; boundary \eqref{e} sums to zero as $\epsilon$ is an augmentation to $p_-$. Therefore the claim follows. 

As a consequence, for any augmentation $\epsilon$, the upper bound of planarity of $Y_-$ (with respect to $\epsilon$) is observed by $v$ (depending on $\epsilon$). Hence $\overline{\Pl}_{\Lambda}(Y_-) \le N$ as $v\in \overline{B}^NV_{\alpha_-}\otimes_\Q\Lambda$ for any augmentation $\epsilon$.
\end{proof}

Contact manifolds in \Cref{cor:divsor,cor:BO} all admit strong cobordism to a subcritically fillable contact manifold $\partial(V\times \D)$. More precisely, in \Cref{cor:divsor}, by putting back the divisor $s^{-1}(0)$, we get a strong cobordism to the standard contact sphere; In \Cref{cor:BO}, by attaching a round handle in the sense of \cite[\S 5.1]{MNW}, we get a strong cobordism to  $\partial(\Sigma \times DT^*S^1\times \D)$. By \cite[Theorem L (3)]{MZ22}, we have $\Pl(\partial(V\times \D))=1$ such that conditions in \Cref{prop:simple} hold for $k=1$. However, to get bounds in \Cref{cor:divsor,cor:BO} using the proof of \Cref{prop:simple}, we need to get effective estimates of the period of Reeb orbits, which is not an easy task. On the other hand, those strong cobordisms to $\partial(V\times \D)$ are exact away from a symplectic hypersurface. As a consequence, curves contributing to the Maurer-Cartan element have a positive intersection with the symplectic hypersurface by positivity of intersection. We will run a similar argument as in \Cref{prop:simple} but use this intersection information to get the bounds in \Cref{cor:divsor,cor:BO}.

\begin{proof}[Proof of \Cref{cor:divsor}]
By \cite[Lemma 7.1]{MZ22}, there is an exact cobordism from the ideal boundary of $\CP^n\backslash (s^{-1}(0)\cup H)$ to the ideal boundary $\CP^n\backslash(H_d\cup H)$, where $H_d$ is a smooth degree $d$ hypersurface in $\CP^n$ and $H_d$ is transversal to $H$. Therefore, it suffices to prove the theorem for $s^{-1}(0)=H_d$. Let $Y_d$ denote the ideal boundary of $\CP^n\backslash (H_d\cup H)$. More precisely, following \cite[4 (b)]{Seidel_biased}, $Y_d$, as a contact manifold, is defined as follows. We equip $\cO(d+1)$ with a hermitian metric such that the curvature is $-2(d+1)\mathbf{i}\omega_{FS}$, where $\omega_{FS}$ is the Fubini-Study form on $\CP^n$. Let $\eta$ be a holomorphic section of $\cO(d+1)$ such that $\eta^{-1}(0)=H_d\cup H$, then $h=-\log|\eta|$ is an exhausting function on the affine variety $\CP^n\backslash (H_d\cup H)$ such that $\rd\rd^{c} h = 2(d+1)\omega_{FS}$. Then $Y_d$, as a contact manifold, can be defined as a level set $h^{-1}(C)$ for $C\gg 0$ larger than all the critical values of $h$ with contact form $\rd^{c}h$. Next let $\zeta$ be a holomorphic section of $\cO(1)$ with zero set $H$, then $\zeta^{d+1}$ is a holomorphic section of $\cO(d+1)$ with zero set $H$ of order $d+1$. We have $2(d+1)\omega_{FS}=\rd\rd^{c} h'$, where $h'=-\log|\zeta^{d+1}|$. Note that $(\CP^n\backslash H, 2(d+1)\omega_{FS})$ is symplectomorphic to a Darboux ball of radius $\sqrt{2(d+1)}$. $(\CP^{n}\backslash H\cup h^{-1}(-\infty,C),\omega)$ is a symplectic cobordism $W$ (after we compactify the convex boundary by adding the boundary of the Darboux ball) with symplectic form $\omega=2(d+1)\omega_{FS}$ from $Y_d$ to $(S^{2n-1}, 2(d+1)\alpha_{std})$, where $S^{2n-1}$ is the unit sphere with $\alpha_{std}=\frac{1}{2}\sum_{i=1}^n (x_i\rd y_i-y_i\rd x_i)$. Note that the symplectic form $\omega$ on  $W\backslash H_d$ is exact with a primitive extending the contact form on $Y_d$. Moreover, as $H_d$ is transverse to $H$, $H_d\backslash H$ is a symplectic hypersurface in the interior of $W$ and the closure
$$S:=\overline{H_d\backslash H}$$ is a symplectic hypersurface in $W$ with a convex boundary 
$$\Gamma = \pi^{-1}(H_d\cap H),$$
where $\pi:S^{2n-1}\to H\simeq \CP^{n-1}$ is the Hopf fibration. Now, we slightly perturb the contact form $(2d+2)\alpha_{std}$ to a non-degenerate ellipsoid contact form $\alpha_+$ and perturb the contact form $\rd^{c}h$ on $Y_d=h^{-1}(C)$ to a non-degenerate contact form $\alpha_-$. Therefore we get a strong symplectic cobordism $(W,\omega)$ from $(Y_d,\alpha_-)$ to $(S^{2n-1},\alpha_{+})$ such that $\omega$ is exact away from a symplectic hypersurface $S$ with a primitive extending the contact form $\alpha_-$. We can assume the minimal Reeb orbit $\gamma_{\min}$ of $\alpha_+$ is disjoint from the boundary $\Gamma$ of $S$ in $S^{2n-1}$. It is straightforward to see that $\gamma_{\min}$ has linking number $d$ with $\Gamma$.  

Now, in the curve counting setup, we choose an almost complex structure such that the completion $\widehat{S}\subset \widehat{W}$ of $S\subset W$ is holomorphic. Holomorphic curves whose positive asymptotic Reeb orbits are disjoint from $\Gamma$ must have non-negative intersections with $\widehat{S}$. All curves we will consider satisfy this condition. Moreover, as the non-exactness of $W$ is concentrated in $S$, curves contributing to the Maurer-Cartan element must have positive intersection with $\widehat{S}$. From the proof of \cite[Theorem L/Theorem 6.6]{MZ22}, there exists an almost complex structure (actually it works for any generic almost complex structure) on the symplectization of $(S^{2n-1},\alpha_+)$ such that $p^{1,l}_+(q_{\gamma_{\min}})=0$, $p^{1,0}_{\bullet,+}(q_{\gamma_{\min}})=1$ and $p^{1,l}_{\bullet,+}(q_{\gamma_{\min}})=0$ for $l>0$. In fact, when the count is zero, we can actually make the moduli space empty in this case. 

Now let $\epsilon$ be an augmentation of $\RSFT_{\Lambda}(Y_d)$. We consider $v\in \overline{\overline{S}V_{\alpha_-}\otimes_{\Q}\Lambda}$ obtained as follows:
\begin{itemize}
    \item It is obtained from gluing one level of $\phi$ and $\mc$ from $W$  to $q_{\gamma_{\min}}$, then capped off by a level of $\epsilon$.
    \item The glued graph is a forest of trees of exactly one output leaf.
    \item Total intersection number of the level of $\phi$ and $\mc$ with $\widehat{S}$ is $d$.
\end{itemize}
This is completely analogous to the definition of $v$ in the proof of \Cref{prop:simple} (i.e.\ \Cref{fig:planarity}) by replacing the weight condition with the intersection condition. Then the positivity of intersection of curves in $\mc$ with $\widehat{S}$ implies that $v$ has length at most $d+1$. 

Next, we will prove that $v\in \overline{B}^{d+1}V_{\alpha_-}\otimes_{\Q}\Lambda$ instead of the completion. Let $u$ be a holomorphic curve in $\widehat{W}$ asymptotic to $\Gamma_+$ and $\Gamma_-$ near positive and negative punctures such that orbits in $\Gamma_+$ are disjoint from the contact submanifold $\Gamma$, and the intersection number between $u$ and $\widehat{S}$ is $k\ge 0$. We first assume $\alpha_+$ is the round contact form, i.e.\ induced from $\rd^{c}h'$, then we claim that 
\begin{equation}\label{eqn:energy}
\int u^*\overline{\omega} = \sum_{\gamma\in \Gamma_+}\int \gamma^*\alpha_+-\sum_{\gamma\in \Gamma_-}\int \gamma^*\alpha_-+2k\pi+\sum_{\gamma\in \Gamma_+}\int \beta^*\gamma,
\end{equation}
where $\beta$ is a closed $1$-form on $S^{2n-1}\backslash \Gamma$. To see it, as $\int u^*\overline{\omega}$ only depends on the homotopy class of $u$ (as a map asymptotic to those Reeb orbits), we may slide the intersections of $u$ with $\widehat{S}$ inside $W$ and assume $u$ has such a property initially. Let $\lambda =\rd^{c}h$, which is defined on $W\backslash S$, and $\lambda'=\rd^{c}h'$, which is defined on $W$. $\lambda-\lambda'$ is closed on $W\backslash S$ by construction. In particular, as $\lambda'$ restricts to the round contact form $\alpha_+$ on $S^{2n-1}$, we get that $\lambda|_{S^{2n-1}\backslash \Gamma}=\alpha_+|_{S^{2n-1}\backslash \Gamma}+\beta$ for $\beta$ closed in $\Omega^1(S^{2n-1}\backslash \Gamma)$. Strictly speaking, we need to perturb the contact form to be non-degenerate, e.g.\ we can push the boundary inside $W$ generically. Such a modification will not affect the argument below, hence we omit the change. 

Now, as $u$ intersects $S$ at isolated points,  Stokes's formula implies  that 
\begin{eqnarray*}
    \int u^*\overline{\omega} & = & \int_{u^{-1}(\R_+\times S^{2n-1})} u^*(\rd\alpha_+) + \int_{u^{-1}(W\backslash S)} u^*(\rd\rd^{c}h)+\int_{u^{-1}(u^{-1}(\R_-\times Y_d))} u^*(\rd\alpha_-) \\
    &= & \sum_{\gamma\in \Gamma_+}\int \gamma^*\alpha_+-\int_{u^{-1}(\partial_+(W\backslash S)\subset \{0\}\times S^{2n-1})} u^*\alpha_+ \\
    &  & +\int_{u^{-1}(\partial_+(W\backslash S))} u^*\lambda-\int_{u^{-1}(\partial_-(W\backslash S)=\{0\}\times Y_d)} u^*\lambda - \sum_{p, u(p)\in S} \lim_{r\to 0} \int_{S^1} u(p+re^{\mathbf{i}\theta})^*\rd^{c}h\\
    &  & + \int_{u^{-1}(\{0\}\times Y_d)} u^*\alpha_--\sum_{\gamma\in \Gamma_-}\int \gamma^*\alpha_-.
\end{eqnarray*}
where $\partial_{\pm}$ denotes the convex/concave boundary of the cobordism. 
Since $\lambda|_{\partial_-(W\backslash S)}=\alpha_-$, $\lambda|_{\partial_+(W\backslash S)}=\alpha_+|_{S^{2n-1}\backslash \Gamma}+\beta$ and $\beta$ is closed, the above formula can be reduced to 
\begin{equation}
    \int u^*\overline{\omega} = \sum_{\gamma\in \Gamma_+}\int \gamma^*\alpha_+-\sum_{\gamma\in \Gamma_-}\int \gamma^*\alpha_-+\sum_{\gamma\in \Gamma_+}\int \beta^*\gamma - \sum_{p, u(p)\in S} \lim_{r\to 0} \int_{S^1} u(p+re^{\mathbf{i}\theta})^*\rd^{c}h
\end{equation}
To compute the last term, it suffices to do it in a standard holomorphic model, as $\int u^*\overline{\omega}$ is topological; we can certainly arrange the intersection to be in the holomorphic model. More specifically, $W$ near $u(p)$ is modeled on a neighborhood of $0$ in $\C^n$ by coordinates $(z,w)\in \C\times \C^{n-1}$, over which the line bundle is trivial. The holomorphic section $\eta$ is given by $z$ and $S$ is $z=0$. The Hermitian metric is $e^{2\phi}z\overline{z}$. Then $\lambda = \rd^{c} (-\log |z|+\phi)$. The map $u$ near $p=0$ is modeled by $z\mapsto (z^m,f(z))$, where $f(0)=0$ and the intersection number with $S$ is $m$. Then it is straightforward to compute that 
$$\lim_{r\to 0} \int_{S^1} u(p+re^{\mathbf{i}\theta})^*\rd^{c}h= \lim_{r\to 0} \int_{S^1} u(re^{\mathbf{i}\theta})^*\rd^{c}(-\log|z|+\phi)=- 2m\pi,$$
which is independent of $\phi$. Then the claim on $\int u^*\overline{\omega}$ follows. As the total intersection number with $\widehat{S}$ is at most $d$ and the input from $S^{2n-1}$ is only $\gamma_{\min}$, the positivity of $\int u^*\overline{\omega}$ implies that $\Gamma_-$ appeared in $v$ can only involve finitely many Reeb orbits of $Y_d$. As a consequence, $v\in \overline{B}^{d+1}V_{\alpha_-}\otimes_{\Q}\Lambda$.

Then similar to the proof of \Cref{prop:simple}, instead of requiring the level in the cobordism having total weight zero, we require the level in the cobordism having total intersection number $d$ with $\widehat{S}$ to conclude that the upper bound of planarity of $Y_d$ and $\epsilon$ is supplied by $v$, i.e.\ $\Pl_{\Lambda}(Y_d)\le d+1$.
\end{proof}

Given a contact open book $\OB(\Sigma,\phi)$ with page a Liouville domain $\Sigma$ and monodromy $\phi$ in the compactly supported symplectic mapping class group, Bourgeois \cite{Bourgeois} constructed a contact structure on $\OB(\Sigma,\phi)\times T^2$, which we shall denote by $\BO(\Sigma,\phi)$. A basic example is when $\phi=\Id$, $\BO(V,\Id)$ is the contact boundary of $\Sigma \times DT^*T^2$. 
\begin{proof}[Proof of \Cref{cor:BO}]
Following \cite{BGM}, as $\BO(\Sigma,\phi)$ contains the contactization $(S^1_t\times \Sigma \times DT^*S^1,\rd t + \lambda_{\Sigma \times DT^*S^1})$, $c_1(\BO(\Sigma,\phi))=0$ implies that $c_1(\Sigma)=0$. By attaching a round handle in the sense of \cite[\S 5.1]{MNW}, we get a strong cobordism $W$ from $\BO(\Sigma,\phi)$ to $\partial(\Sigma \times DT^*S^1 \times \D)$, see \cite{BGM} for details. The cobordism $W$ enjoys similar properties as the cobordism in the proof of \Cref{cor:divsor}, namely, by \cite[Theorem 5.1 and its proof]{MNW}:
\begin{enumerate}
    \item $W$ is exact away from a symplectic hypersurface ($\simeq \Sigma \times DT^*S^1$);
    \item The intersection of the symplectic hypersurface with the convex boundary is $\partial(\Sigma \times DT^*S^1) \times \{0\}$;
    \item A primitive of the symplectic form on this complement can be chosen as an extension of the contact form on the concave boundary.
\end{enumerate}
In the case of $\phi=\Id$, this is simply the strong cobordism from $\partial(\Sigma \times DT^*T^2)$ to $(\Sigma \times DT^*S^1\times \D)$ induced from the cobordism from $\partial DT^*S^1$ to $\partial \D$, i.e.\ the complement of a Weinstein neighborhood of a Lagrangian $S^1$ in $\D$. By \cite[Theorem L/Theorem 6.6 and its proof]{MZ22}, the planarity $1$ of  $\partial(\Sigma \times DT^*S^1 \times \D)$ is contributed by a Reeb orbit contained in $\Sigma \times DT^*S^1\times S^1\subset \partial(\Sigma \times DT^*S^1 \times \D)$ enjoying the same algebraic properties of $\gamma_{\min}$ in the proof of \Cref{cor:divsor}, namely
\begin{enumerate}
    \item $p^{1,0}_{\bullet}(q_{\gamma_{\min}})=1$;
    \item $p^{1,l}(q_{\gamma_{\min}})=0$ for $l\ge 0$ and the relevant moduli spaces are empty;
    \item $p^{1,l}_{\bullet}(q_{\gamma_{\min}})=0$ for $l>0$ and the relevant moduli spaces are empty;
    \item The linking number of $\gamma_{\min}$ with $\partial(\Sigma \times DT^*S^1) \times \{0\}$ is $1$. 
\end{enumerate}
Then by an identical argument as in \Cref{cor:divsor}, we get that $\Pl_{\Lambda}(\BO(\Sigma,\phi))\le 2$.
\end{proof}

We need the following result to complete the proof of \Cref{thm:no_cofilling}.
\begin{theorem}[{\cite[Theorem 7.14]{MZ22}}]\label{thm:hyperbolic}
    Assume $Y$ is the Boothby-Wang contact structure over a symplectic manifold $(X,\omega,J)$ such that there are no $J$-holomorphic spheres in $X$. Then $\Pl(Y)=\infty$. 
\end{theorem}

\begin{proof}[Proof of \Cref{thm:no_cofilling}]
    By \cite[Lemma 2.5, 2.12]{MR4146343}, $Y_{d,g}$ is not cofillable if and only if $d>2g-2$. By \Cref{thm:hyperbolic} and \Cref{prop:simple}, there is no strong cobordism from $Y_{d,g}$ to $(S^3,\xi_{\mathrm{std}})$ for any $d>0$. 
\end{proof}

\begin{remark}
An alternate proof of the non-existence of a cobordism from $Y_{d,g}$ to $(S^3,\xi_{\mathrm{std}})$ can be obtained as a corollary of the results of Eliashberg \cite{MR1171908}, Gromov \cite{gromov1985pseudo}, and McDuff \cite{MR1049697}. If there is a strong cobordism from $Y_{d,g}$ to $(S^3,\xi_{\mathrm{std}})$, then we can glue it with the natural disk bundle filling of $Y_{d,g}$ to get a strong filling $W$ of  $(S^3,\xi_{\mathrm{std}})$, which is diffeomorphic 
    $$B^4\# \overline{\CP^2}\ldots \# \overline{\CP^2}.$$
    Assume the zero section $\Sigma_g$ of the disk bundle filling is represented by $A=\sum k_i E_i$ in homology, where $E_i$ are exceptional classes. Then
    we have 
    $$-\sum k^2_i=A\cdot A = \langle c_1(W),A\rangle +2g-2 = \sum k_i^2+2g-2 $$
    which is impossible if $g\ge 1$. 
\end{remark}
\section{Generalized constraints and higher genus invariants}\label{s5}
There are several directions in which one can generalize the constructions in \cite{MZ22}. In the following, we will briefly describe two directions. The first one is through considering more general constraints than a point, including constraints from general cycles in $Y$, constraints with additional tangency conditions, and multiple constraints. Similar invariants/curves, from a quantitative perspective, were already considered in \cite{CDPT,cieliebak2018punctured,McSi,siegel2019higher}. The second direction considers the full SFT, i.e.\ including curves of all genera. For this part, we will only consider the algebraic theory that should model the full SFT. It will be a generalization of $BL_\infty$ algebras introduced in \S \ref{ss:BL} and a reformulation of the $IBL_\infty$ algebras introduced in \cite{cieliebak2015homological}. From the analytical perspective, embodying those algebraic structures ($IBL_\infty$ structures) with geometric content in SFT is beyond what was established in \cite{MZ22,pardon2016algebraic,pardon2019contact}, see \Cref{rmk:glue}. From the geometric perspective, even if we assume SFT can be phrased using (cocurved) $IBL_\infty$ algebra (\Cref{def:curved_IBL}) introduced in this paper, there are no known examples exhibiting finite torsion or planarity from non-rational curves. We hope to investigate them further in future work.

\subsection{Generalized constraints}\label{ss:genralized}
\subsubsection{General constraints from $Y$}\label{ss:511}
We can pick a closed submanifold of $Y$ or more generally a closed singular chain $C$ of $Y$ to construct a pointed morphism (\Cref{def:pointed}) $p_{C}$ by considering rational holomorphic curves in $\R\times Y$ with one marked point mapped to $\{0\}\times C$. That is, we have a family of operators
$$p^{k,l}_C:S^kV \mapsto S^l V, \quad k\ge 1, l\ge 0$$
such that $\widehat{p}_C\circ \widehat{p}=(-1)^{|\dim C|}\widehat{p}\circ \widehat{p}_C$.

Then given a $BL_\infty$ augmentation $\epsilon$, we have an induced chain map $\widehat{\ell}_{C,\epsilon}:(\overline{B}^k V_\alpha, \widehat{\ell}_\epsilon)\to \Q$ defined using $p^{k,0}_{C,\epsilon}$ (see discussion after \Cref{def:pointed}). When $X$ is an exact cobordism from $Y_-$ to $Y_+$, let $C_-,C_+$ be two closed singular chains in $Y_-,Y_+$, which define two pointed morphisms $p_{C_-},p_{C_+}$. If $C_-$ and $C_+$ are homologous in $X$, then one can show that $p_{C_-},p_{C_+}$ and the $BL_\infty$ morphisms induced from $X$ are compatible in the sense of Definition \ref{def:compatible} by counting holomorphic curves in $\widehat{X}$ with a point passing through a singular chain whose boundary is $-C_-\sqcup C_+$. By the same reasoning, the homotopy class of the chain map $\widehat{\ell}_{C,\epsilon}$ depends only on the homology class $[C]$. Therefore for any $k\ge 1$, we have a linear map
$$\delta^{k,\vee}_{\epsilon}:H_*(\overline{B}^kV_{\alpha},\widehat{\ell}_{\epsilon})\otimes H_*(Y,\Q) \to \Q, \quad (x,[C])\mapsto \widehat{\ell}_{C,\epsilon}(x).$$
Or equivalently, we can write it as $\delta^k_{\epsilon}:H_*(\overline{B}^kV_\alpha,\widehat{\ell}_{\epsilon})\to H^{*}(Y,\Q)$. Here the grading is in the $\Z/2$ sense\footnote{As we graded the $BL_\infty$ algebra in $\Z/2$. In the case of $\Z$-grading by the SFT degree, the structural maps $p^{k,l}$ will have grading other than $-1$. $\delta^k_{\epsilon}$ will also have a grading shift, which can be deduced from the formula of virtual dimensions.}. 

Combining the method in \cite{bourgeois2009exact} and the argument in \S5 of \cite{MZ22}, one can show that 
$$\delta^1_{\epsilon}:H_*(\overline{B}^1V_\alpha,\widehat{\ell}_{\epsilon})\to H^*(Y,\Q)$$ 
is isomorphic to the composition of the maps 
$$SH_{+,S^1}^*(W;\Q)\to H^{*}_{S^1}(W;\Q):=H^{*+1}(W;\Q)\otimes_\Q \Q[U,U^{-1}]/[U]\to H^{*+1}(Y;\Q),$$
where $W$ is an exact filling and $\epsilon$ is the augmentation induced from the filling $W$\footnote{Note that symplectic cohomology in \cite{zhou2019symplectic} is graded by $n-\mu_{CZ}$, which explains the parity discrepancy.}. Here $U$ is a formal variable of degree $2$, which shall be viewed as the generator of $H^*(BS^1)$. The $S^1$-equivariant symplectic cohomology $SH_{+,S^1}^*(W;\Q), H^{*}_{S^1}(W;\Q)$ are modules over $H^*(BS^1)$\footnote{There are three versions of $S^1$-equivariant theory associated to an $S^1$-complex, as the $S^1$-equivariant symplectic cohomology is actually the positive version, see \cite[\S 2]{zbMATH07134620}, we have $H^{*}_{S^1}(W;\Q):=H^{*+1}(W;\Q)\otimes_\Q \Q[U,U^{-1}]/[U]$ instead of the usual $S^1$-equivariant cohomology for $W$ with trivial $S^1$-action.}. This fact was used in \cite{zhou2019symplecticI,zhou2019symplectic} (but not phrased in SFT) to define obstructions to Weinstein fillings. In principle, $H_*(\overline{B}^kV_\alpha,\widehat{\ell}_{\epsilon})\to H^{*}(Y,\Q)$ can be used to obstruct Weinstein fillings. As the map factors through the filling $W$ if the augmentation arises from the filling and $H^*(W)=0$ when $*>\frac{\dim W}{2}$ and $W$ is Weinstein, if the image contains an element of degree $>\frac{\dim Y+1}{2}$ for all possible augmentations that could be from a Weinstein filling, then $Y$ has no Weinstein filling. To sum up, we have the following lemma.

\begin{lemma}\label{lemma:obstruction}
If the image of $\delta^k_{\epsilon}:H_*(\overline{B}^kV_\alpha,\widehat{\ell}_{\epsilon})\to H^{*}(Y,\Q)$ contains an element with a non-trivial component of degree $>\frac{\dim Y+1}{2}$ for any augmentation $\epsilon$, then $Y$ has no Weinstein filling.
\end{lemma}
\begin{remark}
    We do not discuss $\Z$-graded $BL_\infty$ algebras in this paper. But in the $\Z$-graded cases (where $p^{k,l}$ will have degree shift depending on $k$ and $\dim Y$), \Cref{lemma:obstruction} holds for using only $\Z$-graded $BL_\infty$ augmentation $\epsilon$ ($\epsilon^k$ also has degree shift depending on $k$ and $\dim Y$) when $c_1(Y)=0$ and $\dim Y\ge 5$. As in those cases, augmentation from a Weinstein filling must be $\Z$-graded.
\end{remark}

In the following, we will explain the functorial part of $\delta_{\epsilon}$ without proof. Given an exact cobordism $X$ from $Y_-$ to $Y_+$, let $C$ be a closed chain in $X$. Then by counting rational holomorphic curves in $X$ with a marked point mapped to $C$, we obtain a family of maps $\phi_C^{k,l}:S^kV_{\alpha_+}\to S^lV_{\alpha_-}$. Along with the $BL_{\infty}$ morphism $\phi^{k,l}$ from $X$, we can construct a map $\widehat{\phi}_C:EV_{\alpha_+}\to EV_{\alpha_-}$ from $\phi^{k,l}_C,\phi^{k,l}$ by the same rule as $\widehat{\phi}_{\bullet}$. Now that $C$ is closed, we have $\widehat{\phi}_C\circ \widehat{p}_+=\widehat{p}_-\circ \widehat{\phi}_C$. Then given an augmentation $\epsilon$ of $V_{\alpha_-}$, we have a linearized relation $\widehat{\phi}_{C,\epsilon}\circ \widehat{p}_{+,\epsilon \circ \phi}=\widehat{p}_{-,\epsilon}\circ \widehat{\phi}_{C,\epsilon}$, where $\widehat{\phi}_{C,\epsilon}= \widehat{F}_{\epsilon}\circ \widehat{\phi}_C \circ \widehat{F}_{-\epsilon\circ \phi}$. In particular, $\widehat{\phi}^{k,0}_{C,\epsilon}$ defines a chain map $(\overline{B}^kV_{\alpha_+},\widehat{\ell}_{\epsilon\circ \phi})\to \Q$. By a similar argument as before, such a construction yields a map

$$\delta^{k,\vee}_{X,\epsilon}: H_*(\overline{B}^kV_{\alpha_+},\widehat{\ell}_{\epsilon\circ \phi})\otimes H_*(X;\Q)\to \Q$$ 
where the dual version $H_*(\overline{B}^kV_{\alpha_+},\widehat{\ell}_{\epsilon\circ \phi}) \to H^*(X;\Q)$ is denoted by $\delta^k_{X,\epsilon}$. If $C$ comes from $H_*(Y_+)$, then $\delta^{k,\vee}_{X,\epsilon}(- \otimes C)$ coincides with $\delta^{\vee}_{\epsilon\circ \phi}(-\otimes C )$ by the same argument as Proposition 5.14 in \cite{MZ22}. If $C$ comes from $Y_-$, then $\delta^{k,\vee}_{X,\epsilon}$ factors through $\widehat{\phi}_{\epsilon}:H_*(\overline{B}^kV_{\alpha_+},\widehat{\ell}_{\epsilon \circ \phi }) \to H_*(\overline{B}^kV_{\alpha_-},\widehat{\ell}_{\epsilon })$ and $\delta^{k,\vee}_{X,\epsilon}(- \otimes C)=\delta^{k,\vee}_{\epsilon}(\widehat{\phi}_\epsilon(-)\otimes C)$ by an argument similar to \cite[Proposition 5.14]{MZ22}. Dualizing those properties, we have the following commutative diagram:

$$\xymatrix{
H_*(\overline{B}^kV_{\alpha_+},\widehat{\ell}_{\epsilon \circ \phi }) \ar[r]_{\quad \delta^k_{X,\epsilon}}\ar[d]^{\widehat{\phi}_{\epsilon}}\ar@/^2pc/[rr]^{\delta^k_{\epsilon\circ \phi}} & H^{*}(X;\Q) \ar[r]\ar[d] & H^*(Y_+;\Q)\\
H_*(\overline{B}^kV_{\alpha_-},\widehat{\ell}_{\epsilon }) \ar[r]^{\delta^k_{\epsilon}} & H^{*}(Y_-;\Q)}$$

\subsubsection{Tangency conditions}\label{ss:512}
Another type of generalization is considering point constraints with tangency conditions, i.e.\ we consider curves in the symplectization passing through a fixed point $p$ and tangent to a local divisor near $p$ with order $m$. Such holomorphic curves were considered in \cite{CDPT,cieliebak2018punctured,McSi,siegel2019higher}. These curves also give rise to pointed morphisms and hence can be used to define a new algebraic order. In many cases, if we have a holomorphic curve with a point constraint without tangent conditions, then multiple covers of it might have tangent properties. One can show that the algebraic order with a tangency condition for $(S^{2n-1},\xi_{std}),n\ge 2$ is always $1$, no matter what the order of tangency is. Here we give a brief explanation: we claim that for any augmentation, there is a Reeb orbit $\gamma$ such that
\begin{enumerate}
    \item $\gamma$ represents a closed class in the linearized contact homology.
    \item The count of rigid rational curves with one positive puncture asymptotic to $\gamma$, a point constraint with tangent order $r$, possibly with negative punctures that are capped off by the augmentation $\epsilon$, is non-zero.
\end{enumerate}
In a very thin ellipsoid, the Reeb orbit is $\gamma_{\min}^{r+1}$, where $\gamma_{\min}$ is the minimal Reeb orbit. The Conley-Zehnder index of $\gamma_{\min}^{r+1}$ is $n+1+2r$. As all Reeb orbits have even SFT degree, $\gamma_{\min}^{r+1}$ is closed in the linearized contact homology for any augmentation. To see the second claim, those with negative punctures are of negative virtual dimension. For the case without negative punctures, we first consider curves in the thin ellipsoid. By putting the point constraint on the zero and using the complex hypersurface complement to the shortest Reeb orbit, there is one curve passing through the zero with tangent order $r$, namely the $(r+1)$-fold branched cover of the coordinate plane with one branching point that passes through the point constraint. If we use a product complex structure on $\C^n$, then this is the only curve with such properties. A suitable continuation argument shows that the algebraic count of such curves is still one if we use a usual almost complex structure that is compatible with the contact structure.  Then one can use a neck-stretching argument to prove that it is contained in the symplectization and independent of the augmentation, as there is no room in virtual dimension to support a curve with both tangent constraint and negative punctures. 

Invariants with local tangent constraints for exact domains are defined and computed in \cite{LC}.

\subsubsection{Multiple point constraints}\label{ss:513}
It is natural to consider generalizations of pointed morphisms of $BL_\infty$ algebras to maps induced from counting curves with multiple constraints. For example, we can consider rational holomorphic curves with $2$ marked points passing through two fixed points in the contact manifold, where the curve can be disconnected. More specifically, we have three maps from $S^kV$ to $S^lV$. Namely, we have $p_{\bullet\bullet}^{k,l}$ from counting connected holomorphic curves with two marked points, $p_{\bullet_1}^{k,l},p_{\bullet_2}^{k,l}$ coming from counting connected holomorphic curves with each one of the point constraints respectively. Then we can assemble them to $\widehat{p}_{\bullet\bullet}$ by the same rule as $\widehat{p}_{\bullet}$ except the middle level consists of one $p_{\bullet\bullet}$ or both $p_{\bullet_1},p_{\bullet_2}$. 

\begin{figure}[H]
	\begin{center}
		\begin{tikzpicture}
		\node at (0,0) [circle,fill,inner sep=1.5pt] {};
		\node at (1,0) [circle,fill,inner sep=1.5pt] {};
		\node at (2,0) [circle,fill,inner sep=1.5pt] {};
		\node at (3,0) [circle,fill,inner sep=1.5pt] {};
	
		\draw (0,0) to (0.5,1) to (1,0);
		\draw (2,0) to (2.5,1) to (3,0);
	
		\draw (2,0) to (1.5,-1) to (2,-2);
		\draw (1,0) to (1.5,-1) to (1,-2);
		\node at (1.5,-1) [circle, fill=white, draw, outer sep=0pt, inner sep=5 pt] {};
        \node at (1.4,-1) [circle,fill,inner sep=1.5pt] {};
        \node at (1.6,-1) [circle,fill,inner sep=1.5pt] {};
		\node at (2.1,-1) {$p^{2,2}_{\bullet\bullet}$};

		\node at (0,-2) [circle,fill,inner sep=1.5pt] {};
		\node at (1,-2) [circle,fill,inner sep=1.5pt] {};
		\node at (2,-2) [circle,fill,inner sep=1.5pt] {};
		\node at (3,-2) [circle,fill,inner sep=1.5pt] {};

		\draw[dashed] (0,0) to (0,-2);
		\draw[dashed] (3,0) to (3,-2);
		\end{tikzpicture}
        \hspace{1cm}
        \begin{tikzpicture}
		\node at (0,0) [circle,fill,inner sep=1.5pt] {};
		\node at (1,0) [circle,fill,inner sep=1.5pt] {};
		\node at (2,0) [circle,fill,inner sep=1.5pt] {};
		\node at (3,0) [circle,fill,inner sep=1.5pt] {};
	
		\draw (0,0) to (0.5,1) to (1,0);
		\draw (2,0) to (2.5,1) to (3,0);
	
		\draw (2,0) to (1.5,-1) to (2,-2);
		\draw (1,0) to (1.5,-1) to (1,-2);
        \draw (3,0) to (3,-2);
		\node at (1.5,-1) [circle, fill=white, draw, outer sep=0pt, inner sep=5 pt] {};
        \node at (1.5,-1) [circle,fill=blue,inner sep=1.5pt] {};
        \node at (3,-1) [circle, fill=white, draw, outer sep=0pt, inner sep=5 pt] {};
        \node at (3,-1) [circle,fill=red,inner sep=1.5pt] {};
		\node at (2.1,-1) {$p^{2,2}_{\bullet_1}$};
        \node at (3.6,-1) {$p^{1,1}_{\bullet_2}$};

		\node at (0,-2) [circle,fill,inner sep=1.5pt] {};
		\node at (1,-2) [circle,fill,inner sep=1.5pt] {};
		\node at (2,-2) [circle,fill,inner sep=1.5pt] {};
		\node at (3,-2) [circle,fill,inner sep=1.5pt] {};

		\draw[dashed] (0,0) to (0,-2);
		\end{tikzpicture}
	\end{center}
	\caption{Two components of $\widehat{p}_{\bullet\bullet}$ with input from $S^2V\odot S^2V$ and output in $S^4V$.}
\end{figure}

Note that the combinatorics behind $\widehat{p}_{\bullet\bullet}$ is slightly different from $\widehat{p}_{\bullet}$ and $\widehat{p}$, as it can have disconnected curves, i.e.\ the term involving both $p_{\bullet_1},p_{\bullet_2}$. It is also different from the combinatorics behind $\widehat{\phi}$ for $BL_\infty$ morphism, as the (non-trivial) curves can only have at most two connected components. Nevertheless, we have $\widehat{p}_{\bullet\bullet}\circ \widehat{p}=\widehat{p}\circ \widehat{p}_{\bullet\bullet}$ and given a $BL_\infty$ augmentation, we can similarly define $\widehat{p}_{\bullet\bullet,\epsilon}$, which is assembled from $p_{\bullet\bullet,\epsilon},p_{\bullet_1,\epsilon}$ and $p_{\bullet_2,\epsilon}$. 

In general, we can consider disconnected rational holomorphic curves with $k$ marked points passing through $k$ point constraints, which give rise to an operator $\widehat{p}_{k\bullet}$ on $EV$.  $\widehat{p}_{k\bullet}$ is homotopic to $\widehat{p}_{\bullet}^k$.

For a pointed morphism, we can define a slightly different order using the full linearized $BL_\infty$ structure instead of just the $L_\infty$ structure as follows. 
\begin{definition}
	Given a $BL_{\infty}$ augmentation and a pointed map $p_{\bullet}$, we define
	$$\tilde{O}(V,\epsilon,p_{\bullet}):=\min\left\{k\left|1\in \Ima \pi_{\Q}\circ \widehat{p}_{\bullet,\epsilon}|_{H_*(\overline{B}^k \overline{S}V,\widehat{p}_{\epsilon})}\right.\right\},$$
 where $\pi_{\Q}:EV\to \Q \subset SV$ is the projection. 
\end{definition}
Here, as $p^{k,0}_{\epsilon}=0$, $(\overline{B}^k\overline{S}V,\widehat{p}_{\epsilon})$ is a subcomplex of $(EV,\widehat{p}_{\epsilon})$. On the other hand, $\widehat{p}_{\bullet,\epsilon}$ is a chain map from $(EV,\widehat{p}_{\epsilon})$ to itself, but it will not preserve the subcomplex  $(\overline{B}^k\overline{S}V,\widehat{p}_{\epsilon})$. Finally, again by $p^{k,0}_{\epsilon}=0$, $\pi_{\Q}:EV\to \Q$ is a chain map. As a consequence, $\pi_{\Q}\circ \widehat{p}_{\bullet,\epsilon}:H_*(\overline{B}^k\overline{S}V,\widehat{p}_{\epsilon})\to \Q$ is well-defined. In the proof of the following proposition, we will see that $O(V,\epsilon,p_{\bullet})$ and $\tilde{O}(V,\epsilon,p_{\bullet})$ look for similar curves (in the context of SFT), but $\tilde{O}(V,\epsilon,p_{\bullet})$ has a stronger closedness requirement involving $p^{k,l}_{\epsilon}$ for $l\ge 2$.  

\begin{proposition}\label{prop:inequ}
	$O(V,\epsilon,p_{\bullet}) \le \tilde{O}(V,\epsilon,p_{\bullet})$.
\end{proposition}
\begin{proof}
We use $\overline{B}^k_2\subset \overline{B}^k\overline{S}V$ to denote the subspace 
$$\bigoplus_{j=0}^{k-1}\left(\left(\bigoplus_{i=2}^\infty S^i V\right) \odot \left(\left(\overline{S}V\right)^{\odot j}\right)\right).$$ 
Since $p^{k,0}_{\epsilon}=0$, we have $\overline{B}^k_2$ is a subcomplex. Moreover, the quotient complex $\overline{B}^k\overline{S}V/\overline{B}^k_2$ is exactly $(\overline{B}^kV,\widehat{\ell}_{\epsilon})$. Note that $\pi_{\Q}\circ\widehat{p}_{\bullet,\epsilon}$ on $\overline{B}^k_2$ is zero, as $\widehat{p}_{\bullet,\epsilon}|_{\overline{B}^k_2}$ contains at least one $V$ component in the output. Moreover, $\pi_{\Q}\circ \widehat{p}_{\bullet,\epsilon}$ on $\overline{B}^k\overline{S}V/\overline{B}^k_2V$ is identified with $\widehat{\ell}_{\bullet,\epsilon}$ on $\overline{B}^kV$. Therefore $1\in \pi_{\Q}\circ \widehat{p}_{\bullet,\epsilon}|_{H^*(\overline{B}^k \overline{S}V,\widehat{p}_{\epsilon})}$ implies that $1\in \Ima \widehat{\ell}_{\bullet,\epsilon}|_{H_*(\overline{B}^k V, \widehat{\ell}_{\epsilon})}$. Hence the claim follows.
\end{proof}

Similarly, we can define
$$\tilde{O}(V,\epsilon,p_{k\bullet}):=\min\left\{k\left|1\in \Ima \pi_{\Q}\circ \widehat{p}_{k\bullet,\epsilon}|_{H_*(\overline{B}^k \overline{S}V,\widehat{p}_{\epsilon})}\right.\right\}.$$
To give the analog of $O(V_\alpha,\epsilon,p_{\bullet})$ when we have multiple point constraints, we consider the projection $\pi_{m}:\overline{B}^k \overline{S}V_\alpha \to \overline{B}^k \overline{B}^m V_\alpha$. Then we have a filtration $\overline{B}^k \overline{S}V_\alpha:=\ker \pi_0\supset \ker \pi_1 \supset \ldots$. Then we define the width of an element in $\overline{B}^k \overline{S}V_\alpha$ by
	$$w(v)=\sup \left\{m|v\in \ker \pi_m\right\}.$$
	We claim that $w(\widehat{p}_{\epsilon}(v))\ge w(v)$. To see this, we will first prove $\pi_m\circ \widehat{p}_{\epsilon}=\pi_m\circ\widehat{p}_{\epsilon}\circ \pi_m$ as follows. For $x\in \overline{B}^k \overline{S}V_\alpha$, we can write $x=x_0+x_1$ for $x_0=\pi_m(x)\in \overline{B}^k \overline{B}^mV_\alpha,x_1\in \ker \pi_m$. Since $p^{k,0}_{\epsilon}=0$, we have $\pi_m\circ \widehat{p}_{\epsilon}(x_1)=0$. As a consequence, we have $\pi_m\circ \widehat{p}_{\epsilon}(x)=\pi_m\circ \widehat{p}_{\epsilon}\circ \pi_m(x)$. Therefore if $\pi_m(v)=0$, we have $\pi_m\circ \widehat{p}_{\epsilon}(v)=0$. In particular,  $w(\widehat{p}_{\epsilon}(v))\ge w(v)$. Therefore $\pi_m\circ \widehat{p}_{\epsilon}|_{\overline{B}^k\overline{B}^m V_\alpha}$ squares to zero, hence is a differential on $\overline{B}^k\overline{B}^m V_\alpha$. $\pi_m\circ \widehat{p}_{\epsilon}$ uses the knowledge of $p^{k,l}_{\epsilon}$ for $l \le m$ and $\pi_1\circ \widehat{p}_{\epsilon}=\widehat{\ell}_{\epsilon}$. 
	
	On the other hand, given the operator $\widehat{p}_{m\bullet}$ defined using $m$ point constraints, we have the deformed chain map $\widehat{p}_{m\bullet,\epsilon}=\widehat{F}_{\epsilon}\circ \widehat{p}_{m\bullet}\circ \widehat{F}_{-\epsilon}$ from $(EV, \widehat{p}_{\epsilon})$ to itself, where $\widehat{F}_{\pm \epsilon}$ is the change of coordinates used after \Cref{prop:linearized_BL}. Now we consider $\widehat{p}_{m\bullet,\epsilon}$ on the subcomplex $(\overline{E}V:=\overline{S}\overline{S}V,\widehat{p}_{\epsilon})$. Since $\widehat{p}_{m\bullet,\epsilon}$ can have at most $m$ nontrivial connected components, we have $\pi_{\Q}\circ \widehat{p}_{m\bullet,\epsilon}(v)=0$ if $w(v)\ge m+1$. Hence we have $\pi_{\Q}\circ \widehat{p}_{m\bullet,\epsilon}\circ \pi_m \circ \widehat{p}_{\epsilon}=\pi_{\Q}\circ \widehat{p}_{m
	\bullet,\epsilon}\circ \widehat{p}_{\epsilon}=0$ and  $\pi_{\Q}\circ \widehat{p}_{m\bullet,\epsilon}\circ \pi_m = \pi_{\Q}\circ \widehat{p}_{m\bullet,\epsilon}$, i.e.\ we have the following commutative diagram of chain complexes:
	\begin{equation}\label{diagram:order}
	\xymatrix{
	(\overline{B}^k \overline{B}^m V_\alpha,\pi_m\circ\widehat{p}_{\epsilon}) \ar[rrr]^{\pi_{\Q}\circ \widehat{p}_{m\bullet,\epsilon}} & & &\Q \\
	(\overline{B}^k \overline{S}V_\alpha,\widehat{p}_{\epsilon}) \ar[rrr]^{\pi_{\Q}\circ \widehat{p}_{m\bullet,\epsilon}}\ar[u]^{\pi_m}& & &\Q\ar[u]^{=} }
    \end{equation}

	\begin{definition}
		We define 	$O(V,\epsilon,p_{m\bullet}):=\min \left\{ k\left|1\in \Ima \pi_{\Q}\circ \widehat{p}_{m\bullet,\epsilon}|_{H_*(\overline{B}^k \overline{B}^m V, \pi_m\circ \widehat{p}_{\epsilon})}\right. \right\}$.
	\end{definition}
	In particular, $O(V,\epsilon,p_{1\bullet})$ is Definition \ref{def:order}. Moreover, diagram \eqref{diagram:order} implies the following.
	\begin{proposition}
		$O(V,\epsilon,p_{m\bullet})\le \tilde{O}(V,\epsilon,p_{m\bullet})$.
	\end{proposition}
Given a strict exact filling $(W,\lambda)$ of $(Y,\alpha)$, assume $\widetilde{O}(V_\alpha,\epsilon,p_{k\bullet})<\infty$ for the $BL_\infty$ augmentation $\epsilon$ coming from $W$. Then we can define the spectral invariant for $l\ge \widetilde{O}(V_\alpha,\epsilon,p_{k\bullet})$:
	$$\mathfrak{r}^{\le l}\langle \underbrace{p,\ldots,p}_k\rangle :=\inf\left\{a: T^a\in \Ima \pi_{\Lambda_0}\circ \widehat{p}_{k\bullet,\epsilon}|_{H_*(\overline{B}^l\overline{S}V^{\ge 0}_\alpha, \widehat{p}_{\epsilon})} \right\}<\infty,$$
where $V^{\ge 0}_\alpha$ is the free $\Lambda_0:=\{\sum a_i T^{b_i}| b_i\ge 0, \lim b_i =\infty,a_i\in \Q\}$ module generated by good Reeb orbits of $\alpha$. This is exactly the higher symplectic capacity $\mathfrak{r}^{\le l}\langle \underbrace{p,\ldots,p}_k\rangle$ defined by Siegel in \cite[\S 6]{siegel2019higher}. In \cite{siegel2019higher}, the author considered disconnected holomorphic curves in $W$ with $k$ point constraints and viewed them as a chain map from $\overline{S}\CHA(Y)=EV_\alpha$ to $\Lambda$, where $\CHA$ denotes the contact homology algebra. Then the capacity is defined to be the infimum of $a$ such that there is a closed class in $x\in \overline{B}^k SV^{\ge 0}_\alpha$ such that $x$ is mapped to $T^a$ by the chain map. The equivalence of these two definitions can be seen from the same argument as Proposition 5.14 in \cite{MZ22}. Similarly, combining the tangency conditions and multiple point constraints, we can define the analogous orders, whose spectral invariant is again equivalent to the higher capacity $\mathfrak{r}^{\le l}\langle\cT^{m_1}p,\ldots,\cT^{m_k}p\rangle$ in \cite[\S 6]{siegel2019higher}. Similarly, the spectral invariant for $O(V,\epsilon,p_{\bullet})$ is equivalent to $\mathfrak{g}^{\le l}_1$ and the spectral invariant for the analogous version of pointed morphism with tangency conditions of order $m$ is $\mathfrak{g}^{\le l}_{t^m}$ in \cite{siegel2019higher}.
	
	\begin{example}
		Let $W$ be an irrational ellipsoid with $\alpha$ the contact form on $\partial W$. Then we have $\tilde{O}(V_\alpha,\epsilon_W,p_{k\bullet})=1$ for all $k\ge 1$. To see this, we let $\gamma$ denote the shortest Reeb orbit; then for a generic point in $W$ and a generic admissible almost complex structure, there is one\footnote{More precisely, the algebraic count of such curves is $1$, see e.g.\ \cite[Proposition 2.11]{zhou2020mathbb}. Strictly speaking, \cite[Proposition 2.11]{zhou2020mathbb} regards solutions to the Floer equation. One can apply neck-stretching and note that the Reeb orbit has the minimal period, to see that the count for holomorphic curves is also $1$.  This is because the Floer plane will break into a holomorphic plane and a reparameterization of the trivial cylinder solving the Floer equation in \cite[Proposition 4.7]{zbMATH07926124}.} holomorphic plane in $\widehat{W}$ asymptotic to $\gamma$ and passing through the point. Since $\mu_{CZ}(\gamma')+n-3=0 \mod 2$ for any Reeb orbit $\gamma'$, we have $q_{\gamma}\in \overline{B}^1 \overline{S}V_\alpha$ is a closed class and $\pi_\Q \circ \widehat{p}_{\bullet,\epsilon_W}(q_\gamma)=1$ by the holomorphic curve above. In general, we have $ q^k_{\gamma}\in \overline{B}^1 \overline{S}V_\alpha$ is closed by degree reasons. On the other hand, $\pi_\Q \circ \widehat{p}_{k\bullet,\epsilon_W}|_{S^kV_\alpha\subset \overline{B}^1\overline{S}V_{\alpha}}$ must only use disconnected curves with $k$ components and each component has one positive puncture, for otherwise, genus has to be created. As a consequence, we have $\pi_\Q \circ \widehat{p}_{k\bullet,\epsilon_W}(q^k_\gamma)=1$ by the disconnected $k$ copies of the curve above. Then one can show that such a curve is independent of the augmentation, as the curve above is contained in the symplectization after a neck-stretching, see \cite{zhou2020mathbb,zhou2019symplecticI}.
		
		Although there is nothing interesting happening for the order $\tilde{O}(V_\alpha,\epsilon_W,p_{k\bullet})$, we know that the spectral invariant  $\mathfrak{r}^{\le l}\langle \underbrace{p,\ldots,p}_k\rangle$  is defined for all $k\ge 1$ and $l\ge 1=\tilde{O}(V_\alpha,\epsilon_W,p_{k\bullet})$. Those numerical invariants are very sensitive to the shape of $W$ and are powerful tools to study embedding problems, see \cite{siegel2019higher} for details.
	\end{example}
    
	\begin{remark}
		The spectral invariants for $O(V,\epsilon,p_{m\bullet})$  provide a new family of higher symplectic capacities. It is unclear whether they provide new obstructions to symplectic embeddings. 
	\end{remark}

    \begin{remark}\label{rmk:status}
        Structures in \S \ref{ss:511}-\ref{ss:513} can be defined to the same extent as those defined in \cite{MZ22} and recalled in \S \ref{s3}. Namely, the analogue of \Cref{thm:vfc} holds for those structures defined using VFC. In particular, those order invariants from tangency conditions and multiple constraints are well-defined and functorial on the cobordism category by the same proof of \cite[Proposition 3.23, 3.25]{MZ22}. 
    \end{remark}
\subsection{Higher genera}\label{ss:higher_genera}
Another natural direction to generalize is allowing the genus of holomorphic curves to vary, i.e.\ considering the full SFT. Originally, the full SFT was phrased as a differential Weyl algebra with a distinguished odd degree Hamiltonian $\bH$ such that $\bH \star \bH=0$ \cite{EGH}. There are two closely related ways to view the full SFT as a functor from $\cont$ in a more convenient way, namely the $BV_\infty$ formulation \cite{cieliebak2009role} and the $IBL_\infty$ formulation \cite{cieliebak2015homological}. In the following, we will first briefly recall their definitions. In view of the $BL_\infty$ formalism in this paper, we will give a slightly different but equivalent definition of $IBL_\infty$ algebras. Then we will make some speculations assuming the analytical foundation of the full SFT is completed.

\subsubsection{$IBL_\infty$ algebras and $BV_\infty$ algebras} 
Let $V$ be a $\Z/2$ graded vector space over $\bk$ and $\phi:S^kV \to S^l V$. We can define $\widehat{\phi}:\overline{S}V \to \overline{S}V$ by $\widehat{\phi}=0$ on $S^mV$ with $m<k$ and 
$$\widehat{\phi}(v_1 \ldots  v_m)=\sum_{\sigma \in Sh(k,m-k)} \frac{(-1)^{\diamond}}{k!(m-k)!} \phi(v_{\sigma(1)} \ldots  v_{\sigma(k)}) v_{\sigma(k+1)} \ldots  v_{\sigma(m)},$$
for $m\geq k$, where $(-1)^{\diamond}v_{\sigma(1)}\ldots v_{\sigma(m)}=v_1\ldots v_m$ in $S^kV$.
\begin{definition}[{\cite[Definition 2.3]{cieliebak2015homological}}]\label{def:IBL}
	Let $V$ be a $\Z/2$ graded vector space over $\bk$. An $IBL_\infty$ structure on $V$ is a family of operators $p^{k,l,g}:S^kV \to S^l V$ for $k,l\ge 1$ and $g\ge 0$, such that 
	$$\widehat{p}:=\sum_{k,l=1}^\infty\sum_{g=0}^\infty \widehat{p}^{k,l,g}\hbar^{k+g-1}\tau^{k+l+2g-2}:      \overline{S}V[[\hbar,\tau]]\to \overline{S}V[[\hbar,\tau]] $$
	satisfies that $\widehat{p}\circ \widehat{p}=0$ and $|\widehat{p}|=-1$. Here $|\hbar|=0$ and $|\tau|=0$.
\end{definition}
 The original definition of $IBL_\infty$ algebra on a $\Z/2$ graded vector space involves taking a suspension as the formalism for $L_\infty$ algebras in \cite{MZ22}. To align with the notation of $BL_\infty$ algebras, we will not take the suspension. In particular, if $V$ is an $IBL_\infty$ algebra in Definition \ref{def:IBL}, then $V[-1]$ is an $IBL_\infty$ algebra in the sense of \cite{cieliebak2015homological}.

In the case of $V$ being $\Z$ graded, then one can define $IBL_\infty$ structures of degree $d$ by requiring $|\hbar|=2d$,  $|\widehat{p}|=-1$ and $|\tau|=0$. In the case of SFT, $d=n-3$ for a $2n-1$ dimensional contact manifold $(Y,\xi)$ with $c_1(\xi)=0$. If moreover we know that for any $v_1,\ldots,v_k\in V$ and $g\ge 0$ there exist at most finitely many $l\ge 1$ such that $p^{k,l,g}(v_1 \ldots  v_k)\ne 0$, then $\widehat{p}$ is well-defined on $\overline{S}V[[h]]$ by setting $\tau=1$. The formal variable $\tau$ is used to keep track of the (negative) Euler characteristic of the punctured surface \cite[(2) after Remark 2.4]{cieliebak2015homological}. To compare with the $BL_\infty$ algebra in \S \ref{ss:BL}, which does not have the $\tau$-variable, we will always specialize at $\tau=1$ in the following. Note that we only consider $p^{k,l,g}$ with $l\ge 1$ in Definition \ref{def:IBL}. In the context of SFT, such a requirement is equivalent to counts of rigid holomorphic curves in the symplectization with no negative punctures always being zero.  This, in general, does not hold true. Indeed, Definition \ref{def:IBL} is used to model the linearized theory, i.e.\ SFT associated to an exact domain. In general, we will need the following (cocurved) $IBL_\infty$ algebras (specialized at $\tau=1$).

\begin{definition}\label{def:curved_IBL}
	Let $V$ be a $\Z/2$ graded vector space over $\bk$. A (cocurved) $IBL_\infty$ structure on $V$ is a family of operators $p^{k,l,g}:S^kV\to S^lV$ for $k\ge 1,l,g\ge 0$, such that
	\begin{enumerate}
		\item\label{IBL_1} for $v_1,\ldots,v_k\in V, g\ge 0$, there are finitely many $l$ such that $p^{k,l,g}(v_1\ldots v_k)\ne 0$,
		\item $\widehat{p}:=\sum_{k=1}^\infty\sum_{l,g=0}^\infty\widehat{p}^{k,l,g}\hbar^{k+g-1}:SV[[\hbar]]\to SV[[\hbar]]$ satisfies $\widehat{p}\circ \widehat{p}=0$ and $|\widehat{p}|=-1$, where $\widehat{p}(1)$ is defined to be $0$.
	\end{enumerate}
\end{definition}
Note that Definition \ref{def:curved_IBL} gives rise to a $BV_\infty$ algebra introduced in \cite{cieliebak2009role}, which was used to define algebraic torsions in \cite{LW}. For the translations between $BV_\infty$ algebras and differential Weyl algebras with distinguished Hamiltonians, see \cite{cieliebak2009role} for details. In the following, we give an alternative description of Definition \ref{def:curved_IBL}. Assume we are given a family of operators $p^{k,l,g}:S^kV\to S^lV$ for $k\ge 1,l,g\ge 0$, such that \eqref{IBL_1} of Definition \ref{def:curved_IBL} is satisfied. We use $EV[[h]]$ to denote $\overline{S}(SV[[h]])$, where $\overline{S}$ is the symmetric product over $\bk[[\hbar]]$. Then we define $\widehat{p}:EV[[h]]\to EV[[h]]$ by the following graph description. We use a tree with $k+1$ vertices with a label $g\ge 0$ to represent a class $\hbar^gv_1 \ldots  v_k\in SV[[\hbar]]$. Then we use the graph with  $\tikz\draw[black,fill=white] (0,-1) circle (0.4em);$ to represent operators $p^{k,l,g}$, but we label the  $\tikz\draw[black,fill=white] (0,-1) circle (0.4em);$ with a number $g\ge 0$. To define $\widehat{p}$, we represent a class in $EV[[\hbar]]$ by a row of trees with labels, then we glue in one graph representing $p^{k,l,g}$ and dashed vertical lines representing the identity map. Here, we allow cycles to be created. The rule for the output is the same as before with the degree of $\hbar$ determined by the sum of the labels $g$ in the connected component of the glued graph with the number of (independent) cycles in that component. Then $\widehat{p}$ is the sum of all such glued graphs. 
\begin{figure}[H]\label{fig:IBL}
	\begin{center}
		\begin{tikzpicture}
		\node at (0,0) [circle,fill,inner sep=1.5pt] {};
		\node at (1,0) [circle,fill,inner sep=1.5pt] {};
		\node at (2,0) [circle,fill,inner sep=1.5pt] {};
		\node at (3,0) [circle,fill,inner sep=1.5pt] {};
		\node at (4,0) [circle,fill,inner sep=1.5pt] {};
		\node at (5,0) [circle,fill,inner sep=1.5pt] {};
		\node at (6,0) [circle,fill,inner sep=1.5pt] {};
		\node at (7,0) [circle,fill,inner sep=1.5pt] {};

		\draw (0,0) to (1,1) to (1,0);
		\draw (1,1) to (2,0);
		\draw (3,0) to (4,1) to (4,0);
		\draw (4,1) to (5,0);
		\draw (6,0) to (6.5,1) to (7,0);
		
		\node at (1,1.2) {$g_1$};
		\node at (4,1.2) {$g_2$};
		\node at (6.5,1.2) {$g_3$};

		\draw (2,0) to (2.5,-1) to (3,0);
		\draw (2,-2) to (2.5,-1) to (2.5,-2);
		\draw (2.5,-1) to (3,-2);
		\draw (1,0) to (2.5,-1);
		\node at (2.5,-1) [circle, fill=white, draw, outer sep=0pt, inner sep=3 pt] {};
		\node at (2.8,-1) {$g$};

		\node at (0,-2) [circle,fill,inner sep=1.5pt] {};
		\node at (2,-2) [circle,fill,inner sep=1.5pt] {};
		\node at (3,-2) [circle,fill,inner sep=1.5pt] {};
		\node at (4,-2) [circle,fill,inner sep=1.5pt] {};
		\node at (5,-2) [circle,fill,inner sep=1.5pt] {};
		\node at (6,-2) [circle,fill,inner sep=1.5pt] {};
		\node at (7,-2) [circle,fill,inner sep=1.5pt] {};
		\node at (2.5,-2) [circle,fill,inner sep=1.5pt] {};
		
		\draw[dashed] (0,0) to (0,-2);
		\draw[dashed] (4,0) to (4,-2);
		\draw[dashed] (5,0) to (5,-2);
		\draw[dashed] (6,0) to (6,-2);
		\draw[dashed] (7,0) to (7,-2);
		\end{tikzpicture}
	\end{center}
	\caption{A component of $\widehat{p}$ from $\hbar^{g_1}S^3V \odot \hbar^{g_2}S^3 V \odot \hbar^{g_3}S^2 V$ to $\hbar^{g_1+g_2+g+1}S^6V \odot \hbar^{g_3}S^2 V$ using $p^{3,3,g}$.}
\end{figure}
One can think of the tree representing a class in $SV[[\hbar]]$ as a genus $g$ surface with $k$ negative punctures and $p^{k,l,g}$ represented by a genus $g$ surface with $k$ positive punctures and $l$ negative punctures. Then a glued graph represents a possibly disconnected surface with only negative punctures, which represents a class in $EV[[\hbar]]$ with the degree of $\hbar$ being the sum of the genus of all connected components.
\begin{definition}\label{def:curved_IBL'}
	 $(V,p^{k,l,g})$ is a (cocurved) $IBL_\infty$ algebra iff $\widehat{p}:EV[[h]]\to EV[[h]]$ defined above satisfies $\widehat{p}\circ \widehat{p}=0$ and $|\widehat{p}|=1$.
\end{definition}

\begin{proposition}\label{prop:IBL_equi}
	 \Cref{def:curved_IBL,def:curved_IBL'} are equivalent.
\end{proposition}
\begin{proof}
     Throughout this proof, we use $\widehat{p}_S$ to denote $\widehat{p}:SV[[\hbar]]\to SV[[\hbar]]$ in \Cref{def:curved_IBL} and $\widehat{p}_E$ to denote $\widehat{p}:EV[[\hbar]]\to EV[[\hbar]]$ in \Cref{def:curved_IBL'}. Note that $\widehat{p}_E:SV[[\hbar]]\subset EV[[\hbar]] \to SV[[\hbar]]$ as we cannot increase the number of connected components in the graph description of $\widehat{p}_E$. Moreover, $\widehat{p}_E|_{SV[[\hbar]]}$ is exactly the $\widehat{p}_S$ in \Cref{def:curved_IBL} as the $\hbar^{k-1}$ is exactly the number of cycles in the glued graph. Therefore \Cref{def:curved_IBL'} implies \Cref{def:curved_IBL}. 
	
	On the other hand, we assume $\widehat{p}_S:SV[[\hbar]]\to SV[[\hbar]]$ squares to zero. We consider the description of the $\hbar^{N}S^mV$ part of $\widehat{p}_S\circ \widehat{p}_S(v_1 \ldots  v_n)$ for $v_i\in V$, which is zero by assumption. Since $\widehat{p}_S$ on $SV[[\hbar]]$ has the graph description on one tree as explained above, we need to consider possibly disconnected graphs with $n$ input vertices and $m$ output vertices and two  $\tikz\draw[black,fill=white] (0,-1) circle (0.4em);$ vertices such that if we glue all input vertices together by adding a new vertex, the resulting graph has $N$ cycles. Then there are two cases: (1) the two  $\tikz\draw[black,fill=white] (0,-1) circle (0.4em);$ vertices are in different components, then all such graphs will pair up and cancel each other by $|\widehat{p}_S|=1$; (2) the two $\tikz\draw[black,fill=white] (0,-1) circle (0.4em);$ vertices are in the same component with the remaining components being compositions of dashed lines. Assume the component with  $\tikz\draw[black,fill=white] (0,-1) circle (0.4em);$ has $a$ inputs and $b$ outputs and has genus $k$. We will call $(a,b,k)$ the signature of the glued graph. Then we must have $n-a=m-b$ is the number of components from dashed lines and $a+k-1=N$. We use $p_2^{a,b,k}$ to denote the map from $S^aV$ to $S^bV$ defined by two  $\tikz\draw[black,fill=white] (0,-1) circle (0.4em);$ components such that the glued graph with signature $(a,b,k)$. Or equivalently, $p_2^{a,b,k}$ is defined by all possible two-level breakings of a graph with signature $(a,b,k)$. We formally define $p^{a,b,k}=0$ if $b<0$, then the vanishing of the $\hbar^{N}S^mV$ part of $\widehat{p}_S\circ \widehat{p}_S(v_1 \ldots  v_n)$ implies  
	$$\sum_{a=1}^{n} p_2^{a,m-n+a,N-a+1} \ast ^{n-a}\Id=0,$$
	for all $n\ge 1,m,N\ge 0$. By setting $n=1$, we have $p_2^{1,m,N}=0$ for $m,N\ge 0$. Then by setting $n=2$ and $p_2^{1,m,N}=0$, we have $p_2^{2,m,N}=0$. Similarly we have $p_2^{n,m,N}=0$ for all $n\ge 1, m,N\ge 0$. This exactly describes that all maps from two-level breakings of a connected graph with signature $(n,m,N)$ should sum up to zero.  It implies that $\widehat{p}_E^2=0$ on $EV[[\hbar]]$ by the same argument in \cite[Proposition 2.9]{MZ22}.
\end{proof}

\begin{remark}
	From the proof of \Cref{prop:IBL_equi}, both  \Cref{def:curved_IBL,def:curved_IBL'} are equivalent to $p_2^{n,m,N}=0$. In the SFT world, this relation corresponds to the algebraic count of all rigid codimension $1$ breaking of connected holomorphic curves, with $n$ positive punctures, $m$ negative punctures, and genus $N$, is zero. 
\end{remark}
In the following, using \Cref{prop:IBL_equi}, we will switch our perspective of $IBL_\infty$ algebras to \Cref{def:curved_IBL'}, which has a more direct graph description and where orders of $\hbar$ have a more straightforward interpretation. The algebraic package, including morphisms etc., using \Cref{def:IBL}, can be found in \cite[\S 2]{cieliebak2015homological}.

Note that $SV\subset SV[[h]]$ induces an inclusion $EV\subset EV[[h]]$, and we use $\pi_0$ to denote the natural projection $EV[[\hbar]]\to EV$. It is easy to check that if $x\in \ker \pi_0$, then $\widehat{p}(x)\in \ker \pi_0$. As a consequence, $\pi_0\circ \widehat{p}|_{EV}:EV\to EV$ squares to zero. Moreover, $\pi_0\circ \widehat{p}|_{EV}$ is assembled from $p^{k,l,0}$. Then we have the following instant corollary.
\begin{corollary}
	Let $(V,p^{k,l,g})$ be a (cocurved) $IBL_\infty$ algebra, then $(V,p^{k,l}:=p^{k,l,0})$ is a $BL_\infty$ algebra.
\end{corollary}

\subsubsection{A grid of torsions} Given a family of operators $\phi^{k,l,g}:S^kV \to S^lV'$ for $k\ge 1,l,g\ge 0$, assume that for any $g\ge 0$ and $v_1,\ldots,v_k\in V$, there are at most finitely many $l$ such that $\phi^{k,l,g}(v_1 \ldots  v_k)\ne 0$. Then we can assemble  $\widehat{\phi}:EV[[\hbar]]\to EV'[[\hbar]]$ by the same rule for $BL_\infty$ morphisms except that cycles are allowed to be created and the rule for the order of $\hbar$ is the same as for $\widehat{p}$. 

\begin{figure}[H]
	\begin{center}
		\begin{tikzpicture}
		\node at (0,0) [circle,fill,inner sep=1.5pt] {};
		\node at (1,0) [circle,fill,inner sep=1.5pt] {};
		\node at (2,0) [circle,fill,inner sep=1.5pt] {};
		\node at (3,0) [circle,fill,inner sep=1.5pt] {};
		\node at (4,0) [circle,fill,inner sep=1.5pt] {};
		\node at (5,0) [circle,fill,inner sep=1.5pt] {};
		\node at (6,0) [circle,fill,inner sep=1.5pt] {};
		\node at (7,0) [circle,fill,inner sep=1.5pt] {};
		
		\draw (0,0) to (1,1) to (1,0);
		\draw (1,1) to (2,0);
		\draw (3,0) to (4,1) to (4,0);
		\draw (4,1) to (5,0);
		\draw (6,0) to (6.5,1) to (7,0);

		\draw (2,0) to (2.5,-1) to (3,0);
		\draw (2,-2) to (2.5,-1) to (2.5,-2);
		\draw (2.5,-1) to (3,-2);
		\node at (2.5,-1) [circle, fill, draw, outer sep=0pt, inner sep=3 pt] {};
		
		\node at (0,-2) [circle,fill,inner sep=1.5pt] {};
		\node at (2,-2) [circle,fill,inner sep=1.5pt] {};
		\node at (3,-2) [circle,fill,inner sep=1.5pt] {};
		\node at (2.5,-2) [circle,fill,inner sep=1.5pt] {};
	
		\draw (0,0) to (0,-2);
		\draw (5,0) to (5.5,-1);
		\draw (5.5,-1) to (6,0);
		\draw (7,0) to (7,-1);
		
		\node at (0,-1) [circle, fill, draw, outer sep=0pt, inner sep=3 pt] {};
		\node at (5.5,-1) [circle, fill, draw, outer sep=0pt, inner sep=3 pt] {};
		\node at (7,-1) [circle, fill, draw, outer sep=0pt, inner sep=3 pt] {};

        \draw (1,0) to (2.5,-1) to (4,0);

        \node at (1,1.1) {$g_1$};
        \node at (4,1.1) {$g_2$};
        \node at (6.5,1.1) {$g_3$};
        \node at (0.8, -1) {$\phi^{1,1,g_4}$};
        \node at (3.2, -1) {$\phi^{4,3,g_5}$};
        \node at (5.7, -1.4) {$\phi^{2,0,g_6}$};
        \node at (7.2, -1.4) {$\phi^{1,0,g_7}$};
		\end{tikzpicture}
	\end{center}
    \caption{A component of $\widehat{\phi}$ with input from $\hbar^{g_1}S^3V\odot \hbar^{g_2}S^3V \odot \hbar^{g_3}S^2V$ and output in $\hbar^{2+\sum g_i}S^4V'$}
\end{figure}

\begin{definition}
	The family of operators is an $IBL_\infty$ morphism from $(V,p^{k,l,g})$ to $(V',q^{k,l,g})$ iff $\widehat{q}\circ \widehat{\phi}=\widehat{\phi}\circ \widehat{p}$ and $|\widehat{\phi}|=0$.
\end{definition}

Then we have a trivial $IBL_\infty$ algebra $\mathbf{0}:=\{0\}$ with $p^{k,l,g}=0$. $\mathbf{0}$ is an initial object in the category of $IBL_\infty$ algebras with $\phi^{k,l,g}=0$. Then an $IBL_\infty$ augmentation of $V$ is an $IBL_\infty$ morphism from $V$ to $\mathbf{0}$. $IBL_\infty$ augmentations may not always exist. One obstruction is torsion. Unlike the torsions for $BL_\infty$ algebras, there are many more torsions for $IBL_\infty$ algebras. Let 
\begin{equation}
E^{k}V[[\hbar]]:=\overline{B}^kSV[[\hbar]].    
\end{equation}

\begin{definition}
	For $n,m\ge 0$, we say $V$ has an $(n,m)$ torsion if $[\hbar^{n}]=0\in H_*(E^{m+1}V[[\hbar]])$.	
\end{definition}
Then the algebraic torsion in \cite{LW} is the $(n,0)$ torsion of the $IBL_\infty$ algebra associated to a contact manifold by SFT. 

\begin{proposition}
    If $V$ has an $(n,m)$ torsion, then $V$ has no $IBL_\infty$ augmentation.
\end{proposition}
\begin{proof}
    If $V$ has an $IBL_\infty$ augmentation $\epsilon$, then we have chain maps $E^{m+1}\mathbf{0}[[\hbar]]\to E^{m+1}V[[\hbar]]\stackrel{\widehat{\epsilon}}{\to}E^{m+1}\mathbf{0}[[\hbar]]$, whose composition is the identity map. Hence $\widehat{\epsilon}$ sends $[\hbar^n]\in H_*(E^{m+1}V[[\hbar]])$ to $[\hbar^n]\ne 0\in H_*(E^{m+1}\mathbf{0}[[\hbar]])$, which contradicts the torsion assumption. 
\end{proof}

\begin{proposition}\label{prop:relation}
	If $V$ has an $(n,m)$ torsion, then $V$ has an $(n+m,0)$ torsion.
\end{proposition}
\begin{proof}
	We use $\hbar^{-k}SV[[\hbar]]$ to denote $SV[[\hbar]]\otimes_{\bk[[\hbar]]} \hbar^{-k}\bk[[\hbar]]$, where $\hbar^{-k}\bk[[\hbar]]$ is the $\bk[[\hbar]]$ module generated by $\hbar^{-k}$. Then we have a map $c_k:S^k SV[[\hbar]]\to \hbar^{-k+1}SV[[\hbar]]$ for $k\ge 2$ defined by
	$$w_1\odot \ldots \odot w_k \mapsto \hbar^{-k+1}w_1\ldots w_k.$$
	Pictorially, $c_k$ is simply connecting all trees compensated with $\hbar^{-k+1}$. Let $\iota$ be the inclusion $SV[[\hbar]]\to \hbar^{-m+1}SV[[\hbar]]$. We claim that
	$$C_{m}:=\iota+\sum_{i=2}^m c_i$$
	defines a chain map from $E^mV[[\hbar]]$ to $\hbar^{-m+1}SV[[\hbar]]$ for $m\ge 2$. 

    To prove the claim, we will show that $C_m\circ \widehat{p} = \widehat{p}\circ C_m$ on $S^iSV[[\hbar]]$ for each $1\le i \le m$. Using the graph description, $\widehat{p}$ on $S^iSV[[\hbar]]$ is decomposed with respect to the subset $I$ of $\{1,\ldots,i\}$ such that the nontrivial component of $\widehat{p}$ is glued to clusters indexed by $I$. Let $\widehat{p}_j$ denote the component such that the subset has $j$ elements, then $\widehat{p}|_{S^iSV[[\hbar]]}=\sum_{j=1}^i \widehat{p}_j$.  The graph description of $C_m\circ \widehat{p}_j$ is simply connecting all components of $\widehat{p}_j$ multiplied by $\hbar^{j-i}$. On the other hand, in the graph description of $\widehat{p}\circ C_m|_{S^iSV[[\hbar]]}$, we have the same graph. The order of $\hbar$ (in addition to the genus associated to the graph of $\widehat{p}_j$) is $\hbar^{-i+1}$ from $C_m$ on $S^iSV[[\hbar]]$ multiplied by $\hbar^{j-1}$ from the extra genus when we glue $\widehat{p}_j$ to the one combined tree compared to the original $i$ trees. In particular, we have $C_m\circ \widehat{p} = \widehat{p}\circ C_m$ on $S^iSV[[\hbar]]$. Hence the claim follows.
    
    Since $V$ admits an $(n,m)$ torsion, we have that $\widehat{p}(x)=\hbar^n$ for some $x\in E^{m+1}V[[\hbar]]$. Since $C_{m+1}|_{SV[[\hbar]]}=\iota$, in particular, we have $C_{m+1}(\hbar^{n})=\hbar^n$, we know that $\widehat{p}\circ C_{m+1}(x)=\hbar^n$. Then $\widehat{p}(\hbar^m C_{m+1}(x))=\hbar^{n+m}$ and $\hbar^m C_{m+1}(x)\in SV[[\hbar]]$, i.e.\ $V$ carries an $(n+m,0)$-torsion.
\end{proof}

In contrast with \Cref{prop:relation}, the torsion for the associated $BL_\infty$ algebra is not necessarily an $(0,m)$ torsion unless $m=0$ \cite{bourgeois2010towards}\footnote{\cite[Theorem 1]{bourgeois2010towards} showed that the vanishing of contact homology is equivalent to the $0$-algebraic torsion in \cite{LW}, i.e.\ $(0,0)$-torsion here.}. We use $p_0$ to denote the genus $0$ part of $p$, i.e.\ the induced $BL_\infty$ algebra on $V$. Indeed if $1=\widehat{p_0}(x)$ for $x\in E^{m+1}V$, then we only have $\widehat{p}(x)=1+O(\hbar)$, where $O(\hbar)\in \ker \pi_0$ with $O(\hbar)$ potentially coming from holomorphic curves with higher genera. On the other hand, as we will see later, an $(0,m)$ torsion always implies an $m$ torsion for the $BL_\infty$ algebra. To explain the relations, we will explain the following procedure, which increases genus one at a time.

Let $(V,p^{k,l,g})$ be an $IBL_{\infty}$ algebra. We define 
$$EV[[\hbar]]_m:= EV[[\hbar]]\otimes_{\bk[[\hbar]]}(\bk[[\hbar]]/( \hbar^{m+1})),$$ 
where $(\hbar^{m+1})$ is the ideal of $\bk[[\hbar]]$ generated by $\hbar^{m+1}$. Then we consider the projection $\pi_m:EV[[\hbar]] \to EV[[\hbar]]_m$. Then we have a filtration $$EV[[\hbar]]:=\ker \pi_{-1}\supset\ker \pi_0\supset \ker \pi_1\supset \ldots$$ We define the $\hbar$-width $w_{\hbar}(v)\in \N\cup \{\infty\}$ for $v\in EV[[\hbar]]$ to be 
$$w_{\hbar}(v):=\sup\left\{ m|v\in \ker \pi_{m} \right\}+1.$$
We can view $\bk[[\hbar]]/(\hbar^{m+1})$ as polynomials of $\hbar$ of degree at most $m$; in particular, $EV[[\hbar]]_m$ can be viewed as a subspace of $EV[[\hbar]]$. Then for $x\in EV[[\hbar]]$, we have $x=\pi_m x+ x_1$ for $x_1\in \ker \pi_m$. It is straightforward to check that $\pi_m\circ \widehat{p}(x_1)=0$. As a consequence, we have $\pi_m\circ \widehat{p}\circ \pi_m=\pi_m \circ \widehat{p}$. Then we have $w_{\hbar}(\widehat{p}(v))\ge w_{\hbar}(v)$ for all $v\in EV[[\hbar]]$. As a consequence, we have the following commutative diagram:
\begin{equation}\label{eqn:pi_m}
\xymatrix{
	EV[[\hbar]]\ar[rrr]^{\widehat{p}}\ar[d]^{\pi_m} & & & EV[[\hbar]]\ar[d]^{\pi_m}\\
	EV[[\hbar]]_m \ar[rrr]^{\widehat{p}_m:=\pi_m\circ \widehat{p}|_{EV[[\hbar]]_m}} & & & EV[[\hbar]]_m }
\end{equation}
And we have $\widehat{p}_m\circ \widehat{p}_m=0$. If we unwrap the definition, then we know that $\widehat{p}_m$ uses $p^{k,l,g}$ for $g\le m$. Take $m=1$ as an example: $\widehat{p}_1$ uses both $p^{k,l,1}$ and $p^{k,l,0}$. However, to get an output in $EV[[\hbar]]_1$, the $k$ inputs of $p^{k,l,1}$ must glue to $k$ trees while at most $2$ of the $k$ inputs of $p^{k,l,0}$ can glue to the same tree. We can similarly define $E^{k}V[[\hbar]]_m$, and we have analogous descriptions.

\begin{definition}
	We say $(V,p^{k,l,g})$ has an $(n,m)_k$ torsion iff $[\hbar^n]=0\in H^*(E^{m+1}V[[\hbar]]_k)$.
\end{definition}
Then by definition, we always have an $(n,m)_k$ torsion for $n>k$. Moreover, an $(0,m)_0$ torsion is the $m$-torsion for $BL_\infty$ algebras. The $(n,m)$ torsion can be viewed as an $(n,m)_\infty$ torsion. We summarize the basic properties of these torsions in the following.

\begin{proposition}\label{prop:at}
	The torsions have the following properties.
	\begin{enumerate}
		\item\label{1} If $V$ has an $(n,m)_k$ torsion then $V$ has an $(n,m)_{k-1}$, $(n+1,m)_{k}$ and $(n,m+1)_{k}$ torsions. In particular if $V$ has an $(n,m)$ torsion, then $V$ has an $(n,m)_{k}$ torsion for any $k\ge 0$.
		\item\label{2} If $V$ has an $(n,m)_k$ torsion, then $V$ has an $(n+m,0)_k$ torsion. 
	\end{enumerate}
\end{proposition}
\begin{proof}
 For \eqref{1}, that $V$ carries an $(n,m)_{k-1}$ torsion follows from \Cref{eqn:pi_m}. The existence of the other two torsions follows from the definition. \eqref{2} follows from the $EV[[\hbar]]_k$ version of Proposition \ref{prop:relation}.
\end{proof}
As a corollary, the existence of an $(0,k)$ torsion implies both $k$-algebraic torsion and $k$-algebraic planar torsion. Contact manifolds in \cite[Theorem 3.16 and 3.17]{MZ22} actually have an $(0,k)$ torsion by the same argument as in \cite{LW}. Therefore, Proposition \ref{prop:at} implies that both algebraic torsion and algebraic planar torsion are finite and bounded above by $k$. Roughly speaking, different torsions measure the counting of holomorphic curves whose domains have different topological types. For an $(k,0)$ torsion, we can have higher genus curves without negative punctures contributing to the torsion, while counts of all higher genus curves with the same positive asymptotics and non-empty negative punctures must sum up to zero. For an $(0,k)$ torsion, we can only have rational curves without negative punctures contributing to the torsion, while all higher genus curves with the same positive asymptotics and non-empty negative punctures must sum up to zero. For an $(0,k)_0$ torsion, we can only have rational curves without negative punctures contributing to the torsion, and all rational curves with the same positive asymptotics and non-empty negative punctures must sum up to zero.

Let $2^{\N^3}$ denote the category of subsets of $\N^2\times(\N\cup \{\infty\})$, where the arrow from $V$ to $W$ is an inclusion $W\subset V$. $2^{\N^3}$ is a monoidal category where the monoidal structure is given by taking the union. We use $\SFT(Y)$ to denote the full SFT as an $IBL_\infty$ algebra in the sense of Definition \ref{def:curved_IBL'} for a contact manifold $Y$.
\begin{theorem}\label{thm:torsion}	
	Let $Y$ be a contact manifold, we define $\mathrm{T}(V):=\{(m,n,k)|\SFT(Y) \text{ has an } (m,n)_k \text{ torsion}\}$. Then $\mathrm{T}:\cont \to 2^{\N^3}$ is a covariant monoidal functor.
\end{theorem}
\begin{remark}\label{rmk:glue}
	For the proof of Theorem \ref{thm:torsion}, we only need to construct $\SFT(Y)$ to the same extent as \cite[Theorem 3.10]{MZ22}. 
    That is, upon fixing some extra choices, we can define the virtual counts of moduli spaces and define $IBL_\infty$ algebras of a contact manifold and $IBL_\infty$ morphisms of an exact cobordism. But there is no need to establish the homotopy invariance of those structures with respect to different choices made in the construction.
\end{remark}
\begin{remark}[Analog of planarity]
    Given an $IBL_\infty$ augmentation $\epsilon$, then we have similar constructions of linearization to construct another $IBL_\infty$ structure $p^{k,l,g}_{\epsilon}$ such that $p^{k,l,g}_{\epsilon}=0$ whenever $l=0$. As a consequence, we arrive at an $IBL_\infty$ structure in the sense of Definition \ref{def:IBL}. Then we can introduce the analog of pointed morphisms and the analog of orders in the context of $IBL_\infty$ algebras, which in the SFT case considers holomorphic curves passing through a fixed point in the symplectization. Moreover, we will have a grid of orders as in the case of torsions above. However, it is a much harder task to find examples with holomorphic curves of higher genera. 
\end{remark}

\bibliographystyle{plain} 
\bibliography{ref}
\Addresses
\end{document}